\documentclass[11pt,a4paper]{amsart}

\usepackage{amsmath,amsfonts}
\usepackage{latexsym}
\usepackage{amssymb}
\usepackage{color}
\usepackage{epsfig}
\usepackage{subfigure}
\usepackage{mathrsfs}
\usepackage{amscd}
\usepackage{amsthm}
\usepackage{color}
\usepackage{cite}

\usepackage{tikz}
\usepackage{amsmath,amsfonts}
\usepackage{latexsym}
\usepackage{amssymb}
\usepackage{color}
\usepackage{epsfig}
\usepackage{subfigure}
\usepackage{mathrsfs}
\usepackage{amscd}
\usepackage{amsthm}
\usepackage{color}
\usepackage{cite}

\usepackage{graphicx}
\usepackage{float}

\numberwithin{equation}{section}
\newtheorem{proposition}{Proposition}[section]
\newtheorem{definition}{Definition}[section]
\newtheorem{lemma}{Lemma}[section]
\newtheorem{theorem}{Theorem}[section]
\newtheorem{corollary}{Corollary}[section]
\newtheorem{remark}{Remark}[section]

\usepackage[bottom]{footmisc}
\usepackage[colorlinks=true
    ,linkcolor=blue%
    ,citecolor=blue%
    ,urlcolor=blue%
    ,backref=page]{hyperref}
\hypersetup{urlcolor=blue, citecolor=blue}
\usepackage{amssymb}
\usepackage{mathrsfs}
\usepackage{amscd}
\usepackage{cite}

\newcommand{\R}{\mathbb{R}}
\newcommand{\C}{\mathbb{C}}

\newcommand{\z}{{\mathrm{z}}}
\newcommand{\rr}{{\mathrm{r}}}

\newcommand{\Z}{{\mathbb{Z}}}
\newcommand{\LL}{{\mathcal{L}}}
\newcommand{\CC}{{\mathcal{C}}}

\begin{document}
\title[Strichartz estimates on metric cones]
{Global-in-time Strichartz estimates and cubic Schr\"odinger equation in a conical singular space}

\author{Junyong Zhang}
\address{Department of Mathematics, Beijing Institute of Technology, Beijing
100081, China}
\email{zhang\_junyong@bit.edu.cn}

\author{Jiqiang Zheng}
\address{Institute of Applied Physics and Computational Mathematics, Beijing 100088}
\email{zheng\_jiqiang@iapcm.ac.cn; zhengjiqiang@gmail.com}

\maketitle

\begin{abstract}
In this paper, we study Strichartz estimates
for the Schr\"odinger equation on a metric cone
$X$, where $X=C(Y)=(0,\infty)_r\times Y$ and the cross section $Y$ is
a $(n-1)$-dimensional closed Riemannian manifold $(Y,h)$. For the metric $g$ on $X$ given by $g=dr^2+r^2h$, let $\Delta_g$ be
the positive  Friedrichs extension Laplacian on $X$ and $V=V_0 r^{-2}$ where $V_0\in\CC^\infty(Y)$ is a real function such that the operator $P:=\Delta_h+V_0+(n-2)^2/4$ is a strictly positive operator on $L^2(Y)$.
We establish the full range of global-in-time Strichartz estimates without loss for the Schr\"odinger equation associated with the operator $\LL_V=\Delta_g+V_0 r^{-2}$
including the endpoint estimate both in homogeneous and inhomogeneous cases. A new finding reveals that the range of admissible pairs at $\dot H^s$-level is influenced by the smallest eigenvalue of the operator $P$. This additionally proves  the conjecture in Wang [Ann. Inst. Fourier 2006]
and generalizes the results of Ford [Comm. Math. Phys. 2010] and Baskin-Marzuola-Wunsch [Contemp. Math. 2014]. As an application, we show the well-posedness theory and scattering theory for the Schr\"odinger equation with a cubic nonlinearity on
this setting which verifies  a conjecture in Baskin-Marzuola-Wunsch [Contemp. Math. 2014].

\end{abstract}

\begin{center}
 \begin{minipage}{120mm}
 \end{minipage}
 \end{center}

  \tableofcontents

\section{Introduction and Statement of Main Result}

The purposes of this paper are: study the Strichartz estimates for the
Schr\"odinger equations on the setting of metric cone, and establish the well-poseness theory and the scattering theory for the Schr\"odinger equation with a cubic nonlinearity.

\subsection{The setting model }We consider a conic singular space setting. The space $(X,g)$ is given by $X=C(Y)=(0,\infty)_r\times Y$ and $g=dr^2+r^2h$
where $(Y,h)$ is a $(n-1)$-dimensional closed Riemannian manifold, (i.e. $\partial Y=\emptyset$).
The cone $X$ has the simplest geometry singularity and it has an incomplete metric.
One can complete it to $C^*(Y)=C(Y)\cup P$ where $P$ is its cone tip.
Let $\Delta_g$
denote the Friedrichs extension of Laplace-Beltrami operator from the domain
$\CC_c^\infty(X^\circ)$ that consist of the compactly supported smooth functions on the
interior of the metric cone.
There are several works on extending the theory of the Laplace operator $\Delta_g$ on smooth manifolds to 
such conical singular spaces;  see Cheeger \cite{C1,C2} for instance.\vspace{0.1cm}

Consider the Schr\"odinger operator $\mathcal{L}_V=\Delta_g+V$ where $V=V_0(y) r^{-2}$ and $V_0(y)$ is a smooth function on the section $Y$
such that the operator $\Delta_h+V_0+(n-2)^2/4$ is a strictly positive operator on $L^2(Y)$ space
\footnote{The operator $\Delta_h$ is the positive Laplacian on $Y$ and the
number $(n-2)^2/4$ is from the sharp constant for the Hardy inequality which guarantees that the operator $\LL_V$ is strictly positive.}.
This operator has attracted researcher's interests from different disciplines such as geometry, analysis and physics. On the one hand, the operator without potential was first studied
from the wave diffraction from the cone point, see \cite{Som,Frie,Frie1}.
Cheeger and Taylor \cite{CT,CT1} studied the diffractive phenomenon of the wave associated with this operator from the functional calculus and the regularity properties of wave propagation \cite{MS}.
 The heat kernel and Riesz transform
kernel were studied in \cite{ L1,Mo}.
 On the other hand,  the decay of the potential considered here is closely related to the fact that the angular momentum as $r\to\infty$,
the inverse square decay of the potential is in some
sense critical for the spectral and scattering theory.  For the case with the inverse-square potential, the asymptotical behavior of Schr\"odinger propagator was considered in \cite{Carron,wang}
and Riesz transform was studied in \cite{HL}. The restriction estimate for Schr\"odinger solution was investigated by the first author \cite{zhang}.

\subsection{The Strichartz estimates}
In this paper, we study a spacetime-type estimate of the solution $u: \R\times X\rightarrow \C$ to the initial value
problem (IVP) for the Schr\"odinger equation on metric cone $X$,
\begin{equation}\label{equ:S}
i\partial_t u(t,z)+ \mathcal{L}_V u(t,z)=0,\quad u(t,z)|_{t=0}=u_0(z),\quad
(t,z)\in\R\times X.
\end{equation}
It is known that
the Strichartz estimates are powerful tools for studying the behaviour of solutions of the nonlinear dispersive equations, such as Schr\"odinger equation. More precisely, let $u$ be the solution to \eqref{equ:S}. The aim is to establish the inequality in the form of
\begin{equation*}
\|u(t,z)\|_{L^q_tL^\rr_z(I\times X)}\leq
C\|u_0\|_{H^s(X)},
\end{equation*}
where $I$ is a subset interval of $\R$ and $H^s$ denotes the $L^2$-Sobolev space over $X$, and
$(q, \rr)$ is an \emph{admissible pair}, that is
\begin{equation}\label{adm-p}
(q,\rr)\in\Lambda_0:=\big\{2\leq q,\rr\leq\infty, \quad 2/q+n/\rr=n/2,\quad (q,\rr,n)\neq(2,\infty,2)\big\}.
\end{equation}
In particular, when $I=\R$, we say that the Strichartz estimates are global-in-time. If $s=0$, we say that the Strichartz estimates have no loss of derivative.  For example, let consider the Euclidean space and $V=0$. In this case, one has global-in-time Strichartz estimates without loss of derivative. In general, as pointed out in \cite{MMT}, it is harder to obtain global-in-time Strichartz estimates than a local-in-time one.
There are many studies on the Strichartz inequalities on Euclidean space or manifolds in the literature,
we refer the reader to \cite{Str, GV, KT, ST} and the references therein.
However, as an illustration, we will mention some closely related studies on the Strichartz estimates for the Schr\"odinger equation
on some settings such as
asymptotically conic manifolds, exact cones, and cases with inverse-square potentials. We recall that an asymptotically
conic manifold $(M,g)$ is a complete non-compact Riemannian manifold of
dimension $n\geq3$ with one end diffeomorphic to $(0,\infty)\times
Y$ where $Y=\partial {\overline{M}}$ and $\overline{M}$ is a compactified manifold of $M$. If the metric $g$ in the end $[1,\infty]_r\times
Y$ can be written as
$$g=dr^2+r^2 \tilde{h}(r,y)$$
where $\tilde{h}\in \mathcal{C}^\infty(Sym^2(T^*{\overline{M}}))$ is a smooth family of metrics on $\partial {\overline{M}}$,  then the manifold $(M,g)$ is called an asymptotically conic manifold.
We say $M$ is non-trapping, if every geodesic $z(s)$ in $\overline{M}$ reaches the boundary
 $\partial {\overline{M}}$ as $s\to\pm\infty$.
 On the non-trapping asymptotically
conic manifold $M$, the local-in-time Strichartz estimates were established in Hassell, Tao and Wunsch \cite{HTW, HTW1} and Mizutani \cite{Miz}.
Recently, Hassell and the first author \cite{HZ} improved  the results by showing the global-in-time Strichartz inequality and fixing the endpoint estimate.
Burq, Guillarmou and Hassell \cite{BGH} proved  local-in-time Strichartz estimates without loss on the asymptotically conic manifold with a trapped set  which is hyperbolic and of sufficiently small fractal dimension. Most recently, Bouclet and Mizutani \cite{BM} and the authors \cite{ZZ1} showed the same result but global-in-time under the same condition of \cite{BGH} for the asymptotically conic manifold.
Note that the metric cone (considered in this paper) is closed to the end of the asymptotically conic manifold,
thus the establishment of the Strichartz estimates on each manifold is closely related and but with innovative aspects to overcome difficulties (e.g. diffraction) arising from the
setting and the scaling critical potential. \vspace{0.1cm}

There are several other related studies on the Strichartz estimates on cones in the literature. In \cite{Ford}, Ford proved the full
range of global-in-time Strichartz estimates for Schr\"odinger on the flat cone $C(\mathbb{S}^1_\rho)$. Later the Strichartz estimates for wave on $C(\mathbb{S}^1_\rho)$ were proved by  Blair, Ford and Marzuola in \cite{BFM}.
The Strichartz estimates were further proved for polygonal domains by Blair, Ford, Herr and Marzuola\cite{BFHM} and for exterior polygonal domains by Baskin, Marzuola and Wunsch \cite{BMW}.
Note that as it is stated as in \cite{BFM}, one needs the explicit form of propagators when $Y=\mathbb{S}^1_\rho$ in the discussed references. Another challenge is the potential presence of conjugate points within our general cone settings. Therefore, the methods utilized there is not applicable in our general settings.

In addition, as mentioned above, the perturbation of the inverse-square potential is non-trivial since
the inverse-square decay of the potential has the same scaling to the Laplacian operator.
In \cite{BPST, BPSS},   Burq, Planchon, Stalker, and Tahvildar-Zadeh proved Strichartz estimates for the Schr\"odinger and wave on Euclidean space with
inverse-square potentials. In particular, Wang \cite[Remark 2.4]{wang} conjectured the Strichartz estimates for $\LL_V$ which are on the metric cone and with an inverse-square potential.
The main purpose of this paper is to prove the conjecture in \cite{wang}, that is, to prove
the full range of global-in-time Strichartz estimates for Schr\"odinger equation associated with the operator $\mathcal{L}_V$.
As an application, then we will show the global well-poseness  and scattering theory for the cubic Schr\"odinger equation with small initial data on
this setting which verifies  a conjecture in \cite[Conjecture 1]{BMW}.
\vspace{0.2cm}

\subsection{The main results on Strichartz estimates}Now we state
the following main results.

\begin{theorem}[Global-in-time Strichartz estimates]\label{thm:Strichartz} Suppose that $(X, g)$ is a metric cone of dimension
$n\geq3$. Let $\LL_V=\Delta_g+V$ where $r^2V:=V_0\in\CC^\infty(Y)$ such that $\Delta_h+V_0(y)+(n-2)^2/4$ is a strictly positive operator on $L^2(Y)$. Then
the homogenous Strichartz estimates
\begin{equation}\label{Str-est}
\|e^{it\LL_V}u_0\|_{L^q_tL^\rr_z(\mathbb{R}\times X)}\leq
C\|u_0\|_{L^2(X)},
\end{equation}
holds for the admissible pair $(q,\rr)\in [2,\infty]^2$ satisfies
\eqref{adm-p}. Moreover, the inhomogeneous inequality
\begin{equation}\label{eq:inhom}
\Big\|\int_0^t e^{i(t-s)\LL_V}F(s) ds\Big\|_{L^q_tL^\rr_z(\mathbb{R}\times X)}\leq C \| F
\|_{L^{\tilde q'}_tL^{\tilde \rr'}_z(\mathbb{R}\times X)}
\end{equation}
holds for admissible pairs $(q,\rr)$, $(\tilde q, \tilde \rr)$ including the endpoint $q=\tilde{q}=2$.
\end{theorem}

\begin{remark} The result shows
the full range of global-in-time Strichartz estimates without loss of derivative including the endpoint estimates for both homogeneous and inhomogeneous cases.
The result verifies the conjecture in \cite{wang} and generalizes the results in \cite{BMW,BFHM}.
\end{remark}

\begin{remark} This result is the same to the Strichartz estimates for Schr\"odinger without potential in the Euclidean space.
But it does not imply that the influences of the geometry (e.g. conjugated points and diffraction) and the inverse-square potential are trivial. In particular, in Theorem \ref{thm:Strichartz'}, one will see that the range of admissible pairs $(q,\rr)$ is affected by the smallest eigenvalue of the operator $\Delta_h+V_0(y)+(n-2)^2/4$.
\end{remark}

\begin{remark} The physical electrical point-dipole potential $V=ar^{-2}y_3$ with the constant $a>-(n-2)^2/4$ and $y_3\in Y=\mathbb{S}^2$ satisfies the above assumptions.
An example of the Schr\"odinger equation with this potential is from the study of electron capture by polar molecules, see \cite{BPST,LL}.
\end{remark}

We sketch the idea and argument of the proof here.
We use the abstract method of  Keel-Tao's \cite{KT} but one needs to prove dispersive estimate and $L^2$-estimate.
In our setting, however, there are two obstacles to prevent us from obtaining the dispersive estimate with decay rate $O(t^{-n/2})$. The first one is due to
the perturbation of the inverse-square type potential. For example, even on the Euclidean space with the inverse-square potential, the
usual dispersive estimate fails, see \cite{FFFP}.
On the asymptotically conic manifold, as stated in \cite[Remark 3.7]{HZ}, if $V_0 (y)$ takes values in the range $(-(n-2)^2/4, 0)$, then it
follows from \cite[Corollary 1.5]{GHS1} that the $L^1 \to L^\infty$ norm of the propagator is at least a constant times $t^{-(\nu_0 + 1)}$ as $t \to \infty$, where $\nu_0^2$ is the smallest eigenvalue of $\Delta_{h} + V_0 + (n-2)^2/4$. Under the above assumptions on the range of $V_0$,
we see that $\nu_0 <(n-2)/2$. This implies that the dispersive estimate (1-12) in \cite{HZ} will no longer be valid  as $|t-s| \to \infty$.
Hence we can not directly use Keel-Tao's abstract method to prove Theorem \ref{thm:Strichartz} as it is done in \cite{HZ}.\vspace{0.1cm}

Instead, we first prove the Strichartz estimates for $\LL_0$ (without potential) and then obtain
Theorem \ref{thm:Strichartz} by using a perturbation argument. To prove the Strichartz estimates for $\LL_0$,
we will encounter the second challenge caused by
the possible presence of conjugated points. This one is due to the non-trivial perturbation from the metric $h$.
To fix this, we have to microlocalize the propagator to separate the
conjugated points as it is done by Hassell and the first author in\cite{HZ}, since the usual dispersive estimate fails due to the conjugated points, see \cite{HW}. Even though
we follow our previous idea in \cite{HZ}, we have to reconstruct the microlocalized operator due to the differences between the two settings.
For instance the region near the cone tip is different in these settings.
The proof of the Strichartz estimates for $\LL_0$ combines the
methods of \cite{HZ} (where we developed a micro-localized spectral measure to capturing
the decay and oscillation of the Schr\"odinger propagator) and of \cite{zhang} (in which we employed the Cheeger-Taylor's method \cite{CT} to write the propagator as a linear combination of the Hankel transform of the radial part and eigenfunctions). The double endpoint inhomogeneous Strichartz estimate is a bit more complicate than
usual Euclidean case because of the lack of dispersive estimate.\vspace{0.1cm}

After proving the Strichartz estimates for $\LL_0$,  we perform the
perturbation argument \cite{JSS,RS, BPST} to obtain Strichartz estimates for $\LL_V$ through a global-in-time local smoothing estimate.
The key point is to establish a global-in-time local smoothing estimate for $\LL_V$ which is more interesting on its own right.
The general method to obtain the local smoothing  is via establishing the resolvent estimate and Kato's method.
The local smoothing is directly proved by using the formulas with separating variables expression and avoiding the usual resolvent method.
Hence we use the global-in-time local smoothing estimate for $\LL_V$ to obtain the homogeneous Strichartz estimates.\vspace{0.1cm}

The double endpoint inhomogeneous Strichartz estimate is proved by an iterated argument and weighted resolvent estimates.
The new ingredient is to prove the weighted resolvent estimates.
The weighted resolvent estimate implies that the weight $r^{-1}$ is $\LL_V$-supersmooth and hence it can be applied to obtain the endpoint inhomogeneous Strichartz estimates.
Furthermore, as an application of the resolvent estimate, we show a uniform Sobolev inequality for independent interest.
\vspace{0.2cm}

It worths to make some remarks on the differences between the setting models considered in \cite{HZ, BMW, Ford} and the ones we discuss here. The main difference
between the setting in \cite{HZ} and  our setting here lies in
the cone tip and the singular potential. The metric cone is a bit simpler than
the end of asymptotically conic manifold since the metric $h$ is independent from $r$ (while $\tilde{h}$ not), hence we can use a scaling argument.
Another natural property of the metric cone $X$ is automatically non-trapping. Indeed,
let $r(t)$ be the $r$ coordinate of the geodesic at time $t$ and let $\mu$ be the angular momentum.
The symbol of the Laplacian with respect to $g=dr^2+r^2 h$ is $\sigma(\Delta_g)=\tau^2+r^{-2}h^{-1}_{ij}(y)\mu_i\mu_j$ and along geodesics we have
$$\frac{d^2 r(t)}{dt^2}=2\frac{d \tau}{dt}=4r^{-3}h^{-1}_{ij}(y)\mu_i\mu_j>0,$$ hence there is no trapping geodesic in $X$. However in \cite{HZ}, we only consider the
very short range potential $O(r^{-3})$ as $r\to\infty$ and we do not need to consider the region near the cone tip. Compared with the models in \cite{BMW, Ford}, the setting in our work is
such complicated that it possibly contains conjugated points and so we can not obtain the expression of the propagator as in \cite{BMW, Ford}.\vspace{0.2cm}

\subsection{The Strichartz estimates at $\dot H^s$-level} In the Euclidean space, from the Sobolev embedding and
the Strichartz estimates as in Theorem \ref{thm:Strichartz}, we can obtain more Strichartz estimates for initial data in $\dot H^s(\R^n):=(-\Delta_{\R^n})^{-\frac s2}L^2(\R^n)$.
For example, in $\R^n$ we have
\begin{equation*}
\|e^{it\Delta}u_0\|_{L^q_tL^\rr_z(\mathbb{R}\times \R^n)}\leq
C\|u_0\|_{\dot{H}^s(\R^n)}
\end{equation*}
where $s\geq0$ and
\begin{equation}\label{adm-p-s}
(q,\rr)\in\Lambda_s:=\big\{2\leq q,\rr\leq\infty, \quad 2/q+n/\rr=n/2-s,\quad (q,\rr,n)\neq(2,\infty,2)\big\}.
\end{equation}
Similarly we can prove more Strichartz estimates for initial data in $\dot H^s(X):=\LL_V^{-\frac s2}L^2(X)$  in our setting.
However, the result is different from the one in the Euclidean space which  emphasizes the influence of the inverse-square potential.
More precisely, we will prove the following theorem.

\begin{theorem}\label{thm:Strichartz'} Let $(X, g)$ and $\LL_V=\Delta_g+V$
be as in Theorem \ref{thm:Strichartz}. Let
$\nu_0$ be the square root of the smallest eigenvalue of the operator $\Delta_h+V_0(y)+(n-2)^2/4$ on $L^2(Y)$. Then
one has the Strichartz estimates
\begin{equation}\label{Str-est'}
\|e^{it\LL_V}u_0\|_{L^q_tL^\rr_z(\mathbb{R}\times X)}\leq
C\|u_0\|_{\dot{H}^s(X)},
\end{equation}
where $s\geq0$ and
\begin{equation}\label{adm-p-s'}
(q,\rr)\in \Lambda_{s,\nu_0}:=\Lambda_s\cap \{(q,\rr): 1/\rr>1/2-(1+\nu_0)/n\}.
\end{equation}
The restriction of $1/\rr>1/2-(1+\nu_0)/n$ is necessary in the sense that the Strichartz estimates \eqref{Str-est'} possibly fail if $(q,\rr)\in \Lambda_s$ but  $(q,\rr)\notin \{(q,\rr):\frac1\rr>\frac12-\frac{1+\nu_0}n\}$.
\end{theorem}

\begin{remark}
The set $\Lambda_{s,\nu_0}$ is nonempty only if $s\in [0,1+\nu_0)$. Furthermore
$\Lambda_{s,\nu_0}=\Lambda_s$ when $s\in [0,1/2+\nu_0)$ in which the requirement $1/\rr>1/2-(1+\nu_0)/n$
automatically disappears, and while $\Lambda_{s,\nu_0}\subset \Lambda_s$ when $s\in [1/2+\nu_0, 1+\nu_0)$.
\end{remark}

\begin{remark}
Compared with the Euclidean case, the Strichartz estimates hold in the region ABCE but fail in the region CDOE as it is shown in Figure 1.
The result implies that the smallest eigenvalue  of $\Delta_h+V_0(y)+(n-2)^2/4$
plays an important role in the Strichartz estimates.
\end{remark}

\begin{center}
 \begin{tikzpicture}[scale=0.8]
\draw[->] (0,0) -- (4,0) node[anchor=north] {$\frac{1}{q}$};
\draw[->] (0,0) -- (0,4)  node[anchor=east] {$\frac{1}{\rr}$};
\draw (0,0) node[anchor=north] {O}
(3,0) node[anchor=north] {$\frac12$};
\draw  (0, 3) node[anchor=east] {$\frac12$}
       (0, 0.6) node[anchor=east] {$\frac12-\frac{1+\nu_0}{n}$}
     ;

\draw[thick] (3,0) -- (3,1.8)  %实线命令
              (3,0.6) --  (0,1.8)
                    (3,1.8) -- (0,3) ;

\draw[dashed,thick] (3,0.6) -- (0,0.6);
\draw (-0.1,3.2) node[anchor=west] {A};
\draw (2.9,1.8) node[anchor=west] {B};
\draw (2.9,0.6) node[anchor=west] {C};
\draw (2.9,0.15) node[anchor=west] {D};
\draw (-0.1,0.4) node[anchor=west] {E};
\draw (-0.5,1.8) node[anchor=west] {F};

\draw (3.4,0.9) node[anchor=west] {$\frac{2}{q}+\frac{n}{\rr}=\frac{n}{2}-\nu_0$};
\draw (1.65,2.88) node[anchor=west] {$\frac2q+\frac{n}{\rr}=\frac{n}{2}$};

\draw (3.0,0.6) circle (0.06);

\draw[<-] (2.3,2.1) -- (2.6,2.6) node[anchor=south]{$~$};
%\draw[<-] (2,1.65) -- (2.5,2) node[anchor=north]{$~$};
\draw[<-] (2.4,0.9) -- (3.4,0.9) node[anchor=north]{$~$};

\path (2,-1.5) node(caption){Figure 1. Diagrammatic picture of the range of $(q,\rr)$ for $n\geq3$};  %题注

\end{tikzpicture}

\end{center}

\subsection{Application to NLS}As an application of the global-in-time Strichartz estimates, we study the cubic Schr\"odinger equation
\begin{equation}\label{equ:cubic}
\begin{cases}
i\partial_t u+\LL_V u+\gamma|u|^{2}u=0,\qquad
(t,z)\in\R\times X, \\
u(t,z)|_{t=0}=u_0(z),\qquad\qquad z\in X,
\end{cases}
\end{equation}
where $\gamma=\pm1$ which corresponds to the defocusing and focusing case respectively. Here we consider global existence and scattering for the cubic initial value problem.
In particular, for 2-dimensional metric cone,
we \cite{zhang} obtained the  global solution and scattering result for  the mass-critical \eqref{equ:cubic} with small $L^2$-norm radial data.
Due to the Strichartz estimates in Theorem \ref{thm:Strichartz}, one can remove the radial assumption to obtain the similar result for the high dimension mass-critical equation by following
the arguments of Cazenave-Weissler \cite{CW} or Tao \cite{Tbook}.
Here in our work we consider the well-posedness problem for \eqref{equ:cubic} in energy space which lead to a new difficulty caused by chain rule associated with our operator $\LL_V$.
Before stating  the second result, we need some notation.\vspace{0.2cm}

Let $(X,g)$ be the metric manifold above and let $dv=\sqrt{g}dz=r^{n-1}dr dh$ be the measure  induced by the metric $g$, we define the complex
Hilbert space $L^2(X)$ to be given by the inner product
\begin{equation*}
\langle f,
h\rangle_{L^2(X)}=\int_{X}f(z)\overline{h(z)} dv.
\end{equation*} We say that $f\in L^p(X)$ for $1<p<\infty$ if $\int_X |f|^p dv<\infty$.  For $1\leq p<\infty$ we denote the inhomogeneous Sobolev space over $X$ by $H^1_{p}(X)=(\mathrm{Id}+\LL_V)^{-\frac12} L^p(X)$ and
write $H^1(X)=H^1_{p}(X)$ with $p=2$. \vspace{0.2cm}

In this paper, as mentioned above, we are interested in the global existence and scattering for nonlinear equation \eqref{equ:cubic} with $u_0\in H^1(X)$ when the dimension is $3$. The initial value problem falls into
a class of energy-subcritical one on the metric cone. Solutions to \eqref{equ:cubic} preserve the energy,
\begin{equation}\label{equ:energy}
E(u)(t)=\int_X\big(\frac12|\sqrt{\LL_V}u(t,z)|^2+\frac\gamma4|u(t,z)|^4\big)\;dv
\end{equation}
along with the mass
\begin{equation}\label{equ:mass}
M(u)(t)=\int_X|u(t,z)|^2\;dv.
\end{equation}
\vspace{0.2cm}

Our another main result is about the well-posedness and nonlinear scattering of the cubic Schr\"odinger equation.

\begin{theorem}\label{thm:NLS} Let $X$ be
metric cone of dimension $3$,  $\LL_V=\Delta_g+V$ be the same as the one in Theorem \ref{thm:Strichartz}, $\gamma=\pm1$ and suppose that the initial data $u_0\in H^1(X)$.
Then
there exists $T=T(\|u_0\|_{H^1})>0$ such that the nonlinear Schr\"odinger equation
\eqref{equ:cubic} has a unique solution $u$ satisfying
\begin{equation}\label{small}
u\in C(I; H^1(X))\cap L_t^{q}(I; H^1_\rr(X)),\quad I=[0,T), (q,\rr)\in\Lambda_0.
\end{equation}
The solution for the defocusing case (i.e. $\gamma=1$) can be extended to a global one.
Moreover, assume $\|u_0\|_{H^1(X)}\leq \epsilon$ for a small constant $\epsilon$, then there exists a global solution $u$ and the solution $u$ scatters in the sense that
there are $u_\pm\in H^1(X)$ such that
$$\lim\limits_{t\to\pm \infty}\|u(t)-e^{it\LL_V}u_\pm\|_{H^1(X)}=0.$$
\end{theorem}

\begin{remark} The result provides the scattering theory of the cubic Schr\"odinger with small data which was conjectured in \cite[Conjecture 1]{BMW}.
One also can apply the Strichartz estimates to show the well-posedness and scattering result for Schr\"odinger with more general nonlinearity $|u|^p u$.
\end{remark}

This theorem is an analogue of the well known result for nonlinear Schr\"odinger on Euclidean space and the result claim that
the global well-posedness and scattering theory for cubic Schr\"odinger with small
data  hold on the metric cone manifold. Like the Euclidean result, the small initial data result is a cornerstone result for future large data result,
e.g. nonlinear scattering theory on this setting.
The key points of the proof are the global-in-time Strichartz estimates and Leibniz chain rule.
The global-in-time Strichartz estimatess have been done in Theorem \ref{thm:Strichartz}, we have to prove the chain rule for
the operator $\LL_V$ which is a bit different from the classical one
due to the perturbation of the inverse-square potential. In particular, for $\LL_V=-\Delta+ar^{-2}$ with $a>-(n-2)^2/4$ where $\Delta$ is the Laplacian
on Euclidean space,
Killip, Miao, Visan and the authors \cite{KMVZZ1} proved the fractional chain rule whose range of the index $p$ is restricted by the value of $a$. In our setting,
based on Hassell-Lin \cite{HL}, we prove a Hardy inequality to obtain the chain rule which is related to the smallest eigenvalue $\nu_0^2$ of
the operator $\Delta_h+V_0(y)+(n-2)^2/4$.
Hence we have to choose the admissible pairs adapted to the smallest eigenvalue $\nu_0^2$ in existence part of the proof. \vspace{0.2cm}

Now we introduce some notation. We use $A\lesssim B$ to denote
$A\leq CB$ for some large constant C which may vary from line to
line and depend on various parameters. Similarly we use $A\ll B$
to denote $A\leq C^{-1} B$. We employ $A\sim B$ when $A\lesssim
B\lesssim A$. If the constant $C$ depends on a special parameter
other than the above, we shall denote it explicitly by subscripts.
For instance, $C_\epsilon$ should be understood as a positive
constant not only depending on $p, q, n$, and $M$, but also on
$\epsilon$. Throughout this paper, pairs of conjugate indices are
written as $p, p'$, where $\frac{1}p+\frac1{p'}=1$ with $1\leq
p\leq\infty$. We denote $a_\pm$ to be any
quantity of the form $a\pm\epsilon$ for any small $\epsilon>0$. \vspace{0.2cm}

This paper is organized as follows: In Section 2, we use the Hankel
transform and Bessel function to give the expression of the
solution, and we establish the $L^p$-product chain rule. Section 3 is dedicated to considering the spectral measure associated with the operator $\mathcal{L}_0=\Delta_g$ on
the metric cone and hence we prove Strichartz estimates for the Schr\"odinger without potential.  In Section 4, we prove a local-smoothing estimate and then obtain
the homogeneous estimates in Theorem \ref{thm:Strichartz}. In Section 5, we prove a weighted resolvent estimate and then show the endpoint inhomogeneous Strichartz estimates. We prove the Theorem \ref{thm:Strichartz'} in Section 6. In the final section, we use the Strichartz
estimates and $L^p$-product chain rule  to show Theorem
\ref{thm:NLS}.\vspace{0.2cm}

{\bf Acknowledgments:}\quad  The authors would like to thank Andrew Hassell, Andras Vasy and Jared Wunsch for their
helpful discussions and encouragement. The first author is grateful for the hospitality of the Australian National University and
Stanford University, where the project was initiated and finished.
J. Zhang was supported by  National key R\&D program of China: 2022YFA1005700, National Natural Science Foundation of China(12171031) and Beijing Natural Science Foundation(1242011); J. Zheng was supported by National key R\&D program of China: 2021YFA1002500 and NSF grant of China (No. 12271051).\vspace{0.2cm}

\section{Preliminary: analysis results associated with $\LL_V$}
In this section, we study the operator $\LL_V$ over metric cone $X$. By following the separation of variable method \cite{CT},
we first introduce a orthogonal decomposition of
$L^2(Y)$ associated with the eigenfunctions of
$\Delta_h+V_0(y)+(n-2)^2/4$ and we then write the Schr\"odinger propagator as a linear combination
of products of the Hankel transform of the radial part and
eigenfunctions of $\Delta_h+V_0(y)+(n-2)^2/4$. We finally close this section by studying the Sobolev space and
the $L^p$-product chain rule which will serve the existence theory of nonlinear solution.  \vspace{0.2cm}

\subsection{Hankel transform and Bessel function} Let $(r, y)\in
\R_+\times Y$ be some polar coordinates. Then the operator $\LL_V$ is given by
\begin{equation}\label{L_V}
\mathcal{L}_V=\Delta_{g}+\frac{V_0(y)}{r^2},
\end{equation}
on metric cone $X=(0,\infty)_r\times Y$ where $V_0(y)$ is a real
continuous function and metric  $g$ in coordinates
$(r, y)\in \R_+\times Y$ is a metric of the form
\begin{equation*}
g=dr^2+r^2h(y,dy).
\end{equation*}
Note that the Riemannian metric $h$ on $Y$ is independent of $r$ hence we can use the separation of variable method.
In the coordinate $(r,y)$ , we write $\mathcal{L}_V$
\begin{equation}\label{operator-t}
\mathcal{L}_V=-\partial^2_r-\frac{n-1}r\partial_r+\frac1{r^2}\big(\Delta_h+V_0(y)\big).
\end{equation}
where $\Delta_h$ is the Laplace-Beltrami operator on
$(Y,h)$. We assume the operator $\LL_V$ is strictly positive in $L^2(X;dg(z))$, that is, for any nonzero function $f(r,y)\in\mathcal{C}_c^\infty((0,\infty)\times Y)$
\begin{equation}
\langle \LL_V f, f\rangle=\int_0^\infty\int_Y \left(|\partial_r f|^2+\frac{1}{r^2}\big(|\nabla_h f|^2+V_0(y)|f|^2\big)\right)r^{n-1}drdh>0.
\end{equation}
From the Hardy inequality with the sharpest constant $(n-2)^2/4$ (e.g. see\cite{Carron}), the positivity can be guaranteed by  assuming that $V_0$ is a smooth function on $Y$ such that
\begin{equation}\label{a1}
\Delta_h+V_0(y)+(n-2)^2/4>0
\end{equation} is strictly positive on $L^2(Y)$ in a sense that for any $f\in L^2(Y)\backslash\{0\}$
\begin{equation*}
\left\langle\big(\Delta_h+V_0(y)+(n-2)^2/4\big) f,f\right\rangle_{L^2(Y)}>
0 .
\end{equation*}   We denote the smallest eigenvalue of $\Delta_h+V_0(y)+(n-2)^2/4$ by $\nu^2_0$ and second lowest eigenvalue $\nu_1^2$ , with $\nu_0,\nu_1>0$.
Following \cite{wang},  we define $\chi_\infty$
to be the set
\begin{equation}\label{set1}
\chi_\infty=\Big\{\nu: \nu=\sqrt{(n-2)^2/4+\lambda};~
\lambda~\text{is eigenvalue
of}~ \Delta_h+V_0(y)\Big\}.
\end{equation}
For $\nu\in\chi_\infty$, let $d(\nu)$ be the multiplicity of
$\lambda_\nu=\nu^2-\frac14(n-2)^2$ as eigenvalue of
$\widetilde{\Delta}_h:=\Delta_h+V_0(y)$. Let $\{\varphi_{\nu,\ell}(y)\}_{1\leq
\ell\leq d(\nu)}$ be the eigenfunctions of $\widetilde{\Delta}_h$, that
is
\begin{equation}\label{eig-v}
\widetilde{\Delta}_h\varphi_{\nu,\ell}=\lambda_{\nu}\varphi_{\nu,\ell},
\quad \langle
\varphi_{\nu,\ell},\varphi_{\nu,\ell'}\rangle_{L^2(Y)}=\delta_{\ell,\ell'}= \begin{cases} 1, \quad \ell=\ell'\\ 0, \quad \ell\neq \ell'.
\end{cases}
\end{equation}
Let $\mathcal{H}^{\nu}=\text{span}\{\varphi_{\nu,1},\ldots,
\varphi_{\nu,d(\nu)}\}$, we  have the orthogonal decomposition of the $L^2(Y)$ in a sense that
\begin{equation*}
L^2(Y)=\bigoplus_{\nu\in\chi_\infty} \mathcal{H}^{\nu}.
\end{equation*} Define the orthogonal projection $\pi_{\nu}$ on $f\in L^2(X)$
\begin{equation*}
\pi_{\nu}f=\sum_{\ell=1}^{d(\nu)}\varphi_{\nu,\ell}(y)\int_{Y}f(r,y)
\varphi_{\nu,\ell}(y)dh:= \sum_{\ell=1}^{d(\nu)}\varphi_{\nu,\ell}(y) a_{\nu,\ell}(r)
\end{equation*}
where $dh$ is the measure on $Y$ under the
metric $h$. For any $f\in L^2(X)$, we can write $f$ in the form of separation of variable
\begin{equation}\label{sep.v}
f(z)=\sum_{\nu\in\chi_\infty}\pi_{\nu}f
=\sum_{\nu\in\chi_\infty}\sum_{\ell=1}^{d(\nu)}a_{\nu,\ell}(r)\varphi_{\nu,\ell}(y)
\end{equation}
and furthermore
\begin{equation}\label{norm1}
\|f(r,y)\|^2_{L^2(Y)}=\sum_{\nu\in\chi_\infty}\sum_{\ell=1}^{d(\nu)}|a_{\nu,\ell}(r)|^2.
\end{equation}
Let  $\nu>-\frac12$ and $r>0$ and define the Bessel function of order $\nu$ by  its \emph{Poisson representation formula}
\begin{equation}\label{Bessel-f}
J_{\nu}(r)=\frac{(r/2)^{\nu}}{\Gamma\left(\nu+\frac12\right)\Gamma(1/2)}\int_{-1}^{1}e^{isr}(1-s^2)^{(2\nu-1)/2}\mathrm{d
}s
\end{equation} which satisfies the following equation
\begin{equation*}
r^2\frac{d^2}{dr^2}( J_{\nu}(r))+r\frac{d}{dr} (J_{\nu}(r))+(r^2-\nu^2)J_{\nu}(r)=0.
\end{equation*}
A simple computation gives the rough estimates
\begin{equation}\label{bessel-r}
|J_\nu(r)|\leq
\frac{Cr^\nu}{2^\nu\Gamma(\nu+\frac12)\Gamma(1/2)}\left(1+\frac1{\nu+1/2}\right),
\end{equation}
where $C$ is a absolute constant. Let $f\in L^2(X)$, using the Bessel function, we define the Hankel transform of order $\nu$ by
\begin{equation}\label{hankel}
(\mathcal{H}_{\nu}f)(\rho,y)=\int_0^\infty(r\rho)^{-\frac{n-2}2}J_{\nu}(r\rho)f(r,y)r^{n-1}dr.
\end{equation}
Now we outline some notions on functional calculus for cones \cite{Taylor}.
Following the \cite[(8.45) ]{Taylor}, for well-behaved functions $F$, we have by
\begin{equation}\label{funct}
F(\mathcal{L}_V) g(r,y)=\sum_{\nu\in\chi_\infty}\sum_{\ell=1}^{d(\nu)} \varphi_{\nu,\ell}(y) \int_0^\infty F(\rho^2) (r\rho)^{-\frac{n-2}2}J_\nu(r\rho)b_{\nu,\ell}(\rho)\rho^{n-1} d\rho
\end{equation}
where $b_{\nu,\ell}(\rho)=(\mathcal{H}_{\nu}a_{\nu,\ell})(\rho)$ with $g(r,y)=\sum_{\nu\in\chi_\infty}\sum_{\ell=1}^{d(\nu)}a_{\nu,\ell}(r)~\varphi_{\nu,\ell}(y)$.
For $u_0\in L^2(X)$, we can write $u_0$ in the form of separation of variables by \eqref{sep.v} as
$$
u_0(z)=\sum_{\nu\in\chi_\infty}\sum_{\ell=1}^{d(\nu)}a_{\nu,\ell}(r)\varphi_{\nu,\ell}(y).
$$
Therefore the solution of the Cauchy problem
\begin{equation}\label{equ}
\begin{cases}
i\partial_{t}u+\mathcal{L}_V u=0,\\
u(0,z)=u_0(z),
\end{cases}
\end{equation}
can be written  in the form of Hankel transform representation via using \eqref{funct} with $F(\rho^2)=e^{it\rho^2}$, we have
\begin{equation}\label{s.exp}
\begin{split} &u(t,z)=e^{it\mathcal{L}_V}u_0=v(t,r,y)
\\&=\sum_{\nu\in\chi_\infty}\sum_{\ell=1}^{d(\nu)}\varphi_{\nu,\ell}(y)\int_0^\infty(r\rho)^{-\frac{n-2}2}J_{\nu}(r\rho)e^{
it\rho^2}b_{\nu,\ell}(\rho)\rho^{n-1}d\rho
\\&=\sum_{\nu\in\chi_\infty}\sum_{\ell=1}^{d(\nu)}\varphi_{\nu,\ell}(y)\mathcal{H}_{\nu}\big[e^{
it\rho^2}b_{\nu,\ell}(\rho)\big](r).
\end{split}
\end{equation}
where $b_{\nu,\ell}(\rho)=(\mathcal{H}_{\nu}a_{\nu,\ell})(\rho)$. We refer alternatively the reader to \cite{zhang} for more details about this.

Next, we remind some properties of Bessel function $J_\nu(r)$ discussed in
\cite{Stein1,Watson}.

\begin{lemma}[Asymptotics of the Bessel function] \label{lem:Bessel} Assume $\nu\gg1$. Let $J_\nu(r)$ be
the Bessel function of order $\nu$ defined as in \eqref{Bessel-f}. Then there
exist a large constant $C$ and a small constant $c$ independent from
$\nu$ and $r$ such that:

$\bullet$ If $r\leq \frac \nu2$, then
\begin{equation}\label{est:b<}
\begin{split}
|J_\nu(r)|\leq C e^{-c(\nu+r)};
\end{split}
\end{equation}

$\bullet$ If $\frac \nu 2\leq r\leq 2\nu$, then
\begin{equation}\label{est:b-}
\begin{split}
|J_\nu(r)|\leq C \nu^{-\frac13}(\nu^{-\frac13}|r-\nu|+1)^{-\frac14};
\end{split}
\end{equation}

$\bullet$ If $r\geq 2\nu$, then
\begin{equation}\label{est:b>}
\begin{split}
 J_\nu(r)=r^{-\frac12}\sum_{\pm}a_\pm(\nu,r) e^{\pm ir}+E(\nu,r),
 \end{split}
\end{equation}
where $|a_\pm(r,\nu)|\leq C$ and $|E(r,\nu)|\leq Cr^{-1}$. Furthermore, if $0\leq \nu \ll r$, there exists a constant $C_\alpha$ such that
 $$|\partial^\alpha_r a_\pm(\nu,r)|\leq C_\alpha(1+r)^{-\alpha}, \quad |\partial^\alpha_r E(\nu,r)|\leq C_\alpha(1+r)^{-\alpha-1}.$$ 
\end{lemma}\vspace{0.2cm}

As a direct consequence, we have the following lemma \cite[Lemma 2.2]{MZZ}.
\begin{lemma}\label{lem: J}Let $R\gg1$, then there exists a constant $C$ independent from $\nu$ and $R$ such that
\begin{equation}\label{est:b}
\int_{R}^{2R} |J_\nu(r)|^2 dr \leq C.
\end{equation}
\end{lemma} \vspace{0.2cm}

\subsection{$L^p$-product chain rule}
The $L^p$-product rule for fractional derivatives in Euclidean spaces was first proved by Christ and Weinstein \cite{CW}.
The chain rules of
differential operators of non-integer order is valid for $(-\Delta_{\R^n})^{\frac s2}$.  For example, if $1<p<\infty$ and $s>0$, then we have
$$
\| (-\Delta_{\R^n})^{\frac s2}(uv) \|_{L^p(\R^n)} \lesssim \| (-\Delta_{\R^n})^{\frac s2}u\|_{L^{p_1}(\R^n)} \| v \|_{L^{p_2}(\R^d)} + \| u \|_{L^{p_3}(\R^n)}\| (-\Delta_{\R^n})^{\frac s2}v \|_{L^{p_4}(\R^n)}
$$
whenever $\frac1p=\frac1{p_1}+\frac1{p_2}=\frac1{p_3}+\frac1{p_4}$.  For a textbook presentation of these theorems and original references, see \cite{Taylor1}.
However, there are some differences about this chain rule if one consider the perturbation of inverse-square potential, see \cite{KMVZZ1,ZZJFA}.
Since in our work we consider the solution of the nonlinear equation on $H^1(X)$, the result corresponding to Riesz transform can be obtained
from Hassell-Lin \cite{HL} which concluded that: let $\nabla_g=(\partial_r, r^{-1}\nabla_h)$ be the gradient on $X$, then there exists a constant $C$ such that
\begin{equation}\label{Riesz}
\|\nabla_g f\|_{L^p(X)}\leq C\| \sqrt{\LL_V}f\|_{L^p(X)};
\end{equation}
if and only if $p$ is in the interval $R_p$ given by
\begin{equation}\label{p-interval}
R_p:=\left(\frac{n}{\min\{1+\frac n2+\nu_0,n\}}, \frac{n}{\max\{\frac{n}2-\nu_0,0\}}\right)\end{equation}
where $\nu_0>0$ and $\nu_0^2$ is the smallest eigenvalue of the $\Delta_h+V_0+(n-2)^2/4$.
We record that the result about the chain rule as follows.

\begin{proposition}\label{Prop:Leibnitz}
Let $\LL_V$ as above and let $\nu_0>0$ such that $\nu_0^2$ is the smallest eigenvalue of the $\Delta_h+V_0+(n-2)^2/4$.  Then for all $u, v\in\CC_c^\infty(X^\circ)$,
compactly supported smooth functions on the
interior of the metric cone, we have
\begin{align}\label{Leibnitz}
\| \sqrt{\LL_V}(uv)\|_{L^p(X)} \lesssim \| \sqrt{\LL_V} u\|_{L^{p_1}(X)}\|v\|_{L^{p_2}(X)}+\|u\|_{L^{q_1}(X)}\| \sqrt{\LL_V} v\|_{L^{q_2}(X)},
\end{align}
for any exponents satisfying $p,  p', p_1, q_2\in R_p$ defined in \eqref{p-interval} and $\frac1p=\frac1{p_1}+\frac1{p_2}=\frac1{q_1}+\frac1{q_2}$.
\end{proposition}
\begin{remark} The result is not optimal due to the restriction of $p'\in R_p$ but it is enough to prove Theorem \ref{thm:NLS}. In particular for $n=3$, we have $R_p=\big(\frac{6}{\min\{5+2\nu_0,6\}}, \frac{6}{\max\{3-2\nu_0,0\}}\big)$.
\end{remark}

\begin{proof} The proof is based on the boundedness of Riesz transform discussed in \cite{HL} and a dual argument in \cite{Russ}. Let $f=uv$. Since $\sqrt{\LL_V}\mathcal{C}_0^\infty$ is dense in
$L^{p'}$ (see \cite[Appendix]{Russ}), then
$$\|\sqrt{\LL_V}f\|_{L^{p}}=\sup\Big\{\langle \sqrt{\LL_V}f,\sqrt{\LL_V}h\rangle: h\in \mathcal{C}_0^\infty(X^\circ),\|\sqrt{\LL_V}h\|_{L^{p'}}\leq 1\Big\}.$$
Therefore by the definition of the square root of $\LL_V=\nabla_g^*\nabla_g+V(z)$, we observe that
\begin{equation*}
\begin{split}
\|\sqrt{\LL_V} f\|_{L^{p}}&\leq \sup_{\|\sqrt{\LL_V}h\|_{L^{p'}}\leq 1} |\langle
\LL_V^{\frac12}f,\LL_V^{\frac12}h\rangle|=|\langle \LL_Vf,h\rangle|\\&\leq
\|\nabla_g f\|_{L^{p}}\|\nabla_g h\|_{L^{p'}}+\|V_0\|_{L^\infty(Y)}\left\|
r^{-1}f\right\|_{L^{p}}\left\|r^{-1}h\right\|_{L^{p'}}.
\end{split}
\end{equation*}
If $p,p'\in R_p$, then the equation \eqref{Riesz} and the Hardy inequality \eqref{est:hardy} below
imply that
\begin{equation}
\begin{split}
\|\sqrt{\LL_V} f\|_{L^{p}}\leq C \|\sqrt{\LL_V}
h\|_{L^{p'}}\left(\|\nabla_g f\|_{L^{p}}+\left\|
r^{-1}f\right\|_{L^{p}}\right).
\end{split}
\end{equation}
Note that for any  vector field $X_i$ on the manifold $X$, we can write
$$X_i(uv)=X_i(u)v+u X_i(v).$$ Apply this to the vector field $X_i=\partial_r$ and $X_i=r^{-1}\nabla_h$
and use the fact $\nabla_g=(\partial_r, r^{-1}\nabla_h)$, then we have $$\nabla_g(uv)=\nabla_g(u)v+u \nabla_g(v).$$
Hence we further prove by the H\"older inequality
\begin{equation}
\begin{split}
&\|\sqrt{\LL_V} (uv)\|_{L^{p}}\leq C\left(\|\nabla_g (uv)\|_{L^{p}}+\left\|
r^{-1}uv\right\|_{L^{p}}\right)\\&\leq C\left((\| \nabla_g u\|_{L^{p_1}(X)}+\|r^{-1}u\|_{L^{p_1}(X)})\|v\|_{L^{p_2}(X)}+\|u\|_{L^{q_1}(X)}\| \nabla_g v\|_{L^{q_2}(X)}\right).
\end{split}
\end{equation}
Hence by \eqref{Riesz} and the Hardy inequality \eqref{est:hardy} below, we prove \eqref{Leibnitz} if $p_1, q_2$ belong to \eqref{p-interval}.
\vspace{0.2cm}

\end{proof}

\subsection{Hardy inequality} In this subsection, we prove the Hardy inequality associated with $\LL_V$ by modifying the argument in \cite{HL}.
The Hardy inequality claims that
\begin{proposition}[Hardy inequality for $\LL_V$]\label{P:hardy}  Let $n\geq 3$ and $\nu_0$ be in Theorem \ref{thm:Strichartz'}. Suppose $0<s<\min\{1+\nu_0,2\}$, and $1<p<\infty$.  Then the inequality
\begin{equation}\label{est:hardy}
\big\|r^{-s}f(z)\big\|_{L^p(X)}\lesssim \big\|\mathcal L_V^\frac{s}2f\big\|_{L^p(X)}
\end{equation}
holds for
\begin{equation}\label{est:hardy_hyp}
n/\min\{1+\frac n2+\nu_0, n\}<p<n/\max\{s+\frac n2-1-\nu_0, 0\}.\end{equation}
\end{proposition}
\begin{remark} Note that $\nu_0>0$ and choose $s=1$, then \eqref{est:hardy_hyp} is exactly the same as the interval $R_p$ defined in \eqref{p-interval}.
\end{remark}

\begin{proof}[Proof of Proposition \ref{P:hardy}.]
The estimate \eqref{est:hardy} is equivalent to
\begin{equation}\label{est:hardyequi}
\Big\|r^{-s}\mathcal L_V^{-\frac{s}2}h\Big\|_{L^p(X)}\lesssim\|h\|_{L^p(X)}
\end{equation}
where the operator $\mathcal L_V^{-\frac{s}2}$ is defined by the Riesz potentials kernel
$$
\mathcal{L}_V^{-\frac{s}2}(z,z'):=\int_0^\infty \lambda^{1-s}(\LL_V+\lambda^2)^{-1}(z,z') d\lambda.
$$ Therefore we consider the operator $T=r^{-s}\mathcal L_V^{-\frac{s}2}$ with its kernel
$$
T(z,z')=r^{-s}\mathcal{L}_V^{-\frac{s}2}(z,z').
$$
Following the method used in \cite{HL}, we study the kernel $T(z,z')$ based on the resolvent kernel $(\LL_V+\lambda^2)^{-1}(z,z')$.

{\bf Step 1: $L^2$-bounded.} Note $dv=\sqrt{|g|}dz=r^{n-1}dr dh$. In particular $p=2$, we compute that by \eqref{sep.v} and \eqref{norm1}
\begin{align*}
&\big\|r^{-1}f(z)\big\|^2_{L^2(X)}\\&=\int_0^\infty\int_Y\frac{|\sum_{\nu\in\chi_\infty}\sum_{\ell=1}^{d(\nu)}a_{\nu,\ell}(r)\varphi_{\nu,\ell}(y)|^2}{r^2} r^{n-1}dr dh\\
&=\sum_{\nu\in\chi_\infty}\sum_{\ell=1}^{d(\nu)}\int_0^\infty \frac{|a_{\nu,\ell}(r)|^2}{r^2} r^{n-1}dr\\
&\lesssim \sum_{\nu\in\chi_\infty}\sum_{\ell=1}^{d(\nu)}\|\partial_r a_{\nu,\ell}(r)\|^2_{L^2(r^{n-1}dr)}\lesssim \|\nabla_g f\|^2_{L^2(X)}\lesssim \|\LL_V^{1/2} f\|^2_{L^2(X)}.
\end{align*}
Therefore $T$ is bounded on $L^2(X)$.

{\bf Step 2: $L^p$-bounded on far away diagonal region.}  Let $\chi\:[0,\infty)\to [0,1]$ be a smooth cutoff function such that $\chi([0,1/2])=1$ and $\chi([1,\infty))=0$. Define the operators
\begin{align}
T_1(z,z')=\chi(4r/r') r^{-s}\int_0^\infty \lambda^{1-s}(\LL_V+\lambda^2)^{-1}(z,z') d\lambda;
\\ T_2(z,z')=\chi(4r'/r) r^{-s}\int_0^\infty \lambda^{1-s}(\LL_V+\lambda^2)^{-1}(z,z') d\lambda.
\end{align}
Then we decompose $T=T_0+T_1+T_2$ where
\begin{align}
T_0(z,z')=(1-\chi(4r/r')-\chi(4r'/r) ) r^{-s}\int_0^\infty \lambda^{1-s}(\LL_V+\lambda^2)^{-1}(z,z') d\lambda.
\end{align}
In this subsection, we consider the $L^p$-boundedness of $T_1$ and $T_2$.
Since $\mathcal{L}_V$ is homogeneous of degree $-2$,  then by scaling we have
$$(\mathcal{L}_V+\lambda^2)^{-1}(z,z')=\lambda^{n-2}(\mathcal{L}_V+1)^{-1}(\lambda z,\lambda z').$$
To this end, we need two results from \cite{HL}. One is about the resolvent kernel $(\LL_V+1)^{-1}(z,z')$ given in \cite[Theorem 4.11]{HL}
and the other is about the boundedness of a kind of operator discussed in \cite[Corollary 5.9] {HL}.
By \cite[Theorem 4.11]{HL}, for any $N, M>0$ we have
\begin{align}\label{T1}
|\chi(4r/r')(\mathcal{L}_V+1)^{-1}(z,z')|\lesssim r^{1-\frac n2+\nu_0}r'^{1-\frac n2-\nu_0}\langle r'\rangle^{-N},
\end{align}
and \begin{align}\label{T2}
|\chi(4r'/r)(\mathcal{L}_V+1)^{-1}(z,z')|\lesssim r^{1-\frac n2-\nu_0}r'^{1-\frac n2+\nu_0}\langle r\rangle^{-M}.
\end{align}
Therefore, by \eqref{T1} for any $N>1-s$ and $s<2$, we have
\begin{equation}
\begin{split}
&\left|T_1(z,z')\right|\lesssim \left|r^{-s}\int_0^\infty \lambda^{n-1-s}\chi(4r/r')(\LL_V+1)^{-1}(\lambda z, \lambda z') d\lambda\right|\\
&\lesssim r^{-s}r'^{2-n}(r/r')^{1-\frac n2+\nu_0}\left(\int_0^{1/r'} \lambda^{1-s}d\lambda+r'^{-N}\int_{1/r'}^\infty \lambda^{1-s-N}d\lambda\right)\\
&\lesssim r'^{-n}(r/r')^{1-\frac n2+\nu_0-s}.
\end{split}
\end{equation}
Similarly we have
\begin{equation}
\begin{split}
&\left|T_2(z,z')\right|\lesssim \left|r^{-s}\int_0^\infty \lambda^{n-1-s}\chi(4r'/r)(\LL_V+1)^{-1}(\lambda z, \lambda z') d\lambda\right|\\
&\lesssim r^{-s}r^{2-n}(r'/r)^{1-\frac n2+\nu_0}\left(\int_0^{1/r} \lambda^{1-s}d\lambda+r^{-N}\int_{1/r}^\infty \lambda^{1-s-N}d\lambda\right)\\
&\lesssim r^{-n}(r'/r)^{1-\frac n2+\nu_0}.
\end{split}
\end{equation}
Hence by \cite[Corollary 5.9] {HL}, $T_1$ is bounded on $L^p(X)$ for $1<p<n/\max\{s+\frac n2-1-\nu_0, 0\}$ and that $T_2$ is bounded on $L^p(X)$ for $p>n/\min\{1+\frac n2+\nu_0, n\}$.

{\bf Step 3: $L^p$-bounded on diagonal region.}  Recall that the distance on a metric cone is
\begin{equation}
d(z,z')=\begin{cases}\sqrt{r^2+r'^2-2rr'\cos(d_Y(y,y'))},\quad &d_Y(y,y')\leq \pi;\\
r+r', &d_Y(y,y')\geq \pi.\
\end{cases}
\end{equation}
In the diagonal region (i.e. the support of $1-\chi(4r/r')-\chi(4r'/r)$), we have
\begin{equation}\label{dis-r}
d(z,z')^{-1}\geq (r+r')^{-1}=r^{-1}(1+r'/r)^{-1}\geq r^{-1}/9.\end{equation}
We claim the following conclusions:

 (i) $T_0$ is bounded on $L^2(X)$;

 (ii) $|T_0(z,z')|\leq C d(z,z')^{-n}$;

 (iii) $|\nabla_z T_0(z,z')|\leq C d(z,z')^{-(n+1)}$ and $|\nabla_{z'} T_0(z,z')|\leq C d(z,z')^{-(n+1)}$.

Now we verify these conclusions. We have seen that $T, T_1, T_2$ are bounded on $L^2(X)$ from the first two steps if $s<1+\nu_0$, hence $T_0=T-T_1-T_2$ is also bounded on $L^2(X)$.
To verity (ii) and (iii), we recall \cite[Lemma 5.4] {HL} which implies for any integer $j\geq0$ and any $N>0$ that
\begin{align*}
\left|\nabla^j_{z,z'}K(z,z')\right|\lesssim
\begin{cases}d(z,z')^{2-n-j}, \quad &d(z,z')\leq 1;\\d(z,z')^{-N}, \quad &d(z,z')\geq 1.
\end{cases}
\end{align*}
where $K(z,z')=(1-\chi(4r/r')-\chi(4r'/r) )(\mathcal{L}_V+1)^{-1}(z,z')$.
Therefore, by using $d(\lambda z, \lambda z')=\lambda d(z,z')$, we compute that
\begin{align*}
\left|K(\lambda z,\lambda z')\right|\lesssim
\begin{cases}\lambda^{2-n} d(z,z')^{2-n}, &d(z,z')\leq 1/\lambda;\\\lambda^{-N}d(z,z')^{-N},  &d(z,z')\geq 1/\lambda.
\end{cases}
\end{align*}
and
\begin{align*}
\left|\nabla_{z,z'}\left(K(\lambda z,\lambda z')\right)\right|\lesssim
\begin{cases}\lambda^{2-n} d(z,z')^{1-n}, &d(z,z')\leq 1/\lambda;\\\lambda^{-N+1}d(z,z')^{-N},  &d(z,z')\geq 1/\lambda.
\end{cases}
\end{align*}
Let $s<2$ and let $N>n-s$, by considering \eqref{dis-r} we obtain
\begin{equation*}
\begin{split}
|T_0(z,z')|&= r^{-s}\int_0^\infty \lambda^{n-1-s}|K(\lambda z,\lambda z')| d\lambda\\
&\lesssim d(z,z')^{-s}\left(d(z,z')^{2-n}  \int_0^{1/d(z,z')} \lambda^{1-s} d\lambda+d(z,z')^{-N}\int_{1/d(z,z')}^\infty \lambda^{n-1-s-N}  d\lambda\right)
\\&\lesssim d(z,z')^{-n}
\end{split}
\end{equation*}
which verifies (ii). Similarly we have
\begin{equation*}
\begin{split}
&|\nabla_{z} T_0(z,z')|\lesssim r^{-s-1}\int_0^\infty \lambda^{n-1-s}|K(\lambda z,\lambda z')| d\lambda+r^{-s}\int_0^\infty \lambda^{n-1-s}|\nabla_z(K(\lambda z,\lambda z'))| d\lambda\\
&\lesssim d(z,z')^{-n-1}+d(z,z')^{-s}\left(d(z,z')^{1-n}  \int_0^{1/d(z,z')} \lambda^{1-s} d\lambda+d(z,z')^{-N}\int_{1/d(z,z')}^\infty \lambda^{n-s-N}  d\lambda\right)
\\&\lesssim d(z,z')^{-n-1}.
\end{split}
\end{equation*}
We also have $|\nabla_{z'} T_0(z,z')|\lesssim d(z,z')^{-n-1}$. These two inequalities prove (iii). As a consequence of Calder\'on-Zygmund theory, we conclude that $T_0$ is bounded from $L^1(X)$ to weak $L^1(X)$.
By using interpolation, we conclude that $T_0$ is bounded on $L^p(X)$ for all $1<p\leq 2$. By duality, the following proposition also holds.
\begin{proposition}
The operator $T_0$ is a bounded operator on $L^p(X)$ for all $p>1$.
\end{proposition}
Having all results from the last two steps in mind, we prove $T$ is bounded on $L^p(X)$ for all $p$ that satisfies \eqref{est:hardy_hyp}. This completes the proof of the Proposition \ref{P:hardy}.

\end{proof}

\section{Spectral measure and Strichartz estimates associated with $\mathcal{L}_0$}

In this section, we study the spectral measure associated with the operator $H=\sqrt{\mathcal{L}_0}$ where $\LL_0=\Delta_g$ is the positive Laplacian on
the metric cone. And then we apply the property of the spectral measure to prove the Strichartz estimates when $V=0$.

\subsection{Spectral measure} First we briefly review the spectral measure. For the following definition, we follow  Kato \cite{Kato}.

Consider the given self-adjoint operator $H: D(H)\subset \mathcal{H}\to \mathcal{H}$ where $D(H)$ is the domain of $H$ and $\mathcal{H}$ is a Hilbert space.
The \emph{spectral family} $\{E(\lambda)\}$ of $H$, also known as a \emph{resolution of the identity}, is a family
of projection operators in $\mathcal{H}$ with the following properties:

$\bullet$ $E(\lambda)$ is nondecreasing, i.e., $E(\mu)\leq E(\lambda)$ for $\mu\leq \lambda$;

$\bullet$ the strong convergence property holds, that is, $\lim_{\lambda\to -\infty} E(\lambda)=0$,  $\lim_{\lambda\to \infty} E(\lambda)=1$.

Let $u, v\in \mathcal{H}$ be given. From the spectral family, one can define a complex function of \emph{ bounded variation} on  the real line
\begin{equation*}
\R\ni\lambda \to (E(\lambda)u, v)_{\mathcal{H}}.
\end{equation*}
It is known that such function gives a measure (depending on $u,v$) called \emph{spectral measure}. In our situation, the\emph{ subspace of absolute continuity} $\mathcal{H}_{\mathrm{ac}}$
with respect to the operator $H$ is equal to $\mathcal{H}=L^2$, in other word, the operator $H$ is \emph{spectrally absolutely continuous}. From \cite[Chapter X, Theorem 1.7]{Kato}, the bilinear form
\begin{equation*}
\frac{d}{d\lambda}\langle E(\lambda)u, v\rangle_{L^2}:L^2\times L^2\to \C
\end{equation*}
is bounded, and then it induces a bounded operator $A(\lambda):L^2\to L^2$ defined via
\begin{equation*}
\langle A(\lambda)u, v\rangle_{L^2}=\frac{d}{d\lambda}\langle E(\lambda)u, v\rangle_{L^2}=\langle \frac{d}{d\lambda} E(\lambda)u, v\rangle_{L^2}, \quad u, v\in L^2.
\end{equation*}
The operator $A(\lambda)$ is called the \emph{density of states} of the operator $H$. Finally we define the spectral measure associated with $H=\sqrt{\LL_0}$
\begin{equation}\label{spect-measure}
dE_{\sqrt{\LL_0}}(\lambda)= A(\lambda)d\lambda
\end{equation}
where $d\lambda$ is the Lebesgue measure. The main purpose of this section is to prove the following proposition.

\begin{proposition}
\label{prop:localized spectral measure} Let $(X,g)$ be metric cone manifold and $\LL_0=\Delta_g$. Then there exists a $\lambda$-dependent  operator partition of unity on
$L^2(X)$
$$
\mathrm{Id}=\sum_{j=0}^{N}Q_j(\lambda),
$$
with $N$ independent of $\lambda$,
such that for each $0 \leq j \leq N$ we can write
\begin{equation}\label{beanQ}\begin{gathered}
(Q_j(\lambda)dE_{\sqrt{\LL_0}}(\lambda)Q_j^*(\lambda))(z,z')=\lambda^{n-1} \Big(  \sum_{\pm} e^{\pm
i\lambda d(z,z')}a_\pm(\lambda; z,z') +  b(\lambda; z, z') \Big),
\end{gathered}\end{equation}
and for $0\leq j, j'\leq N$ with either $j=0$ or $j'=0$, also we can write
\begin{equation}\label{beanQ'}\begin{gathered}
(Q_j(\lambda)dE_{\sqrt{\LL_0}}(\lambda)Q_{j'}^*(\lambda))(z,z')=\lambda^{n-1}\Big(\sum_{\pm} e^{\pm
i\lambda\max\{r, r'\}}\tilde{a}_\pm(\lambda; z,z') + c(\lambda; z, z') \Big),
\end{gathered}\end{equation}
where $a_\pm(\lambda, z, z')$, $b(\lambda, z, z')$ and $c(\lambda,z,z') $ satisfy the estimates for any $\alpha\geq0$
\begin{equation}\label{bean}\begin{gathered}
\big|\partial_\lambda^\alpha a_\pm(\lambda,z,z') \big|+\big|\partial_\lambda^\alpha \tilde{a}_\pm(\lambda,z,z') \big|\leq C_\alpha
\lambda^{-\alpha}(1+\lambda d(z,z'))^{-\frac{n-1}2},
\end{gathered}\end{equation}
\begin{equation}\label{beans}\begin{gathered}
\big| \partial_\lambda^\alpha b(\lambda,z,z') \big|\leq C_{\alpha}
\lambda^{-\alpha}(1+\lambda d(z,z'))^{-K} \text{ for any } K\geq 0,
\end{gathered}\end{equation}
and
\begin{equation}\label{beanc}\begin{gathered}
\big| \partial_\lambda^\alpha c(\lambda,z,z') \big|\leq C_{\alpha, X}
\lambda^{-\alpha}(1+\lambda d(z,z'))^{-\frac{n+1}2}.
\end{gathered}\end{equation}
Here $d(\cdot, \cdot)$ is the distance on $X$.
\end{proposition}

\begin{remark}The expression of the spectral measure captures both the decay and the oscillatory behaviour of the spectral measure, which is crucial in obtaining sharp
microlocalized dispersive estimates.
\end{remark}
\begin{remark}The expression \eqref{beanQ'} with estimate \eqref{beanc} are only used in the proof of inhomogeneous Strichartz estimates, not for homogeneous estimate.
\end{remark}

We prove this proposition by considering two separate regimes, near the cone tip and far away from the cone tip. Near the cone tip,  we use a variable separating expression for the spectral
measure as in Lemma \ref{lem:spect} below. Far away from the cone tip, we combine the argument of \cite[Proposition 1.5] {HZ}.\vspace{0.2cm}

Now we show an explicit formula for the spectral measure which will be used to treat the regime near the cone tip.
\begin{lemma}\label{lem:spect} Let $\nu_j^2$ be the eigenvalues of the positive operator
$\Delta_h+V_0(y)+(n-2)^2/4$ and let
$\varphi_j(y)$ be $L^2$-normalized corresponding eigenfunction, we have the explicit formula for the spectral measure given by
\begin{equation}
\begin{split}
dE_{\sqrt{\LL_V}}(1,z,z')=\frac{\pi}2
(rr')^{-\frac{n-2}2}\sum_{j}\varphi_j(y)\overline{\varphi_j(y')}J_{\nu_j}(r)J_{\nu_j}(r').
\end{split}
\end{equation}

\end{lemma}

\begin{remark}
Here we will show the result with the potential $V$ even though we only need the result with $V\equiv 0$ for considering the spectral measure associated with $\LL_0$ near the cone tip in the
following argument.\end{remark}

\begin{remark}
This expression provides little asymptotic behavior
of the kernel, as both $r,r'$ tend to $\infty$, but gives good converges as both $r,r'$ tend to $0$.\end{remark}

\begin{proof} The proof is from \cite{GHS1, HL} which give an explicit formula for the resolvent $(\LL_V-(1+i0))^{-1}$. We consider the expression of the spectral measure here.
Write the operator $\LL_V$ on $X$ as
\begin{equation}\label{LvPb}
\begin{split}
\LL_V&=-\partial_r^2-\frac{n-1}r\partial_r+\frac{\Delta_h}{r^2}+\frac{V_0(y)}{r^2}\\&=r^{-1-\frac
n2}\Big(-(r\partial_r)^2+\Delta_h+V_0(y)+(n-2)^2/4\Big)r^{\frac
n2-1} \\&=r^{-1-\frac n2}P_br^{\frac n2-1}
\end{split}
\end{equation}
where $$P_b=-(r\partial_r)^2+\Delta_h+V_0(y)+(n-2)^2/4.$$ Let $g_b=r^{-2}g=r^{-2}dr^2+h$ the conformal metric to $g$. Then we can check that
$\LL_V$ is formally self-adjoint with respect to $g$, and $P_b$ is formally self-adjoint with respect to $g_b$. Indeed, for any $f,\tilde{f}\in \mathcal{C}_0^\infty(X)$,  by integration by parts we have
\begin{equation}
\langle \LL_Vf, \tilde{f}\rangle_{L^2_g(X)}=\int_{X} (r^{-1-\frac n2}P_br^{\frac n2-1}f) \tilde{f} r^{n-1}drdh=\langle f, \LL_V\tilde{f}\rangle_{L^2_g(X)}.
\end{equation}
and
\begin{equation}
\langle P_b f, \tilde{f}\rangle_{L^2_{g_b}(X)}=\int_{X} P_b f \tilde{f} \frac{dr}{r}dh=\langle f, P_b\tilde{f}\rangle_{L^2_{g_b}(X)}.
\end{equation}
Our purpose is to obtain the Schwartz kernel of the operator $(\LL_V+k^2)^{-1}$. First we consider the kernel of  the operator $\LL_V+k^2$.
To this end, we regard this operator as acting on half-densities, using the flat connection on half-densities that annihilates the Riemannian half-density $|dg|^{\frac12}$.
It is natural to introduce the operators $\LL_V$ as acting on half-densities $|dg|^{\frac12}$ by the
formula
$$(\LL_V+k^2)(f|dg|^{\frac12})=\big((\LL_V+k^2) f\big)|dg|^{\frac12},\quad \text{with}~ dg=r^{n-1}drdh.$$
A way to understand these formulas is that the term vanishes if the derivatives hit the half-densities. This occurs because the derivatives are endowed with the flat connection on the half-density bundle which
annihilates the half-density $|dg|^{1/2}$.

However we want to consider the operator $\LL_V+k^2$ acting on
half-density $|dg_b|^{1/2}$, as it is done in\cite{GHS1}. Therefore
from \eqref{LvPb} we obtain
\begin{equation*}
\begin{split}
&r(\LL_V+k^2)r(f|dg_b|^{\frac12})=r(\LL_V+k^2)r(f|r^{-1}drdh|^{\frac12})\\&=r^{-\frac n2}\left(r^{1+\frac n2}(\LL_V+k^2)r^{1-\frac n2}\right)(f|r^{n-1}drdh|^{\frac12})\\&=
r^{-\frac n2}\left((P_b+k^2r^2)f\right)(|r^{n-1}drdh|^{\frac12})=
\left((P_b+k^2r^2)f\right)(|dg_b|^{\frac12}).
\end{split}
\end{equation*}
We use that the flat connection on half-densities that annihilates the Riemannian half-density $|dg|^{1/2}$ in the third equality.
 This implies that $r(\LL_V+k^2)r$ is equivalent to the operator $P_b+k^2r^2$ endowed with the flat connection that annihilates the half-density $|dg_b|^{1/2}=|r^{-1}drdh|^{1/2}$.
Hence if we consider the Schwartz kernel of $(P_b+k^2r^2)^{-1}$ denoted by $G$, then
the Schwartz kernel of $(\LL_V+k^2)^{-1}$ is $rr'G$ with respect to
$$w=|dg_bdg'_b|^{\frac12}=|(rr')^{-1}dr dr'dh(y)dh(y')|^{\frac12}.$$
Now consider the resolvent kernel
\begin{equation}
G=\big(P_b+k^2r^2\big)^{-1}\quad \text{for}~k>0.
\end{equation}
Recall that $\nu_j^2$ is the eigenvalues of the positive operator
$$\Delta_h+V_0(y)+(n-2)^2/4$$ corresponding to the $L^2$-normalized eigenfunction
$\varphi_j(y)=\varphi_{\nu_j}(y)$, that is
\begin{equation*}
\big(\Delta_h+V_0(y)+(n-2)^2/4\big)\varphi_j=\nu_j^2\varphi_j.
\end{equation*}
Let $\Pi_j=\pi_{\nu_j}$ be projection onto the $\varphi_j$ eigen-space.  Then we have
the decomposition
\begin{equation*}
P_b+k^2r^2=\sum_{j}\big(-(r\partial_r)^2+k^2r^2+\nu_j^2\big)\Pi_j.
\end{equation*}
Therefore, we have
\begin{equation*}
\big(P_b+k^2r^2\big)^{-1}=\sum_{j}\Pi_j\big(-(r\partial_r)^2+k^2r^2+\nu_j^2\big)^{-1}.
\end{equation*}
Let $T_j=-(r\partial_r)^2+k^2r^2+\nu_j^2$ and let $T_j^{-1}(r,r')$ be
the kernel of the inverse $T^{-1}_j$. When $r\neq r'$, as in \cite{GH}, the solution
space is spanned by the modified Bessel functions
\begin{equation}
\begin{split}
I_{\nu_j}(r)&=\frac{2^{-\nu_j}r^{\nu_j}}{\sqrt{\pi}\Gamma(\nu_j+\frac12)}\int_{-1}^1(1-t^2)^{\nu_j-\frac12}e^{-rt}dt,\\
K_{\nu_j}(r)&=\frac{\sqrt{\pi}2^{-\nu_j}r^{\nu_j}}{\Gamma(\nu_j+\frac12)}\int_{1}^\infty(t^2-1)^{\nu_j-\frac12}e^{-rt}dt.
\end{split}
\end{equation}
Therefore, similar to the case \cite{HL}, we have
\begin{equation}
T_j^{-1}(r,r')=\begin{cases}
I_{\nu_j}(kr)K_{\nu_j}(kr')\Big|\frac{dr}r\frac{dr'}{r'}\Big|^{\frac12},\quad
r<r';\\
K_{\nu_j}(kr)I_{\nu_j}(kr')\Big|\frac{dr}r\frac{dr'}{r'}\Big|^{\frac12},\quad
r>r'. \end{cases}
\end{equation}
Hence from the Sturm-Liouville theory, we obtain an explicit formula separating the $r$ and $y$ variables for the resolvent kernel given by
\begin{equation}
\big(P_b+k^2r^2\big)^{-1}=\begin{cases}
\sum_{j}\varphi_j(y)\overline{\varphi_j(y')}I_{\nu_j}(kr)K_{\nu_j}(kr')\Big|\frac{dr}r\frac{dr'}{r'} dh(y) dh(y')\Big|^{\frac12},\quad
r<r';\\
\sum_{j}\varphi_j(y)\overline{\varphi_j(y')}
K_{\nu_j}(kr)I_{\nu_j}(kr')\Big|\frac{dr}r\frac{dr'}{r'}dh(y) dh(y')\Big|^{\frac12},\quad
r>r'.
\end{cases}
\end{equation}
This formula analytically continues to the imaginary axis, so
setting $k=-i$, and using the following formulae
$$I_{\nu}(-iz)=e^{-\nu\pi i/2}J_{\nu}(z), \quad K_{\nu}(-iz)=\frac{\pi i}2e^{\nu\pi
i/2}H^{(1)}_{\nu}(z),$$ we see that
\begin{equation}
\big(\LL_V-(1+i0)\big)^{-1}=\begin{cases} \frac{\pi
irr'}2\sum_{j}\varphi_j(y)\overline{\varphi_j(y')}J_{\nu_j}(r)H^{(1)}_{\nu_j}(r')\Big|\frac{drdr'}{rr'}dh(y)dh(y')\Big|^{\frac12},\quad
r<r';\\
\frac{\pi
irr'}2\sum_{j}\varphi_j(y)\overline{\varphi_j(y')}J_{\nu_j}(r')H^{(1)}_{\nu_j}(r)\Big|\frac{drdr'}{rr'}dh(y) dh(y')\Big|^{\frac12},\quad
r>r',
\end{cases}
\end{equation}
where $J_\nu$ and $H^{(1)}_{\nu}$ are standard Bessel and Hankel
functions, see e.g. \cite{Watson}.  Let $H$ be the Heaviside function, since $\mathrm{Im}(i
H_{\nu}^{(1)})(r)=J_\nu(r)$, we have
\begin{equation*}
\begin{split}
&dE_{\sqrt{\LL_V}}(1,z,z')=
\mathrm{Im}\Big(\big(\LL_V-(1+i0)\big)^{-1}\Big)\\&=\frac{\pi
rr'}2\sum_{j}\varphi_j(y)\overline{\varphi_j(y')}\mathrm{Im}\Big(J_{\nu_j}(r)iH^{(1)}_{\nu_j}(r')H(r'-r)
+J_{\nu_j}(r')iH^{(1)}_{\nu_j}(r)H(r-r')\Big)\Big|\frac{drdr'}{rr'}dh(y) dh(y')\Big|^{\frac12}\\&=\frac{\pi
rr'}2\sum_{j}\varphi_j(y)\overline{\varphi_j(y')}\Big(J_{\nu_j}(r)J_{\nu_j}(r')H(r'-r)
+J_{\nu_j}(r')J_{\nu_j}(r)H(r-r')\Big)\Big|\frac{drdr'}{rr'}dh(y) dh(y')\Big|^{\frac12}
\\&=\frac{\pi
rr'}2\sum_{j}\varphi_j(y)\overline{\varphi_j(y')}J_{\nu_j}(r)J_{\nu_j}(r')\Big|\frac{drdr'}{rr'}dh(y) dh(y')\Big|^{\frac12}.
\end{split}
\end{equation*}
Now we return to half-density $|dg|^{1/2}$ and obtain
\begin{equation*}
\begin{split}
dE_{\sqrt{\LL_V}}(1,z,z')=\frac{\pi}2
(rr')^{-\frac{n-2}2}\sum_{j}\varphi_j(y)\overline{\varphi_j(y')}J_{\nu_j}(r)J_{\nu_j}(r')\Big|dg dg'\Big|^{\frac12}.
\end{split}
\end{equation*}

\end{proof}

{\bf The proof of Proposition \ref{prop:localized spectral measure}. }
Now we prove the main result of this section. The main idea is microlocalizing the spectral measure associated with $\LL_0$ to capture the decay and oscillation behavior, as described in Proposition
\ref{prop:localized spectral measure}.
According to the Stone's formulae,
the Schwartz kernel of the spectral measure can be expressed in
terms of the difference between the outgoing and incoming resolvents
$R(\lambda\pm i0)$, where $R(\sigma)=(\LL_0-\sigma^2)^{-1}$.
More precisely,
\begin{equation*}
dE_{\sqrt{\LL_0}}(\lambda)=\frac{d}{d\lambda}E_{\sqrt{\LL_0}}(\lambda)d\lambda=\frac{\lambda}{\pi
i}\big(R(\lambda+i0)-R(\lambda-i0)\big)d\lambda,\quad \lambda>0.
\end{equation*}
Write the resolvent as a kernel $$R(\sigma)(z,z')=(\LL_0-\sigma^2)^{-1}(r,y,r',y')$$ where $z=(r,y)$ and $z'=(r',y')$ are the left and right  variables respectively.  Since $\mathcal{L}_0$ is homogeneous of degree $-2$,  then
$$(\mathcal{L}_0-\lambda^2)^{-1}(r,y,r',y')=\lambda^{n-2}(\mathcal{L}_0-1)^{-1}(\lambda r,y,\lambda r',y'),$$
hence
\begin{equation*}
dE_{\sqrt{\LL_0}}(\lambda; r,y,r',y')=\frac{\lambda^{n-1}}{\pi i} dE_{\sqrt{\LL_0}}(1;\lambda r,y,\lambda r',y').
\end{equation*}
Recall $X=(0,\infty)\times Y$, then $(z,z')=(r,y,r',y')\in X\times X$. Define $\tilde{r}=\lambda r$ and $\tilde{r}'=\lambda r'$. Due to the scaling invariant of $X$, $(\tilde{r},y,\tilde{r}',y')\in X\times X$.
Hence, from now on, we instead consider the spectral measure $dE_{\sqrt{\LL_0}}(1; \tilde{r},y,\tilde{r}',y')$ on $X\times X$.\vspace{0.2cm}

We  employ a similar partition of the identity as discussed in \cite{GHS2, HZ}. This partition is defined by
specifying the symbols of these operators, which must form a partition of unity on the phase space. Let $\tilde{x}=1/\tilde{r}$ and $\epsilon=1/R$ with $R$ being chosen later.

{\bf $\bullet$ The region near the cone tip.} Let $O_0$ consist of all points near the cone tip, that is, the points with $\tilde{x}> \epsilon$, i.e. $\tilde{r}< R$.
 Define the operator $Q_0=\chi(\tilde{r})$ where $\chi\in \mathcal{C}^\infty_0(\R^+)$ satisfies $\chi(s)=1$ for $s\leq
R$ and $\chi(s)=0$ for $s\ge2R$. Now we prove the properties given in Proposition \ref{prop:localized spectral measure} associated with $Q_0$.

From Lemma \ref{lem:spect}, we have
\begin{equation*}
\begin{split}
&Q_0 dE_{\sqrt{\LL_0}}(1,\tilde{z},\tilde{z}')Q_0^*
\\&=\frac{\pi}2
(\lambda^2rr')^{-\frac{n-2}2}\sum_{j}\varphi_j(y)\overline{\varphi_j(y')}J_{\nu_j}(\lambda r)J_{\nu_j}(\lambda r')\chi(\lambda r)\chi(\lambda r').
\end{split}
\end{equation*}
Recall that if $\mathrm{Re} \nu>-1/2$, then one has the recursion relation, see for example \cite[Page 45, 3.2-(4)]{Watson} $$\frac{d}{dt} \left(t^{-\nu}J_{\nu}(t)\right)=-t^{-\nu} J_{\nu+1}(t),$$
hence we have
\begin{equation}
\begin{split}
&\Big|\Big(\frac{d}{d\lambda}\Big)^{\alpha}\Big(
(\lambda^2rr')^{-\frac{n-2}2}\sum_{j}\varphi_j(y)\overline{\varphi_j(y')}J_{\nu_j}(\lambda r)J_{\nu_j}(\lambda r')\chi(\lambda r)\chi(\lambda r')\Big)\Big|
\\&\lesssim \lambda^{-\alpha}(\lambda^2rr')^{-\frac{n-2}2}\sum_{\alpha_1+\alpha_2+\alpha_3+\alpha_4\leq\alpha; }(\lambda r)^{\alpha_1+\alpha_3} (\lambda r')^{\alpha_2+\alpha_4}
\\&\quad \times \Big|\sum_{j\geq 0}\varphi_j(y)\overline{\varphi_j(y')}J_{\nu_j+\alpha_1}(\lambda r)J_{\nu_j+\alpha_2}(\lambda r')\chi^{(\alpha_3)}(\lambda r)\chi^{(\alpha_4)}(\lambda r')\Big|.
\end{split}
\end{equation}
On the other hand, recall from \eqref {bessel-r} that, the Bessel function satisfies
\begin{equation*}
|J_\nu(t)|\leq
\frac{Ct^\nu}{2^\nu\Gamma(\nu+\frac12)\Gamma(1/2)}\big(1+\frac1{\nu+1/2}\big).
\end{equation*}

Notice that here we consider $V=0$, hence the smallest eigenvalue satisfies $\nu_0>(n-2)/2$. By H\"ormander's $L^\infty$-estimate, for example see \cite{Hor}, we know that
 $\|\varphi_j\|_{L^\infty}\leq C \nu_j^{(n-1)/2}$. Hence we obtain
\begin{equation}
\begin{split}
&\Big|\Big(\frac{d}{d\lambda}\Big)^{\alpha}\Big(
(\lambda^2rr')^{-\frac{n-2}2}\sum_{j}\varphi_j(y)\overline{\varphi_j(y')}J_{\nu_j}(\lambda r)J_{\nu_j}(\lambda r')\chi(\lambda r)\chi(\lambda r')\Big)\Big|
\\&\lesssim \lambda^{-\alpha}\sum_{j} \nu_j^{n-1}\frac{(\lambda r)^{\nu_j-\frac{n-2}2}}{2^{\nu_j}\Gamma(\nu_j+\frac12)\Gamma(1/2)}
\frac{(\lambda r')^{\nu_j-\frac{n-2}2}}{2^{\nu_j}\Gamma(\nu_j+\frac12)\Gamma(1/2)}\chi^{(\alpha_3)}(\lambda r)\chi^{(\alpha_4)}(\lambda r')
\\&\lesssim \lambda^{-\alpha}\sum_{j} \nu_j^{n-1}\frac{R^{2\nu_j-(n-2)}}{2^{2\nu_j}\Gamma(\nu_j+\frac12)\Gamma(\nu_j+\frac12)}\lesssim \lambda^{-\alpha}.
\end{split}
\end{equation}
On the other hand, using notation $\tilde{z}=(\tilde{r}, y)$, since $|\tilde{z}|\leq R$ and $|\tilde{z}'|\leq R$,
hence $d(\tilde{z},\tilde{z}')\leq C(R)$.  Then we have that for any $K>0$
\begin{equation}
\begin{split}
&\Big(\frac{d}{d\lambda}\Big)^{\alpha}\left(
(\lambda^2rr')^{-\frac{n-2}2}\sum_{j}\varphi_j(y)\overline{\varphi_j(y')}J_{\nu_j}(\lambda r)J_{\nu_j}(\lambda r')\chi(\lambda r)\chi(\lambda r')\right)
\\&\leq C(R) \lambda^{-\alpha}(1+d(\tilde{z},\tilde{z}'))^{-K}\leq \tilde{C}(R) \lambda^{-\alpha}(1+\lambda d(z,z'))^{-K}
\end{split}
\end{equation}
which has the property \eqref{beans} of $b(\lambda,z,z')$.
Now we are left to consider the remaining points which are away from
the cone tip.

{\bf $\bullet$ The region away from the cone tip.}  Recalling the definition of the asymptotically conic manifold in the introduction, the metric $g$, outside the compact set $\{(r,y)\in X: \frac1r=x\geq \epsilon \}$, is a scattering metric.
Therefore the method in \cite{HZ} works well for our situation here, although we only consider the spectral at a fixed energy $\lambda=1$ by using the scaling invariant. Indeed, in this region,
the operator $\LL_0=\Delta_g$ is a scattering operator introduced by Melrose \cite{Melrose}.
Without confusion, we use the notation $r$ and $x$ instead of $\tilde{r}$ and $\tilde{x}$ for adapting to Melrose's notation. Write the the operator $\Delta_g$ on $X$ as
$$\Delta_g=-\partial_r^2-\frac{n-1}r\partial_r+\frac{\Delta_Y}{r^2}.$$
With $x=1/r$, it gives
$$\Delta_g=-(x^2\partial_x)^2+(n-1)x^3\partial_x+x^2\Delta_Y$$
which is an elliptic scattering differential operator near $x=0$. Using the scattering symbol calculus defined in \cite{Melrose}, we have
the principal symbol of $\Delta_g$ is $|\mu|_h^2+\nu^2 $ where $(\nu,\mu)$ as
rescaled cotangent variables. In other word, if $(\xi,\eta)$ is the dual cotangent variables to $(x,y)$,
then $\nu=x^2\xi, \mu=x\eta$. Next we define the operator $Q_j, (j\geq 1)$ whose micro-support is relied to the characteristic set.

Define $O_1$ to be consist of points away from the cone tip, say $\tilde{x}< 2\epsilon$,  but away from the characteristic variety, that is,
satisfying $|\mu|_h^2+\nu^2 < 1/2$ or $|\mu|_h^2+\nu^2 > 3/2$. Finally divide the set $\{\tilde{ x} < 2\epsilon, |\mu|_h^2+\nu^2 \in [1/4, 2] \}$ into a finite number of sets
$O_2, \dots, O_{N}$ such that, for each set $O_j$ is equal to $\{ \tilde{x} < 2\epsilon, |\mu|_h^2+\nu^2 \in [1/4, 2]; |\nu-\nu_j|\leq \delta \}$,
where $\delta$ is taken sufficiently small such that we can separate the conjugated points
and $\nu_j$ is the finite points in this set, as we did in \cite{HZ}.

Therefore we have an open cover $O_0 \cup \dots \cup O_{N}$ of phase space. We then constructe
 a partition of unity subordinate to the above open cover, and take these as the principal symbols of
pseudo-differential operators $Q_j$ $(j\geq1)$ in the scattering pseudo-differential operator class $\Psi^0_k$ described in \cite{HZ} and are micro-locally supported in $O_j, (j\geq 1)$.
As mentioned above, the setting is a special case (i.e. $h$ independent of $r$) of \cite{HZ} in this region.
Hence by using \cite[Proposition 1.5] {HZ} at fixed energy $\lambda=1$, we have for $j\geq 1$
\begin{equation*}
\begin{split}
Q_j dE_{\sqrt{\LL_0}}(1;\tilde{r},y, \tilde{r}',y') Q_j^*= \sum_{\pm} e^{\pm
i d( (\tilde{r},y),(\tilde{r}',y'))}a_\pm(1;\tilde{r},y,\tilde{r}', y') +  b(1;\tilde{r},y,\tilde{r}', y').
\end{split}
\end{equation*}
From the proof of Proposition 1.5 in \cite[section 4]{HZ} and identities $\tilde{r}=\lambda r$ and $\tilde{r}'=\lambda r'$, we see that
$$a_\pm(1,\tilde{r},y,\tilde{r}', y')=a_\pm (\lambda; r, y, r', y), \quad b(1,\tilde{r},y,\tilde{r}', y')=b(\lambda; r, y, r', y).$$ On the other hand, the conic metric distant function
is homogeneous in such a way that $$d(\lambda r, y, \lambda r', y)=\lambda d(r, y, r', y).$$  This can be obtained from the explicit expression $d(z,z')=r+r'$ when
$d_{Y}(y,y') \geq \pi$ and when $d_{Y}(y,y') \leq \pi$
\begin{equation*}
\begin{split}
d(z,z')=d(r, y, r', y') &= \sqrt{ r^2 + {r'}^2 - 2r r' \cos
d_{Y}(y, y')} . \end{split}\end{equation*} Therefore we have \eqref{beanQ} when $j\geq 1$.  From \cite[Proposition 1.5] {HZ},
we obtain the analogue property \eqref{bean} and \eqref{beans}.\vspace{0.2cm}

Finally we verify \eqref{beanQ'}. Recall that $Q_0$ is constructed to micro-localizely supported in
$\{|\lambda r|\lesssim 1\}$ and that $Q_j$ for $j\geq 1$ is constructed to micro-localizely supported in $\{|\lambda r|\gg 1\}$. Without loss of generality, we only consider $j=0$ and $j'\geq 1$ in \eqref{beanQ'}. Set 
$$c_1(\lambda, z, z'):=(\lambda^2rr')^{-\frac{n-2}2}\sum_{\{j:\nu_j\gtrsim (\lambda r')\}}\varphi_j(y)\overline{\varphi_j(y')}J_{\nu_j}(\lambda r)J_{\nu_j}(\lambda r')\chi(\lambda r)(1-\chi(\lambda r')).$$
Indeed, we follow the similar approach as in the proof of \eqref{beans}, we have
\begin{equation*}
\begin{split}
&\Big|\Big(\frac{d}{d\lambda}\Big)^{\alpha}c_1(\lambda, z, z')\Big|
\\&\lesssim \lambda^{-\alpha}(\lambda^2rr')^{-\frac{n-2}2}\sum_{\alpha_1+\alpha_2+\alpha_3+\alpha_4\leq\alpha; }(\lambda r)^{\alpha_1+\alpha_3} (\lambda r')^{\alpha_2+\alpha_4}
\\&\quad \times \Big|\sum_{\{j:\nu_j\gtrsim \lambda r'\}}\varphi_j(y)\overline{\varphi_j(y')}J_{\nu_j+\alpha_1}(\lambda r)J_{\nu_j+\alpha_2}(\lambda r')\chi^{(\alpha_3)}(\lambda r)(1-\chi)^{(\alpha_4)}(\lambda r')\Big|
\end{split}
\end{equation*}
From Lemma \ref{lem:Bessel}, we know that $|J_\nu(r)|\leq Cr^{-1/3}, r\gg1$ with constant $C$ that is independent of $\nu$. 
Then, since $\lambda r\lesssim R\lesssim \lambda r'$ due to the support of $\chi$, it follows
\begin{equation*}
\begin{split}
&\Big|\Big(\frac{d}{d\lambda}\Big)^{\alpha}\Big(
(\lambda^2rr')^{-\frac{n-2}2}\sum_{\{j:\nu_j\gtrsim \lambda r'\}}\varphi_j(y)\overline{\varphi_j(y')}J_{\nu_j}(\lambda r)J_{\nu_j}(\lambda r')\chi(\lambda r)(1-\chi(\lambda r'))\Big)\Big|
\\&\lesssim \lambda^{-\alpha}\sum_{\{j:\nu_j\gtrsim \lambda r'\gtrsim R\gg1\}} \nu_j^{n-1}\frac{(\lambda r)^{\nu_j+2\alpha_1-\frac{n-2}2}}{2^{\nu_j}\Gamma(\nu_j+\frac12)\Gamma(1/2)}\chi^{(\alpha_3)}(\lambda r)
(\lambda r')^{\alpha_2-\frac13-\frac{n-2}2}(1-\chi)^{(\alpha_4)}(\lambda r')
\\&\lesssim \lambda^{-\alpha}\sum_{\{j:\nu_j\gtrsim \lambda r'\gtrsim R\gg1\}} \nu_j^{n-1}\frac{R^{\nu_j+2\alpha_1-(n-2)/2}\nu_j^{2K}}{2^{2\nu_j}\Gamma(\nu_j+\frac12)}(\lambda r')^{\alpha_2-\frac13-\frac{n-2}2-2K}\lesssim \lambda^{-\alpha}(1+\lambda d(z,z'))^{-K}.
\end{split}
\end{equation*}
Since $\lambda r'\sim \lambda d(z,z')$ when $\lambda r\lesssim R \ll \lambda r'$, this is proved by choosing $K$ large enough.

Finally we consider  the summation in the case that ${\{j: \nu_j\ll \lambda r'\}}$. 
According Lemma \ref{lem:Bessel}, we define
$$\tilde{a}_\pm(\lambda, z, z'):=(\lambda^2rr')^{-\frac{n-2}2}\sum_{\{j:\nu_j\ll (\lambda r')\}}\varphi_j(y)\overline{\varphi_j(y')}J_{\nu_j}(\lambda r)(\lambda r')^{-\frac12}a_\pm(\nu_j, \lambda r') \chi(\lambda r)(1-\chi(\lambda r')),$$
and 
$$c_2(\lambda, z, z'):=(\lambda^2rr')^{-\frac{n-2}2}\sum_{\{j:\nu_j\ll (\lambda r')\}}\varphi_j(y)\overline{\varphi_j(y')}J_{\nu_j}(\lambda r)E(\nu_j, \lambda r') \chi(\lambda r)(1-\chi(\lambda r')).$$
From Lemma \ref{lem:Bessel} and $\nu_0\geq \frac{n-2}2$, we use the similar argument as above to obtain
\begin{equation}
\begin{split}
&\Big|\Big(\frac{d}{d\lambda}\Big)^{\alpha}
\tilde{a}_\pm(\lambda, z, z')\Big|
\\&\lesssim \lambda^{-\alpha}\sup_{0\leq \alpha_1,\alpha_2\leq \alpha}\sum_{\{j:\nu_j\ll (\lambda r')\}} \nu_j^{n-2} \frac{(\lambda r)^{\nu_j+\alpha_1-\frac{n-2}2}}{2^{\nu_j}\Gamma(\nu_j+\frac12)\Gamma(1/2)}
(\lambda r')^{-\alpha_2} (\lambda r')^{-\frac{n-1}2}, \quad \\
&\lesssim \lambda^{-\alpha} (1+\lambda r_2)^{-\frac{n-1}2}\sim \lambda^{-\alpha} (1+\lambda d(z,z'))^{-\frac{n-1}2},
\end{split}
\end{equation}
which satisfies \eqref{bean}.
Similarly, for any $0\leq \alpha_1, \alpha_2\leq \alpha$, we obtain
\begin{equation}
\begin{split}
&\Big|\Big(\frac{d}{d\lambda}\Big)^{\alpha}
c_2(\lambda, z, z')\Big|
\\&\lesssim \lambda^{-\alpha}\sup_{0\leq \alpha_1,\alpha_2\leq \alpha}\sum_{\{j:\nu_j\ll (\lambda r')\}} \nu_j^{n-2} \frac{(\lambda r)^{\nu_j+\alpha_1-\frac{n-2}2}}{2^{\nu_j}\Gamma(\nu_j+\frac12)\Gamma(1/2)}
(\lambda r')^{-\alpha_2} (\lambda r')^{-\frac{n+1}2}, \quad \\
&\lesssim \lambda^{-\alpha} (1+\lambda r_2)^{-\frac{n+1}2}\sim \lambda^{-\alpha} (1+\lambda d(z,z'))^{-\frac{n+1}2},
\end{split}
\end{equation}
which gives \eqref{beanc}. 
This proves the property \eqref{beanc} of $c(\lambda,z,z'):=c_1(\lambda,z,z')+c_2(\lambda,z,z')$. Considering the discussions above all together, we conclude the proof of Proposition \ref{prop:localized spectral measure}.\vspace{0.2cm}

\subsection{Strichartz estimates for the case $V=0$}

Our next main result in this section  is the Strichartz estimates for the Schr\"odinger associated with $\LL_0$.

\begin{proposition}\label{Str-L0} Let $(q,\rr), (\tilde{q},\tilde{\rr})$ be admissible pairs satisfying \eqref{adm-p}, the following Strichartz estimates hold: the homogeneous inequality
\begin{equation}\label{h-Str-L0}
\|e^{it\LL_0} u_0\|_{L^q(\R;L^{\rr,2}(X))}\leq C\|u_0\|_{L^2(X)}
\end{equation}
and  inhomogeneous Strichartz estimates
\begin{equation}\label{ih-Str-L0}
\|\int_0^te^{i(t-s)\LL_0} F(s)ds\|_{L^q(\R;L^{\rr,2}(X))}\leq C\|F\|_{L^{\tilde{q}'}(\R;L^{\tilde{\rr}',2}(X))}.
\end{equation}

\end{proposition}

\begin{remark} We obtain the full set of global-in-time Strichartz estimates both in homogenous and inhomogeneous inequalities.
\end{remark}

\begin{proof} Having the Proposition \ref{prop:localized spectral measure} in hand, we can follow the argument in \cite{HZ}. For convenience, we sketch the
proof here. We first microlocalize (in phase space) Schr\"odinger propagators
$U_j(t)$ associated with $Q_j$ by
\begin{equation}\label{Uiti}
\begin{split}
U_j(t) = \int_0^\infty e^{it\lambda^2} Q_j(\lambda)
dE_{\sqrt{\LL_0}}(\lambda), \quad 0 \leq j \leq N,
\end{split}\end{equation}
where $Q_j(\lambda)$ is a partition of the identity operator in
Proposition \ref{prop:localized spectral measure}.
Then the operator $U_j(t) U_j(s)^*$
is given by
\begin{equation}
U_j(t) U_j(s)^* =  \int e^{i(t-s)\lambda^2}  Q_j(\lambda)
dE_{\sqrt{\LL_0}}(\lambda) Q_j(\lambda)^*.
\label{Uiti2}\end{equation}
Mimicing the argument of \cite[Proposition 5.1]{HZ}, we prove that there exists a constant $C$ such that
$$\| U_j(t) \|_{L^2 \to L^2}\leq C.$$
On the other hand, using the stationary phase argument as did in \cite[Proposition 6.1]{HZ}, we prove that the uniform boundedness of the dispersive estimate for $U_j(t) U_j(s)^*$ with norm $O(|t-s|^{-\frac{n}2})$.
Then by Keel-Tao's formalism \cite{KT}, for each $j$ we have
\begin{equation*}
\|U_j u_0\|_{L^{q}_t(\R;L^{\rr,2}(X))}\leq C\|u_0\|_{L^2(X)}.
\end{equation*}
As stated in \cite{KT}, we sharp the inequality to a Lorentz space norm $L^{\rr,2}(X)$.
Summing over $j$,  the homogeneous Strichartz estimates \eqref{h-Str-L0} for $e^{it\LL_0}$ finally are proved. The non-endpoint inhomogeneous Strichartz estimates follow
from the homogeneous estimates and the Christ-Kiselev lemma.
For the endpoint inhomogeneous estimate, we required additional argument \cite[section 8]{HZ} to obtain dispersive estimate on $U_j(t) U_{j'}(s)^*$ for $j' \neq j$ and the Keel-Tao's argument showed
the desirable endpoint inhomogeneous Strichartz estimates.  We sketch the proof in what follows.\vspace{0.2cm}

Even $Q_j$ with $j\geq1$ are micro-localized away from the cone tip, the relationships between $Q_j$ and $Q_{j'}$ are a bit more complicated than 
\cite{HZ,GH} due to the cone tip. This difference gives rise to a simplified version of the diffractive geometry in \cite{BW, FW, Wunsch} on Riemannian manifold with conic singularities.
The geodesics missing the cone point (i.e. boundary component) are same to the geodesics in \cite{HZ,GH}. The geodesics, which enter into the propagation of singularities on manifolds with cone points, 
are called ``geometric" and ``diffractive" bicharacteristic in \cite{BW, FW, Wunsch}. In our setting, we use the following definitions of \cite{BW} with one cone point.

\begin{definition} A \emph{diffractive geodesic} $\gamma$ on $X$ is a union of two closed, oriented geodesic segments $\gamma_1$ and $\gamma_2$ in $X$ such that $\gamma_1$ ends at the boundary component (blow up from cone point) at which $\gamma_2$ begins. A \emph{geometric geodesic} is a diffractive geodesic such that in addition, the final point of $\gamma_1$ and the initial point of $\gamma_2$ are connected by a geodesic of length $\pi$ in the boundary component.
\end{definition}

Therefore, as \cite[Lemma 8.2]{HZ}, we also similarly have: 
\begin{proposition}\label{prop:class} Let $Q_j$ with $j\geq1$ be in Proposition \ref{prop:localized spectral measure}.
We can divide $(j,j')$, $1 \leq j,j' \leq N$ into three classes
\begin{equation}\label{class}
\{ 1, \dots, N \}^2 = J_{near} \cup J_{not-out} \cup J_{not-inc},
\end{equation}
so that
\begin{itemize}
\item If $(j,j') \in J_{near}$, then $Q_j(\lambda) dE_{\sqrt{\LL_0}}Q_{j'}(\lambda)^*$ satisfies the conclusions of Proposition~\ref{prop:localized spectral measure};

\item If $(j,j') \in J_{non-inc}$, then $Q_j(\lambda)$ is not incoming-related to $Q_{j'}(\lambda)$ in the sense that no point in the operator wavefront set (microlocal support)
of $Q_j(\lambda)$ is related to a point in the operator wavefront
set of $Q_{j'}(\lambda)$ by backward bicharacteristic flow;

\item If $(j,j') \in J_{non-out}$, then $Q_j(\lambda)$ is not outgoing-related to $Q_{j'}(\lambda)$ in the sense that no point in the operator wavefront set  of
$Q_j(\lambda)$ is related to a point in the operator wavefront set
of $Q_{j'}(\lambda)$ by forward bicharacteristic flow.
\end{itemize}
\end{proposition}

\begin{proof} If $Q_j, Q_{j'}$ are not mircolocalized on the geodesics hitting the cone point,  by using \cite[Lemma 8.2]{HZ} (see also \cite[Lemmas 5.3 and 5.4]{GH}), we have done.
So we focus on the case that $Q_j, Q_{j'}$ are localized on the geodesics hitting the cone point (i.e through the boundary component), in which there exits broken geodesics. If
$Q_j, Q_{j'}$ both are localized on ``geometric" or ``diffractive" bicharacteristic $\gamma$ defined as above, we have the conclusions of Proposition~\ref{prop:localized spectral measure} due to that the distance 
$d(z,z')=r+r'$ when
$d_{Y}(y,y') \geq \pi$. If
$Q_j$ is localized on ``geometric" or ``diffractive" bicharacteristic $\gamma$ and $Q_{j'}$ is localized on diffractive geodesic $\gamma'$ where $\gamma\neq \gamma'$, then one can use the same argument of \cite[Lemmas 5.3 and 5.4]{GH} to show that either $(j,j') \in J_{non-inc}$ or $(j,j') \in J_{non-out}$. These terms are concerning the term $R_4$ of \cite[Theorem 5.1]{GHS1} with the oscillation factor $e^{\pm i\lambda(r+r')}$.
\end{proof}

Having these in mind, then we will have the following proposition.
\begin{proposition}\label{prop:Dispersive} Let $U_{j}(t)$ be the operator defined in \eqref{Uiti}.
Then there exists a constant $C$ independent of $t, z, z'$ for all
$j,j'\geq 0$ such that the following dispersive estimates on $U_{j}(t) U_{j'}(s)^*$ hold:

\begin{itemize}
\item If $(j,j') \in J_{near}$ or $(j,j')=(0,j'), (j,0)$, then for all $t \neq s$ we have
\begin{equation}
\big\|U_{j}(t)U^*_{j'}(s)\big\|_{L^1\rightarrow L^\infty}\leq C
|t-s|^{-n/2}, \label{UiUjnear}\end{equation}

\item If $(j,j')$ such that $Q_j$ is not outgoing related to
$Q_{j'}$, and $t<s$, then
\begin{equation}
\big\|U_{j}(t)U^*_{j'}(s)\big\|_{L^1\rightarrow L^\infty}\leq C
|t-s|^{-n/2},  \label{UiUj}\end{equation}

\item Similarly, if $(j,j')$ such that $Q_j$ is not incoming
related to $Q_{j'}$, and $s<t$, then
\begin{equation}
\big\|U_{j}(t)U^*_{j'}(s)\big\|_{L^1\rightarrow L^\infty}\leq C
|t-s|^{-n/2}. \label{UiUj2}\end{equation}
\end{itemize}

\end{proposition}

\begin{proof}
The results are proved in \cite[Lemma 8.6]{HZ} except for the cases $(j,j')=(0,j')$ or $(j,j')=(j,0)$ .
If $(j,j')=(0,j')$ or $(j,j')=(j,0)$, by Proposition \ref{prop:localized spectral measure} i.e. \eqref{beanQ'}, we can use the integration by parts  to
show \eqref{UiUjnear}. Indeed, we only need to consider the case $(j,j')=(0,j')$. For $j=0$, then we can use \eqref{beanQ'} 
\begin{equation*}
(Q_j(\lambda)dE_{\sqrt{\LL_0}}(\lambda)Q_{j'}^*(\lambda))(z,z')=\lambda^{n-1}\Big(\sum_{\pm} e^{\pm
i\lambda r'}\tilde{a}_\pm(\lambda; z,z') + c(\lambda; z, z') \Big),
\end{equation*}
to estimate 
\begin{equation*}
\begin{split}
&\Big|\int_0^\infty e^{it\lambda^2} \big(Q_0(\lambda)
dE_{\sqrt{\LL_0}}(\lambda)Q_{j'}^*(\lambda)\big)(z,z') d\lambda\Big|\\
&\leq \sum_\pm \Big|\int_0^\infty e^{it\lambda^2} e^{\pm
i\lambda r'}\tilde{a}_\pm(\lambda; z,z')  \lambda^{n-1}d\lambda\Big|+ \Big|\int_0^\infty e^{it\lambda^2} c(\lambda; z,z')  \lambda^{n-1}d\lambda\Big|.
\end{split}
\end{equation*}
Since the first term with $\tilde{a}_\pm (\lambda; z,z') $ has the same property of $a_\pm (\lambda; z,z')$, then we can use same argument of  \cite[Lemma 8.6]{HZ}
to prove \eqref{UiUjnear}.
For the second term with $c(\lambda; z,z') $, we use estimate \eqref{beanc} to obtain
\begin{equation}\label{beanc'}
\begin{split}
\Big|\big(\frac{d}{d\lambda}\big)^{N}c(\lambda,z,z')\Big|\leq
C_N\lambda^{n-1-N}\quad \forall N\in\mathbb{N}.
\end{split}
\end{equation}
Let $\phi\in
C_c^\infty([\frac12,2])$ such that $\sum_{k \in \Z}\phi(2^{-k}\lambda)=1$; we
denote $\phi_0(\lambda)=\sum_{k\leq -1}\phi(2^{-k}\lambda)$ and let $\delta$ be a small constant to be chosen later. Then the following inequality holds
\begin{equation*}
\Big|\int_0^\infty e^{it\lambda^2} c(\lambda; z,z') \lambda^{n-1} d\lambda\Big|\leq
C\int_0^\delta\lambda^{n-1}d\lambda\leq C\delta^n.
\end{equation*}
 We use  integration by parts $N$ times and use  \eqref{beanc'} to show
\begin{equation*}
\begin{split}
&\Big|\int_0^\infty e^{it\lambda^2}
\sum_{k\geq0}\phi(\frac{\lambda}{2^k\delta})c(\lambda; z,z') \lambda^{n-1} 
d\lambda\Big|\\
&\leq \sum_{k\geq 0}\Big|\int_0^\infty
\big(\frac1{2\lambda
t}\frac\partial{\partial\lambda}\big)^{N}\big(e^{it\lambda^2}\big)
\phi(\frac{\lambda}{2^k\delta})c(\lambda; z,z') \lambda^{n-1} 
d\lambda\Big|\\& \leq
C_N|t|^{-N}\sum_{k\geq0}\int_{2^{k-1}\delta}^{2^{k+1}\delta}\lambda^{n-1-2N}d\lambda\leq
C_N|t|^{-N}\delta^{n-2N}.
\end{split}
\end{equation*}
Choosing $\delta=|t|^{-\frac12}$, we have thus proved
\begin{equation}\label{4.3}
\begin{split}
&\Big|\int_0^\infty e^{it\lambda^2} c(\lambda; z,z') \lambda^{n-1} 
d\lambda\Big|\leq C_N|t|^{-\frac n2}.
\end{split}
\end{equation}

This completes the proof.

\end{proof}
Hence using the above dispersive estimate \eqref{UiUjnear}-\eqref{UiUj2} and the Keel-Tao's argument \cite{KT}, we show
that there exists a constant $C$ such that for each pair  $(j,j') \in J_{near}$ or $(j,j')=(0,j'), (j,0)$ or $(j,j') \in J_{non-inc}$ we have
\begin{equation}\label{bilinear:s<t}
\iint_{s<t}\langle U_j(t)U_{j'}^*(s)F(s), G(t)\rangle_{L^2}~ dsdt\leq C
\|F\|_{L^2_tL^{\frac{2n}{n+2},2}_z}\|G\|_{L^2_tL^{\frac{2n}{n+2},2}_z},
\end{equation}
and for $(j,j') \in J_{non-out}$ we have
\begin{equation}\label{bilinear:s>t}
\iint_{s>t}\langle U_j(t)U_{j'}^*(s)F(s), G(t)\rangle_{L^2}~ dsdt\leq C
\|F\|_{L^2_tL^{\frac{2n}{n+2},2}_z}\|G\|_{L^2_tL^{\frac{2n}{n+2},2}_z}.
\end{equation}
On the other hand, we showed that for all $0\leq j\leq N$, we have
\begin{equation*}
\|U_j(t)u_0\|_{L^2_tL^\rr_z}\lesssim\|u_0\|_{L^2},
\end{equation*}
hence it follows by duality that for all $0\leq j,j'\leq N$, the following inequality holds
\begin{equation}\label{R2}
\iint_{\R^2}\langle U_j(t)U_{j'}^*(s)F(s), G(t)\rangle_{L^2}~ dsdt\leq
C \|F\|_{L^2_tL^{\frac{2n}{n+2},2}_z}\|G\|_{L^2_tL^{\frac{2n}{n+2},2}_z}.
\end{equation}
Subtracting \eqref{bilinear:s>t} from \eqref{R2}  shows that
\eqref{bilinear:s<t} holds for each pair $(j,j') \in J_{non-out}$.  Therefore we show that
\eqref{bilinear:s<t} holds for every pair $(j,j') \in \{0,1,\cdots,N\}^2$. Now by summing
over all $j$ and $j'$, we obtain
\begin{equation}
\iint_{s<t}\langle U(t)U^*(s)F(s), G(t)\rangle_{L^2}~ dsdt\leq C
\|F\|_{L^2_tL^{\frac{2n}{n+2},2}_z}\|G\|_{L^2_tL^{\frac{2n}{n+2},2}_z}
\end{equation}
which is equivalent to \eqref{ih-Str-L0}.

\end{proof}

\section{Local-smoothing and Strichartz estimates}

In this section, we prove a local-smoothing estimate and then obtain the Strichartz estimates by using the
perturbation argument in \cite{JSS,RS}.
Note that we directly prove the local smoothing by avoiding the resolvent estimate of $\LL_V$.

\begin{proposition}[Local-smoothing]\label{prop:local-s} Let $u$ be the solution of \eqref{equ:S} and $z=(r,y)\in X$, then there exists a constant $C$
independent of $u_0$ such that
\begin{equation}\label{local-s}
\begin{split}
\|r^{-\beta}\partial_t^\alpha \mathcal{L}_V^{\frac s2}u(t,z)\|_{L^2_t(\R;L^2(X))}\leq
C\|u_0\|_{\dot H^{2\alpha+s+\beta-1}(X)}
\end{split}
\end{equation}
where $\alpha,s\in \R$ and $1/2<\beta<1+\nu_0$ with $\nu_0$ being the positive square root of the smallest eigenvalue of $\Delta_h+V_0(y)+(n-2)^2/4$.
\end{proposition}
\begin{proof}Recall that $dv=r^{n-1}dr dh(y)$. By the Plancherel theorem with respect to time $t$, it suffices to estimate
\begin{equation}\label{local-s-aim}
\begin{split}
&\int_{\R} \int_{X}|
\partial_t^\alpha \mathcal{L}_V^{\frac s2} u(t,z)|^2\frac{dtdv}{r^{2\beta}}=\int_{X}\int_{\R}\big|\tau^\alpha \mathcal{L}_V^{\frac s2} \hat{u}(z,\tau)\big|^2\frac{d\tau dv}{r^{2\beta}}.
\end{split}
\end{equation}
Recall that $$
u_0(z)=\sum_{\nu\in\chi_\infty}\sum_{\ell=1}^{d(\nu)}a_{\nu,\ell}(r)\varphi_{\nu,\ell}(y), \quad b_{\nu,\ell}(\rho)=(\mathcal{H}_{\nu}a_{\nu,\ell})(\rho),
$$ we have by \eqref{funct} with $F(\rho^2)=\rho^s e^{it\rho^2}$
\begin{equation}
\begin{split}
\mathcal{L}_V^{\frac s2} \hat{u}(\tau)&=\int_{\R} e^{-it\tau} \mathcal{L}_V^{\frac s2} u(z,t)dt
\\&=\sum_{\nu\in\chi_\infty}\sum_{\ell=1}^{d(\nu)}\varphi_{\nu,\ell}(y)\int_{\R}\int_0^\infty(r\rho)^{-\frac{n-2}2}J_{\nu}(r\rho)e^{
it(\rho^2-\tau)}b_{\nu,\ell}(\rho)\rho^{s}\rho^{n-1}d\rho dt.
\end{split}
\end{equation}
Put this into \eqref{local-s-aim}, we need to estimate
\begin{equation}\label{local-s-aim1}
\begin{split}
&\int_{X}\int_{\R}\big|\sum_{\nu\in\chi_\infty}\sum_{\ell=1}^{d(\nu)}\varphi_{\nu,\ell}(y)\int_{\R}\int_0^\infty(r\rho)^{-\frac{n-2}2}
J_{\nu}(r\rho)e^{
it(\rho^2-\tau)}b_{\nu,\ell}(\rho)\tau^{\alpha}\rho^s\rho^{n-1}d\rho dt\big|^2\frac{d\tau dv}{r^{2\beta}}\\
&=\int_{X}\int_{\R}\big|\sum_{\nu\in\chi_\infty}\sum_{\ell=1}^{d(\nu)}\varphi_{\nu,\ell}(y)\int_0^\infty(r\rho)^{-\frac{n-2}2}J_{\nu}
(r\rho)\delta(\tau-\rho^2)
b_{\nu,\ell}(\rho)\tau^\alpha\rho^s\rho^{n-1}d\rho\big|^2\frac{d\tau dv}{r^{2\beta}}.
\end{split}
\end{equation}
Using the delta function definition and changing the role of $\rho$ and $\tau$, the above estimation is reduced to estimate the following integral
\begin{equation}\label{local-s-aim2}
\begin{split}
&\int_{X}\int_0^\infty\big|\sum_{\nu\in\chi_\infty}\sum_{\ell=1}^{d(\nu)}\varphi_{\nu,\ell}(y)(r\sqrt{\rho})^{-\frac{n-2}2}J_{\nu}
(r\sqrt{\rho})b_{\nu,\ell}(\sqrt{\rho})
\rho^{\frac{n-1}2}\rho^{\alpha+\frac s2}\rho^{-\frac12}\big|^2\frac{d\rho dv}{r^{2\beta}}
\\&=\int_{X}\int_0^\infty\big|\sum_{\nu\in\chi_\infty}\sum_{\ell=1}^{d(\nu)}\varphi_{\nu,\ell}(y)(r\rho)^{-\frac{n-2}2}J_{\nu}(r\rho)
b_{\nu,\ell}(\rho)
\rho^{2\alpha+s}\rho^{n-2}\big|^2\frac{\rho d\rho dv}{r^{2\beta}}
\end{split}
\end{equation}

By the orthogonality, since
\begin{equation*}
\begin{split}
\int_{Y} \big|\sum_{\nu\in\chi_\infty}\sum_{\ell=1}^{d(\nu)}\varphi_{\nu,\ell}(y)J_{\nu}(r\rho)b_{\nu,\ell}(\rho)
\big|^2 dy=\sum_{\nu\in\chi_\infty}\sum_{\ell=1}^{d(\nu)}\big|J_{\nu}(r\rho)b_{\nu,\ell}(\rho)
\big|^2
\end{split}
\end{equation*}

we see that \eqref{local-s-aim2} equals
\begin{equation}\label{local-s-aim3}
\begin{split}
\sum_{\nu\in\chi_\infty}\sum_{\ell=1}^{d(\nu)}\int_0^\infty\int_0^\infty\big|(r\rho)^{-\frac{n-2}2}J_{\nu}(r\rho)b_{\nu,\ell}(\rho)\rho^{2\alpha+s}
\rho^{n-2}\big|^2\rho d\rho r^{n-1-2\beta}dr.
\end{split}
\end{equation}
To estimate it, we make a dyadic decomposition into the integral. Let $\chi$ be a smoothing function supported in $[1,2]$, we see that \eqref{local-s-aim3} is less than
\begin{equation}\label{scal-reduce}
\begin{split}
&\sum_{\nu\in\chi_\infty}\sum_{\ell=1}^{d(\nu)}\sum_{M\in2^{\Z}}\int_0^\infty\int_0^\infty\big|(r\rho)^{-\frac{n-2}2}J_{\nu}(r\rho)b_{\nu,\ell}(\rho)
\rho^{2\alpha+s}\rho^{n-2}\chi(\frac{\rho}{M})\big|^2\rho d\rho r^{n-1-2\beta}dr\\&\lesssim
\sum_{\nu\in\chi_\infty}\sum_{\ell=1}^{d(\nu)}\sum_{M\in2^{\Z}}\sum_{R\in2^{\Z}}M^{n+4\alpha+2s+2\beta-2}R^{n-1-2\beta}\int_{R}^{2R}
\int_{0}^\infty\big|(r\rho)^{-\frac{n-2}2}J_{\nu}(r\rho)b_{\nu,\ell}(M\rho)\chi(\rho)
\big|^2d\rho dr.
\end{split}
\end{equation}
Define
\begin{equation}\label{def:G}
\begin{split}
G_{\nu,\ell}(R,M)=\int_{R}^{2R}\int_{0}^\infty\big|(r\rho)^{-\frac{n-2}2}J_{\nu}(r\rho)b_{\nu,\ell}(M\rho)\chi(\rho)
\big|^2d\rho dr.\end{split}
\end{equation}
\begin{proposition}\label{prop:G} We have the following inequality
\begin{equation}\label{est:G}
G_{\nu,\ell}(R,M) \lesssim
\begin{cases}
R^{2\nu-n+3}M^{-n}\|b_{\nu,\ell}(\rho)\chi(\frac{\rho}M)\rho^{\frac{n-1}2}\|^2_{L^2},~
R\lesssim 1;\\
R^{-(n-2)}M^{-n}\|b_{\nu,\ell}(\rho)\chi(\frac{\rho}M)\rho^{\frac{n-1}2}\|^2_{L^2},~
R\gg1.
\end{cases}
\end{equation}
\end{proposition}
\begin{proof} We prove \eqref{est:G} by considering the following two
cases:\vspace{0.2cm}

$\bullet$ Case 1: If $R\lesssim1$, then since $\rho\sim1$, so
$r\rho\lesssim1$. By the property of the Bessel function
\eqref{bessel-r}, we obtain
\begin{equation}\label{3.13}
\begin{split}
G_{\nu,\ell}(R,M)&\lesssim\int_{R}^{2R}\int_{0}^\infty\Big| \frac{
(r\rho)^{\nu}(r\rho)^{-\frac{n-2}2}}{2^{\nu}\Gamma(\nu+\frac12)\Gamma(\frac12)}b_{\nu,\ell}(M\rho)\chi(\rho)\Big|^2d\rho dr\\& \lesssim
R^{2\nu-n+3}M^{-n}\|b_{\nu,\ell}(\rho)\chi(\frac{\rho}M)\rho^{\frac{n-1}2}\|^2_{L^2}.
\end{split}
\end{equation}

$\bullet$ Case 2: If $R\gg1$, then since $\rho\sim1$, so $r\rho\gg 1$. We
estimate by Lemma \ref{lem: J} on Bessel function
\begin{equation}\label{3.14}
\begin{split}
G_{\nu,\ell}(R,M)&\lesssim
R^{-(n-2)}\int_{0}^\infty\big|b_{\nu,\ell}(M\rho)\chi(\rho)\big|^2\int_{R}^{2R}\big|J_{\nu}(
r\rho)\big|^2 drd\rho\\& \lesssim R^{-(n-2)}\int_{0}^\infty\big|b_{\nu,\ell}(M\rho)\chi(\rho)\big|^2d\rho\lesssim
R^{-(n-2)}M^{-n}\|b_{\nu,\ell}(\rho)\chi(\frac{\rho}M)\rho^{\frac{n-1}2}\|^2_{L^2}.
\end{split}
\end{equation}
Thus we prove \eqref{est:G}.
\end{proof}
Now we turn our attention to estimate the following.
\begin{equation*}
\begin{split}
&\sum_{\nu\in\chi_\infty}\sum_{\ell=1}^{d(\nu)}\sum_{M\in2^{\Z}}\int_0^\infty\int_0^\infty\big|(r\rho)^{-\frac{n-2}2}J_{\nu}(r\rho)b_{\nu,\ell}(\rho)
\rho^{2\alpha+s}\rho^{n-2}\chi(\frac{\rho}{M})\big|^2\rho d\rho r^{n-1-2\beta}dr\\&\lesssim
\sum_{\nu\in\chi_\infty}\sum_{\ell=1}^{d(\nu)}\sum_{M\in2^{\Z}}\sum_{R\in2^{\Z}}M^{n+4\alpha+2s+2\beta-2}R^{n-1-2\beta}G_{\nu,\ell}(R,M)\\&\lesssim
\sum_{\nu\in\chi_\infty}\sum_{\ell=1}^{d(\nu)}\sum_{M\in2^{\Z}}\Big(\sum_{R\in2^{\Z}, R\lesssim 1}M^{n+4\alpha+2s+2\beta-2}R^{n-1-2\beta}R^{2\nu-n+3}M^{-n}
\\&\quad\qquad+\sum_{R\in2^{\Z}, R\gg 1}M^{n+4\alpha+2s+2\beta-2}R^{n-1-2\beta}R^{-(n-2)}M^{-n}\Big)\|b_{\nu,\ell}(\rho)\chi(\frac{\rho}M)\rho^{\frac{n-1}2}\|^2_{L^2}
\\&\lesssim
\sum_{\nu\in\chi_\infty}\sum_{\ell=1}^{d(\nu)}\sum_{M\in2^{\Z}}\Big(\sum_{R\in2^{\Z}, R\lesssim 1}M^{4\alpha+2s+2\beta-2}R^{2+2\nu-2\beta}
\\&\quad\qquad+\sum_{R\in2^{\Z}, R\gg 1}M^{4\alpha+2s+2\beta-2}R^{1-2\beta}\Big)\|b_{\nu,\ell}(\rho)\chi(\frac{\rho}M)\rho^{\frac{n-1}2}\|^2_{L^2}.
\end{split}
\end{equation*}
Note that if $1/2<\beta<1+\nu $, then the summations in $R$ converges  and further converges to $\|u_0\|^2_{\dot H^{2\alpha+s+\beta-1}(X)}$. This completes the proof of \eqref{local-s}.
\end{proof}

Note $\nu_0>0$, we can choose $\alpha=s=0$ and $\beta=1$ to obtain
\begin{corollary}\label{cor:loc} Let $u$ be the solution of \eqref{equ:S}, then there exists a constant $C$
independent of $u_0$ such that
\begin{equation}\label{local-s1}
\begin{split}
\|r^{-1}u(t,z)\|_{L^2_t(\R;L^2(X))}\leq
C\|u_0\|_{L^2(X)}.
\end{split}
\end{equation}
\end{corollary}
This corollary is enough for our purpose of this paper,  we state the following additional results due to the independent interest of local smoothing.
We find that the above argument only gives the result with a small $\epsilon>0$ loss of regularity
\begin{equation*}
\begin{split}
\|r^{-(1/2+\epsilon)}\mathcal{L}_V^{\frac14}u(t,z)\|_{L^2_t(\R;L^2(X))}\leq
C\|u_0\|_{\dot H^{\epsilon}(X)}.
\end{split}
\end{equation*}
However if we replace the weight $r^{-\beta}$ by $\beta(r)$ where $\beta$ is a smooth function compact supported, we obtain
\begin{corollary} Let $u$ be the solution of \eqref{equ:S} and let $\beta\in C_c^\infty([0, 1])$, then there exists a constant $C$
independent of $u_0$ such that
\begin{equation}\label{local-s2}
\begin{split}
\|\beta(r)\mathcal{L}_V^{\frac14}u(t,z)\|_{L^2_t(\R;L^2(X))}\leq
C\|u_0\|_{L^2(X)}.
\end{split}
\end{equation}
\end{corollary}

\begin{proof} We use the above argument with $\alpha=0$ and $s=1/2$ and replace the weight $r^{-\beta}$ by $\beta(r)$, we will only need to sum in $R\lesssim M$.
We modify the argument to obtain
\begin{equation*}
\begin{split}
&\sum_{\nu\in\chi_\infty}\sum_{\ell=1}^{d(\nu)}\sum_{M\in2^{\Z}}\int_0^\infty\int_0^\infty\big|(r\rho)^{-\frac{n-2}2}J_{\nu}(r\rho)b_{\nu,\ell}(\rho)
\rho^{1/2}\rho^{n-2}\chi(\frac{\rho}{M})\big|^2\rho d\rho r^{n-1}\beta^2(r)dr\\&\lesssim
\sum_{\nu\in\chi_\infty}\sum_{\ell=1}^{d(\nu)}\sum_{M\in2^{\Z}}\Big(\sum_{R\in2^{\Z}, R\lesssim \min\{1,M\}}M^{n-1}R^{n-1}R^{2\nu-n+3}M^{-n}
\\&\quad\qquad+\sum_{R\in2^{\Z}, 1\ll R\lesssim M}M^{n-1}R^{n-1}R^{-(n-2)}M^{-n}\Big)\|b_{\nu,\ell}(\rho)\chi(\frac{\rho}M)\rho^{\frac{n-1}2}\|^2_{L^2}
\\&\lesssim
\sum_{\nu\in\chi_\infty}\sum_{\ell=1}^{d(\nu)}\sum_{M\in2^{\Z}}\|b_{\nu,\ell}(\rho)\chi(\frac{\rho}M)\rho^{\frac{n-1}2}\|^2_{L^2}=\|u_0\|^2_{L^2(X)}.
\end{split}
\end{equation*}
\end{proof}

Now we prove Theorem \ref{thm:Strichartz}. By Duhamel formula and Proposition \ref{Str-L0}, for $r\geq2$ we have
\begin{equation}
\begin{split}
&\|e^{it\mathcal{L}_V}u_0\|_{L^q_tL^\rr_z}\lesssim \|e^{it\mathcal{L}_V}u_0\|_{L^q_tL^{\rr,2}_z}\\&\lesssim
\|e^{it\Delta_g}u_0\|_{L^q_tL^{\rr,2}_z}+\|\int_0^t
e^{i(t-s)\Delta_g}V(z)e^{is\mathcal{L}_V}u_0 ds\|_{L^q_tL^{\rr,2}_z}
\\&\lesssim
\|u_0\|_{L^2_z}+\|V(z)e^{is\mathcal{L}_V}u_0\|_{L^2_tL^{\frac{2n}{n+2},2}_z}
\\&\lesssim \|u_0\|_{L^2_z}+\|r
V(z)\|_{L^{n,\infty}}\| r^{-1}e^{is\mathcal{L}_V}u_0\|_{L^{2}_sL^{2}_z}.
\end{split}
\end{equation}
Note that $\|r
V(z)\|_{L^{n,\infty}}\leq C \|V_0\|_{L^\infty(Y)}$. By Corollary \ref{cor:loc}, we have the global-in-time local smoothing
\begin{equation}
\begin{split}
\|r^{-1} e^{it\mathcal{L}_V}u_0\|_{L^{2}_t(\R;L^{2}_z(X))}\leq
C\|u_0\|_{L^2}.
\end{split}
\end{equation}
Hence we prove the homogeneous Strichartz estimates \eqref{Str-est} for all admissible pair $(q,r)$.
By duality, the estimate is equivalent to
\begin{equation*}
\Big\|\int_{\R}e^{i(t-s)\LL_V}F(s)ds\Big\|_{L^q_tL^\rr_z}\lesssim\|F\|_{L^{\tilde{q}'}_tL^{\tilde{r}'}_z},
\end{equation*}
where both $(q,r)$ and $(\tilde{q},\tilde{r})$ satisfy \eqref{adm-p}.
By the Christ-Kiselev lemma \cite{CK}, for $q>\tilde{q}'$ we obtain
\begin{equation}
\Big\|\int_{s<t}e^{i(t-s)\LL_V}F(s) ds \Big\|_{L^q_tL^\rr_z}\lesssim\|F\|_{L^{\tilde{q}'}_tL^{\tilde{r}'}_z}.
\end{equation}
Notice that $\tilde{q}'\leq 2 \leq q$, therefore we have proved all inhomogeneous
Strichartz estimates except for the endpoint
$(q,r)=(\tilde{q},\tilde{r})=(2,\frac{2n}{n-2})$.

\section{The endpoint inhomogeneous Strichartz estimate}
In this section, we focus on the inhomogeneous
Strichartz estimate at the endpoint
$(q,r)=(\tilde{q},\tilde{r})=(2,\frac{2n}{n-2})$. We prove the endpoint Strichartz estimate by using the argument in \cite{D} and \cite{BM1}. But we need to prove
a resolvent estimate.  \vspace{0.2cm}

\subsection{A perturbation and iterated reduction} Recall $\mathcal{L}_V=\Delta_g+V$ and $\mathcal{L}_0=\Delta_g$, define the operators
\begin{equation}
\begin{split}
\mathcal{N}_0 F(t)=\int_0^t e^{i(t-s)\mathcal{L}_0} F(s) ds,\quad \mathcal{N} F(t)=\int_0^t e^{i(t-s)\mathcal{L}_V} F(s) ds.
\end{split}
\end{equation}
Set $u(t)=e^{i(t-s)\mathcal{L}_V} F(s)$, then we can write
$$u(t)=e^{i(t-s)\mathcal{L}_0} F(s)-i\int_s^t e^{i(t-\tau)\mathcal{L}_0} \left(Ve^{i(\tau-s)\mathcal{L}_V} F(s)\right) d\tau.$$
By integrating in $s\in[0,t]$ and using the Fubini Theorem, we have
\begin{equation*}
\begin{split}
 \mathcal{N} F(t)&=\int_0^t e^{i(t-s)\mathcal{L}_V} F(s) ds\\&=\int_0^t e^{i(t-s)\mathcal{L}_0} F(s) ds-i\int_0^t \int_s^t e^{i(t-\tau)\mathcal{L}_0} \left(Ve^{i(\tau-s)\mathcal{L}_V} F(s)\right) d\tau ds
 \\&=\mathcal{N}_0 F(t)-i\int_0^t \int_0^\tau e^{i(t-\tau)\mathcal{L}_0} \left(Ve^{i(\tau-s)\mathcal{L}_V} F(s)\right)  ds d\tau
 \\&=\mathcal{N}_0 F(t)-i\int_0^t e^{i(t-\tau)\mathcal{L}_0} \left( V \int_0^\tau e^{i(\tau-s)\mathcal{L}_V} F(s) ds \right)  d\tau
  \\&=\mathcal{N}_0 F(t)-i\mathcal{N}_0 \left( V (\mathcal{N} F) \right)(t)
\end{split}
\end{equation*}
and thus \begin{equation}\label{NF}
\begin{split}
 \mathcal{N} F(t)=\mathcal{N}_0 F(t)-i\mathcal{N}_0 \left( V (\mathcal{N} F) \right)(t).
\end{split}
\end{equation}
On the other hand, by similar argument we have
$$\mathcal{N}_0 F(t)=\mathcal{N} F(t)+i\mathcal{N} \left( V (\mathcal{N}_0 F) \right)(t),$$ hence $$\mathcal{N} F(t)=\mathcal{N}_0 F-i\mathcal{N} \left( V (\mathcal{N}_0 F) \right)(t).$$
Plugging it into \eqref{NF},  we obtain
\begin{equation*}
\begin{split}
 \mathcal{N} F=\mathcal{N}_0 F-i\mathcal{N}_0 \left( V (\mathcal{N}_0 F) \right)-\mathcal{N}_0 \left( V (\mathcal{N} (V (\mathcal{N}_0 F))) \right),
\end{split}
\end{equation*}
that is
\begin{equation}
\begin{split}
 \mathcal{N} F=\mathcal{N}_0 F-i\left( \mathcal{N}_0 V \mathcal{N}_0  \right)F-\mathcal{N}_0 (V \mathcal{N} V) \mathcal{N}_0 F.
\end{split}
\end{equation}
From now on, let $(q,r)=(\tilde{q},\tilde{r})=(2,\frac{2n}{n-2})$.
Hence we need to estimate
\begin{equation*}
\begin{split}
&\| \mathcal{N} F\|_{L^q(\R;L^{\rr,2})}\\&
\leq \|\mathcal{N}_0 F\|_{L^q(\R;L^{\rr,2})}+\|\left(\mathcal{N}_0 V \mathcal{N}_0  \right)F\|_{L^q(\R;L^{\rr,2})}+\|\mathcal{N}_0 (V \mathcal{N} V) \mathcal{N}_0 F\|_{L^q(\R;L^{\rr,2})}.
\end{split}
\end{equation*}
By the inhomogeneous Strichartz estimate in  Proposition \ref{Str-L0},  we have
\begin{equation}\label{N-0}
\begin{split}
\|\mathcal{N}_0 F\|_{L^q(\R;L^{\rr,2})}\lesssim \|F\|_{L^{\tilde{q}'}(\R;L^{\tilde{r}',2})}.
\end{split}
\end{equation}
Since $V=r^{-2}V_0(y)$ and $V_0\in \mathcal{C}^\infty(Y)$, then $V\in L^{\frac n2,\infty}$. Thus from the Strichartz estimate in Proposition \ref{Str-L0} again, we obtain
\begin{equation}\label{N-1}
\begin{split}
&\|\left(\mathcal{N}_0 V \mathcal{N}_0  \right)F\|_{L^q(\R;L^{\rr,2})}\\&\lesssim \|V \mathcal{N}_0 F\|_{L^2(\R;L^{\frac{2n}{n+2},2})}\lesssim \|\mathcal{N}_0 F\|_{L^2(\R;L^{\frac{2n}{n-2},2})}\lesssim \|F\|_{L^{\tilde{q}'}(\R;L^{\tilde{r}',2})}.
\end{split}
\end{equation}
Using the property of $V$ twice we obtain $r V\in L^{n, \infty}$ and $r^2 V\in L^{\infty}$. Similarly to the above argument, we prove
\begin{equation}\label{N-2}
\begin{split}
&\|\mathcal{N}_0 (V \mathcal{N} V) \mathcal{N}_0 F\|_{L^q(\R;L^{\rr,2})}
\\&\lesssim \| (V \mathcal{N} V) \mathcal{N}_0 F\|_{L^2(\R;L^{\frac{2n}{n+2},2})}\lesssim \| (r^{-1}\mathcal{N} r^{-1}) (r^2 V)r^{-1} \mathcal{N}_0 F\|_{L^2(\R;L^{2})}\\&\lesssim \|r^{-1} \mathcal{N}_0 F\|_{L^2(\R;L^{2})}
\lesssim \|\mathcal{N}_0 F\|_{L^2(\R;L^{\frac{2n}{n-2},2})}\lesssim \|F\|_{L^{\tilde{q}'}(\R;L^{\tilde{r}',2})}
\end{split}
\end{equation}
where we use the Strichartz estimate \eqref{N-0} in the first inequality and the  H\"older inequality in Lorentz space for the second inequality.
Here we also use the following Lemma \ref{lem:in-local-sm} about the local smooth estimate for the third inequality and use \eqref{N-0} for the last inequality.
\begin{lemma}\label{lem:in-local-sm} Let $\mathcal{L}_V$ be as above, then we have
\begin{equation}\label{in-local-sm}
\|r^{-1}\int_0^t e^{i(t-s)\mathcal{L}_V}r^{-1} F ds\|_{L^2(\R;L^{2})}\leq C\|F\|_{L^2(\R;L^{2})}.
\end{equation}
\end{lemma}
\begin{proof} This is a consequence of   D'Ancona's result \cite[Theorem 2.3]{D} and the resolvent estimate \eqref{resolvent} given below with $\alpha=1$.
\end{proof}

\subsection{A resolvent estimate} In this subsection, we show a weighted resolvent estimate.

\begin{proposition}\label{prop:resolvent}
Let $\mathcal{L}_V$ and $\nu_0$ be as above. We have the  weighted resolvent estimate
\begin{equation}\label{resolvent}
\sup_{\sigma\notin \R^+} \|r^{-\alpha}(\LL_V-\sigma)^{-1}r^{-2+\alpha}f\|_{L^{2}(X)}\leq C\|f\|_{L^{2}(X)}.
\end{equation}
where $\alpha\in R_{\nu_0}$ defined by
 \begin{equation}\label{Rv} R_{\nu_0}=
 \begin{cases} (\frac12,\frac32), \quad \nu_0>1/2;\\
 (1-\frac{\nu_0^2}{1-2\nu_0^2}, 1+\frac{\nu_0^2}{1-2\nu_0^2}),\quad 0<\nu_0\leq1/2.
 \end{cases}
 \end{equation}
\end{proposition}

\begin{remark} This is a generalization of \cite [Theorem 2.1]{BPST} in which they proved \eqref{resolvent} in the Euclidean space with $\alpha=1$.
The smallest eigenvalue plays an important role in \eqref{resolvent}.
\end{remark}

\begin{proof}The proof is inspirit of \cite{BPST}, but
some modifications and improvements are required.
To obtain the general result, we have to replace the multiplier $r e^{-2r\tau}\phi(r)\partial_r \bar{v}$ by $r^\beta e^{-2r\tau}\phi(r)\partial_r \bar{v}$
which brings much more harder treating terms
in  the weighted Hardy's inequality. By the duality, we only need to prove \eqref{resolvent} with  $R_{\nu_0}\ni \alpha\geq1$. Indeed, if we could prove
\begin{equation}
\|r^{-\alpha}(\LL_V-\sigma)^{-1}r^{-2+\alpha}\|_{L^{2}(X)\to L^2(X)}\leq C, \quad 1\leq \alpha<\alpha_0
\end{equation}
where $\alpha_0=3/2$ or $1+\frac{\nu_0^2}{1-2\nu_0^2}$, by taking the adjoint of this estimate and replacing $\sigma$ by $\bar{\sigma}$,
we also have
\begin{equation}
\|r^{-2+\alpha}(\LL_V-\sigma)^{-1}r^{-\alpha}\|_{L^{2}(X)\to L^2(X)}\leq C, \quad 1\leq \alpha<\alpha_0
\end{equation}
which shows that
\begin{equation}
\|r^{-\alpha'}(\LL_V-\sigma)^{-1}r^{-2+\alpha'}\|_{L^{2}(X)\to L^2(X)}\leq C, \quad 2-\alpha_0\leq \alpha'\leq1.
\end{equation}
where $\alpha'=2-\alpha$. So we only need to prove \eqref{resolvent} with $1\leq \alpha<\alpha_0$.\vspace{0.2cm}

Let $\z=\sqrt{-\sigma}$ with the branch such that $\mathrm{Re}~\z=\tau>0$, then for given $f\in L^2(X)$ and $\sigma\in \C\setminus \R^+$, we
consider the Helmholtz equation
\begin{equation}\label{helm-eq}
\LL_Vu+\z^2u=f.
\end{equation} By density argument, we can take $f\in \mathcal{C}_0^\infty(X)$. Then $u$ is a classical solution of  \eqref{helm-eq} and define
 $v(r,y): (0,\infty)\times Y\to \C$ by
 $$v(r,y)=r^{\frac{n-1}2} e^{r\z}u(r,y).$$
 Then we see that
 \begin{equation*}
 \begin{split}
 \partial_r v&=r^{\frac{n-1}2} e^{r\z}\left(\frac{n-1}{2r}u+\z u+\partial_r u\right)\\
-\partial_r^2 v&=r^{\frac{n-1}2} e^{r\z}\left(-\partial_r^2 u-2\Big(\frac{n-1}{2r}+\z\Big)\partial_r u-\Big(\frac{(n-1)(n-3)}{4r^2}+\frac{(n-1)\z}{r}+\z^2 \Big)u\right)\\
\z\partial_r v&=r^{\frac{n-1}2} e^{r\z}\Big( \z\partial_r u+\big(\frac{(n-1)\z}{2r}+\z^2 \big)u\Big)
\end{split}
\end{equation*}
therefore $v$ satisfies
 \begin{equation}\label{v}
 \begin{split}
&-\partial_r^2 v+2\z\partial_r v+\left(\frac{(n-1)(n-3)}{4}-\Delta_h+V_0(y)\right)\frac{v}{r^2}
\\&=r^{\frac{n-1}2} e^{r\z}\left(-\partial_r^2 u-\frac{n-1}{r}\partial_r u+\Big(-\Delta_h+V_0(y)+\z^2 \Big)u\right)
\\&=r^{\frac{n-1}2} e^{r\z}f.
\end{split}
\end{equation}
For fixed $M>m>0$, let $\phi=\phi_{m,M}(r)$ be a smooth cut-off function such that $0\leq\phi\leq1$ with being zero outside $[0, M+1]$ and equaling to $1$ on $[m,M]$.  By multiplying \eqref{v}
by $r^\beta e^{-2r\tau}\phi(r)\partial_r \bar{v}$ with $\beta$ being chosen later and taking the real part, we show that
 \begin{equation*}
 \begin{split}
&-\frac12 r^{\beta}e^{-2r\tau}\phi(r)\partial_r|\partial_r v|^2+2\tau r^{\beta} e^{-2r\tau}\phi(r)|\partial_r v|^2\\&+\frac{1}{2r^{2-\beta}}e^{-2r\tau}\phi(r) \left(\frac{(n-1)(n-3)}{4}+V_0(y)\right)\partial_r|v|^2
+\frac{1}{r^{2-\beta}}e^{-2r\tau}\phi(r) \mathrm{Re}(-\Delta_h v\partial_r\bar{v})\\& =r^{\frac{n-1}2+\beta}\phi(r) \mathrm{Re}\left(e^{r(\z-2\tau)}\partial_r \bar{v} f\right).
\end{split}
\end{equation*}
Integrating the above formula on $(0,\infty)\times Y$ but with volume $drdh$ and performing the integration by parts, we have
 \begin{equation*}
 \begin{split}
&\frac12 \int_0^\infty\int_{Y} \partial_r\left(r^{\beta}e^{-2r\tau}\phi(r)\right)|\partial_r v|^2 dr dh
\\&+2\tau \int_0^\infty\int_{Y} r^{\beta}e^{-2r\tau}\phi(r)|\partial_r v|^2 dr dh\\&-\frac{1}{2}\int_0^\infty\int_{Y}\partial_r\left(r^{-2+\beta} e^{-2r\tau}\phi(r)\right) \left(\frac{(n-1)(n-3)}{4}+V_0(y)\right)|v|^2 dr dh
\\ \quad &-\frac{1}{2}\int_0^\infty\int_{Y}\partial_r\left(r^{-2+\beta} e^{-2r\tau}\phi(r)\right)|\nabla_h v|^2dr dh
\\&=\int_0^\infty\int_{Y} r^{\frac{n-1}2+\beta}\phi(r) \mathrm{Re}\left(e^{r(\z-2\tau)}\partial_r \bar{v} f\right) dr dh.
\end{split}
\end{equation*}
Furthermore we compute that
 \begin{equation*}
 \begin{split}
&\frac12 \int_0^\infty\int_{Y} e^{-2r\tau}\phi(r)r^{\beta-1} \left(\beta- 2r\tau \right)|\partial_r v|^2 dr dh+\frac12 \int_0^\infty\int_{Y} r^{\beta}e^{-2r\tau}\phi'(r)|\partial_r v|^2 dr dh
\\&+2\tau \int_0^\infty\int_{Y} r^{\beta}e^{-2r\tau}\phi(r)|\partial_r v|^2 dr dh\\&+\frac{1}{2}\int_0^\infty\int_{Y}e^{-2r\tau}\phi(r) r^{-3+\beta}\left((2-\beta)+2r\tau \right) \left(\frac{(n-2)^2}{4}-\frac14+V_0(y)\right)|v|^2 dr dh
\\&-\frac12 \int_0^\infty\int_{Y} r^{-2+\beta}e^{-2r\tau}\phi'(r)\left(\frac{(n-2)^2}{4}-\frac14+V_0(y)\right)|v|^2 dr dh
\\ \quad &+\frac{1}{2}\int_0^\infty\int_{Y}e^{-2r\tau}\phi(r) r^{-3+\beta}\left((2-\beta)+2r\tau \right)  |\nabla_h v|^2 dr dh
\\&-\frac12 \int_0^\infty\int_{Y} r^{-2+\beta}e^{-2r\tau}\phi'(r)|\nabla_h v|^2 dr dh
\\&=\int_0^\infty\int_{Y} r^{\frac{n-1}2+\beta}\phi(r) \mathrm{Re}\left(e^{r(\z-2\tau)}\partial_r \bar{v} f\right) dr dh.
\end{split}
\end{equation*}
Therefore we show
 \begin{equation}
 \begin{split}
&\frac12 \int_0^\infty\int_{Y} e^{-2r\tau}\phi(r) r^{\beta-1}\left(\left(\beta+ 2r\tau \right)|\partial_r v|^2-\left((2-\beta)+2r\tau \right) \frac{|v|^2}{4r^2}\right)  dr dh
\\&+\frac{1}{2}\int_0^\infty\int_{Y}e^{-2r\tau}\phi(r) r^{-3+\beta}\left((2-\beta)+2r\tau \right)  \left(\left(\frac{(n-2)^2}{4}+V_0(y)\right)|v|^2+|\nabla_h v|^2\right) dr dh
\\&+\frac12 \int_0^\infty\int_{Y} r^{\beta}e^{-2r\tau}\phi'(r)\left(|\partial_r v|^2+\frac{1}{4r^2}|v|^2-\frac1{r^2}\big(|\nabla_h v|^2+(V_0(y)+\frac{(n-2)^2}4)|v|^2\big)\right) dr dh
\\&=\int_0^\infty\int_{Y} r^{\frac{n-1}2+\beta}\phi(r) \mathrm{Re}\left(e^{r(\z-2\tau)}\partial_r \bar{v} f\right) dr dh.
\end{split}
\end{equation}
On the other hand, since $-\Delta_h+V_0(y)+(n-2)^2/4$ is positive on ${Y}$ with the smallest eigenvalue $\nu_0^2>0$, that is,
\begin{equation}\label{lowb}
\int_{Y}\left(\left(\frac{(n-2)^2}{4}+V_0(y)\right)|v|^2+|\nabla_h v|^2\right) dh\geq \nu_0^2\int_{Y} |v(r,y)|^2 dh\geq 0.
\end{equation}
Hence we show that for all $\epsilon>0$
 \begin{equation}\label{est:v}
 \begin{split}
&\frac12 \int_0^\infty\int_{Y} e^{-2r\tau}\phi(r) r^{\beta-1}\left(\left(\beta+ 2r\tau \right)|\partial_r v|^2+\left((2-\beta)+2r\tau \right)(\nu_0^2-\frac14) \frac{|v|^2}{r^2}\right)  dr dh
\\&+\frac12 \int_0^\infty\int_{Y} r^\beta e^{-2r\tau}\phi'(r)\left(|\partial_r v|^2+\frac{1}{4r^2}|v|^2-\frac1{r^2}\big(|\nabla_h v|^2+(V_0(y)+\frac{(n-2)^2}4)|v|^2\big)\right) dr dh
\\&\leq\int_0^\infty\int_{Y} r^{\frac{n-1}2+\beta}\phi(r) \mathrm{Re}\left(e^{r(\z-2\tau)}\partial_r \bar{v} f\right) dr dh
\\&\leq \frac{1}{4\epsilon^2}\|r^{\frac{1+\beta}2}f\|_{L^2}^2+\epsilon^2\int_0^\infty\int_{Y} \phi(r) e^{-2r\tau}r^{\beta-1}|\partial_r v|^2  dr dh.
\end{split}
\end{equation}

For our purpose, we first need the following lemma.

\begin{lemma}\label{to0}Let $0\leq \beta\leq 1$, we have following estimate for $ m\to0, M\to \infty$
\begin{equation}
\int_0^\infty\int_{Y} r^\beta e^{-2r\tau}\phi'(r)\left(|\partial_r v|^2+\frac{1}{4r^2}|v|^2-\frac1{r^2}\big(|\nabla_h v|^2+(V_0(y)+\frac{(n-2)^2}4)|v|^2\big)\right) dr dh\geq 0.
\end{equation}
\end{lemma}
We postpone the proof to the next subsection.\vspace{0.2cm}

By  taking the limits $m\to 0$ and $M\to \infty$ and using Lemma \ref{to0} and \eqref{est:v}, we have
 \begin{equation*}
 \begin{split}
\frac12 \int_0^\infty\int_{Y} e^{-2r\tau} &r^{\beta-1}\Big(\left(\beta+ 2r\tau-2\epsilon^2 \right)|\partial_r v|^2
\\&+\big((2-\beta)+2r\tau \big)(\nu_0^2-\frac14) \frac{|v|^2}{r^2}\Big)  dr dh
\leq \frac{1}{4\epsilon^2}\|r^{\frac{1+\beta}2}f\|_{L^2}^2.
\end{split}
\end{equation*}
Furthermore  for $0<\beta\leq1$ we obtain
\begin{equation}\label{est:reduction}
 \begin{split}
(1&-\frac{2\epsilon^2}\beta) \int_0^\infty\int_{Y} e^{-2r\tau} r^{\beta-1}\left(\beta+ 2r\tau\right)|\partial_r v|^2dr dh
\\&+(\nu_0^2-\frac14)\int_0^\infty\int_{Y} e^{-2r\tau} r^{\beta-1}\big((2-\beta)+2r\tau \big) \frac{|v|^2}{r^2}  dr dh
\leq \frac{1}{2\epsilon^2}\|r^{\frac{1+\beta}2}f\|_{L^2}^2.
\end{split}
\end{equation}

{\bf Case 1: $\nu_0>1/2$.} Since $0<\beta\leq 1$ and $r\tau>0$, we have
\begin{equation}
 \begin{split}
\int_0^\infty\int_{Y} e^{-2r\tau} r^{\beta-1} \frac{|v|^2}{r^2}  dr dh
\leq C_{\nu_0}\|r^{\frac{1+\beta}2}f\|_{L^2}^2.
\end{split}
\end{equation}
which implies $$\|r^{-\alpha} u\|_{L^2(X)}\leq C_{\nu_0}\|r^{2-\alpha}f\|_{L^2(X)}, \quad \beta=3-2\alpha.$$
Note $0<\beta\leq1$, we have showed that if $\nu_0>1/2$ and $1\leq \alpha<3/2$.
\begin{equation}
\|r^{-\alpha}(\LL_V-\z^2)^{-1}r^{\alpha-2}\|_{L^2\to L^2}\leq C.
\end{equation}\vspace{0.2cm}

{\bf Case 2: $0<\nu_0\leq 1/2$.} In this case, we need a weighted Hardy's inequality. We have

\begin{lemma}[Weighted Hardy's inequality]\label{lem:w-hardy} Let $w\in \mathcal{C}^2(\R^+\setminus\{0\};\R)$ satisfy
\begin{equation}\label{equ:wcond}
w(r)\geq0, w'(r)\leq 0, \quad r(w'(r)^2+2w(r)w''(r))\geq 2w(r)w'(r), \forall r\geq0.
\end{equation}
Let $g:\R^+\to\C$ be such that
\begin{equation}\label{equ:funccond}
  \int_0^\infty \big(w^2(r)|g'|^2+(w'(r))^2|g|^2\big)\;dr<+\infty
\end{equation}
and
\begin{equation}\label{equ:fcond}
  \liminf_{r\to0}w(r)w'(r)|g(r)|^2=0.
\end{equation} Then
\begin{equation}
\int_0^\infty w^2\frac{|g(r)|^2}{r^2}dr\leq 4\int_0^\infty w^2|g'(r)|^2dr.
\end{equation}
\end{lemma}

Next we will use the above weighted Hardy's inequality to show

\begin{lemma}\label{lem:asmalllv}
Let $\max\big\{0,1-2\nu_0\big\}<\beta\leq1$, then we have
\begin{align}\label{equ:asmalllv}
 & \int_0^\infty\int_{Y} e^{-2r\tau} r^{\beta-1}\left(\beta+ 2r\tau\right)|\partial_r v|^2dr dh \\\nonumber
 \geq&\frac14 \int_0^\infty\int_{Y} e^{-2r\tau} r^{\beta-1}\big(\beta+2r\tau \big) \frac{|v|^2}{r^2}  dr dh.
\end{align}
\end{lemma}
We postpone the proof of the Lemmas \ref{lem:w-hardy} and \ref{lem:asmalllv} to the end of this section.

This together with \eqref{est:reduction} implies that
\begin{equation}\label{est:reduction'}
 \begin{split}
&C\frac{1}{2\epsilon^2}\|r^{\frac{1+\beta}2}f\|_{L^2}^2\\\geq& \frac14(1-\frac{2\epsilon^2}\beta) \int_0^\infty\int_{Y} e^{-2r\tau} r^{\beta-1}\left(\beta+ 2r\tau\right)\frac{|v|^2}{r^2} dr dh
\\&+(\nu_0^2-\frac14)\int_0^\infty\int_{Y} e^{-2r\tau} r^{\beta-1}\big((2-\beta)+2r\tau \big) \frac{|v|^2}{r^2}  dr dh\\
\geq& \frac\beta 4 (1-\frac{2\epsilon^2}\beta) \int_0^\infty\int_{Y} e^{-2r\tau} r^{\beta-1}\frac{|v(r)|^2}{r^2}\;dh\;dr
\\&+(2-\beta)(\nu_0^2-\frac14)\int_0^\infty\int_{Y} e^{-2r\tau} r^{\beta-1} \frac{|v|^2}{r^2}  dr dh
\end{split}
\end{equation}
this implies that $$\|r^{-\alpha} u\|_{L^2(X)}\leq C_{\nu_0}\|r^{2-\alpha}f\|_{L^2(X)}, \quad \beta=3-2\alpha$$
provided that
 $$\frac\beta 4>(2-\beta)(\frac14-\nu_0^2)\Leftrightarrow \beta>1-\frac{2\nu_0^2}{1-2\nu_0^2}.$$
 Therefore we show
\begin{equation}
\|r^{-\alpha}(\LL_V-\z^2)^{-1}r^{\alpha-2}\|_{L^2\to L^2}\leq C,\quad 1\leq \alpha<1+\frac{\nu_0^2}{1-2\nu_0^2}.
\end{equation}
Therefore we prove Proposition \ref{prop:resolvent} if we could prove Lemma \ref{to0}, Lemma \ref{lem:w-hardy}  and Lemma \ref{lem:asmalllv}.
\end{proof}
\vspace{0.2cm}

To complete the proof, we have to prove the lemmas which will be done in the rest of section. In the proof, we have to be careful with the factor associated with the index $\beta$.

 \begin{proof}[{\bf The proof Lemma \ref{to0}}]
Before proving Lemma \ref{to0}, we show the following lemmas.
\begin{lemma}\label{lem:uhelmequlv} Let $\sigma\notin\R^+$ and $f\in \mathcal{C}_0^\infty(X)$. Assume that $u$ is the classical solution to
\begin{equation}\label{helm-eqinblv}
\LL_V u-\sigma u=f.
\end{equation}
Then, $u\in \dot{H}^1(X)$ if $\sigma=0$ and $u\in {H}^1(X)$ if $\sigma\neq 0$, and
\begin{equation}\label{equ:nabluestlv}
\int_{X}\frac{|\partial_r u|^2}{|x|^{1-\beta}}\;dx+\int_{X}\frac{|u|^2}{|x|^{3-\beta}}\;dx<+\infty.
\end{equation}
Here $\beta>\max\{0,1-2\nu_0\}$
where $\nu_0$ is the positive square root of the smallest eigenvalue of the operator $-\Delta_h+V_0(y)+(n-2)^2/4$.

\end{lemma}

To prove Lemma \ref{lem:uhelmequlv}, we first show the modified Hardy inequality.
\begin{lemma}\label{lem:w-hardy'} Let $k\neq\frac{n}{2}$. There holds
\begin{equation}\label{equ:modhardylv}
 \int_{0}^\infty\int_{Y}\frac{|f|^2}{r^{2k}}\;r^{n-1}drdh\leq\frac{4}{(n-2k)^2} \int_{0}^\infty\int_{Y} \frac{|\partial_r f|^2}{r^{2k-2}}\;r^{n-1}drdh.
\end{equation}
\end{lemma}

\begin{proof}
First, by the sharp Hardy's inequality\cite{KSWW}, we have
  \begin{align}\label{equ:byharslv}
    \int_{0}^\infty \frac{|f(r,y)|^2}{r^{2k}}\;r^{n-1}dr\leq&\frac{4}{(n-2)^2}\int_{0}^\infty \Big|\partial_r\Big(\frac{f(r,y)}{r^{k-1}}\Big)\Big|^2\;r^{n-1}dr\\\nonumber
    =&\frac{4}{(n-2)^2}\int_{0}^\infty \frac{1}{r^{2k-2}}\big|\partial_r f-(k-1)f(r,y)r^{-1}\big|^2 r^{n-1}\;dr.
  \end{align}
Noting that
$$\big|\partial_r f-(k-1)r^{-1}f\big|^2=|\partial_r f|^2-2(k-1)r^{-1}\cdot{\rm Re}\big(\bar{f}\partial_r f\big)+(k-1)^2\frac{|f|^2}{r^2},$$
we get
\begin{align*}
  &\int_{0}^\infty \frac{1}{r^{2k-2}}\big|\partial_r f-(k-1)f(r,y)r^{-1}\big|^2 r^{n-1}\;dr\\
  =&\int_{0}^\infty \left(|\partial_r f|^2-2(k-1)r^{-1}\cdot{\rm Re}\big(\bar{f}\partial_r f\big)+(k-1)^2\frac{|f|^2}{r^2}\right)\;r^{n-1}dr\\
  =&\int_{0}^\infty \frac{|\partial_r f|^2}{r^{2k-2}}\;r^{n-1}dr+(k-1)(n-k-1)\int_{0}^\infty \frac{|f|^2}{r^{2k}}\;r^{n-1}dr.
\end{align*}
  Plugging this into \eqref{equ:byharslv} leads to the following inequality
 \begin{equation}
 \begin{split}
  &\int_{0}^\infty \frac{|f(r,y)|^2}{r^{2k}}\;r^{n-1}dr \\&\leq\frac{1}{(n-2)^2/4-(k-1)(n-k-1)}\int_{0}^\infty \frac{|\partial_r f|^2}{r^{2k-2}}\;r^{n-1}dr\\
  &=\frac{4}{(n-2k)^2}\int_{0}^\infty \frac{|\partial_r f|^2}{r^{2k-2}}\;r^{n-1}dr,
  \end{split}
  \end{equation}
 and so   \eqref{equ:modhardylv} follows. Now by integrating on ${Y}$, the proof of Lemma \ref{lem:w-hardy'} follows immediately.

\end{proof}

\begin{proof}[{\bf Proof of Lemma \ref{lem:uhelmequlv}:}]
  We first consider the case $\sigma=0$, that is, $u$ solves
\begin{equation}\label{equ:sigma0lv}
  -\Delta u+r^{-2}V_0(y)u=f.
\end{equation}
Multiplying the above equality by $\bar{u}$ and integrating on $X$, we obtain
$$\int_{X}\Big(|\nabla u|^2+V_0(y)r^{-2}|u|^2\Big)\;dx={\rm Re}\int_{X}f\bar{u}\;dx.$$
By Young's inequality and \cite[Proposition 1]{BPSS}, we have
\begin{align*}
 \|u\|_{\dot{H}^1}^2\sim\big\|\sqrt{\LL_V}u\big\|_{L^2}^2&=\int_{X}\Big(|\nabla u|^2+V_0(y)r^{-2}|u|^2\Big)\;dx\\& \leq  C(\epsilon)\int_{X} r^2 |f|^2\;dx+\epsilon \int_{X}|\nabla u|^2\;dx.
\end{align*}
Hence,
$$\int_{X}|\nabla u|^2\;dx\leq  C\int_{X}r^2 |f|^2\;dx,$$
and so $u\in\dot{H}^1$.

Next, multiplying \eqref{equ:sigma0lv} by $\frac{\bar{u}}{r^{1-\beta}}$, integrating in $X$ and taking real part, we obtain
\begin{align*}
\int_{0}^\infty \int_{Y} \frac{|\nabla u|^2}{r^{1-\beta}}\;r^{n-1}drdh-\frac12\int_{0}^\infty \int_{Y} |u|^2\Delta\big(r^{\beta-1}\big)\;r^{n-1}drdh\\ +\int_{0}^\infty \int_{Y}\frac{V_0(y)}{r^{3-\beta}}|u|^2\;r^{n-1}dr dh
={\rm Re}\int_{0}^\infty \int_{Y} \frac{f\bar{u}}{r^{1-\beta}}\;r^{n-1}drdh,
\end{align*}
which implies that
\begin{align*}
&\int_{0}^\infty \int_{Y} \frac{|\nabla u|^2}{r^{1-\beta}}\;r^{n-1}drdh+\int_{0}^\infty \int_{Y} \big(V_0(y)+\frac{(1-\beta)(n-3+\beta)}2\big)\frac{|u|^2}{r^{3-\beta}}\;r^{n-1}drdh\\
=&{\rm Re}\int_{0}^\infty \int_{Y}\frac{f\bar{u}}{r^{1-\beta}}\;r^{n-1}drdh.\end{align*}
On the other hand, since $-\Delta_h+V_0(y)+(n-2)^2/4$ is positive on ${Y}$ with the smallest eigenvalue $\nu_0^2>0$, that is,
\begin{equation}\label{equ:lowb}
\int_{Y}\left(\left(\frac{(n-2)^2}{4}+V_0(y)\right)|u|^2+|\nabla_h u|^2\right) dh\geq \nu_0^2\int_{Y} |u(r,y)|^2 dh\geq 0.
\end{equation}
Note that $\nabla=(\partial_r, r^{-1}\nabla_h)$, this implies that
\begin{align*}
   & \int_{0}^\infty \int_{Y}  \frac{|\partial_r u|^2}{r^{1-\beta}}\;r^{n-1}drdh+\int_{0}^\infty \int_{Y} \big(\nu_0^2-\frac{(n-2)^2}4+\frac{(1-\beta)(n-3+\beta)}2\big)\frac{|u|^2}{r^{3-\beta}}\;r^{n-1}drdh \\
  \leq & C(\epsilon)\int_{0}^\infty \int_{Y}  |f|^2r^{1+\beta}\;r^{n-1}drdh+\epsilon \int_{0}^\infty \int_{Y} \frac{|u|^2}{r^{3-\beta}}\;r^{n-1}drdh .
\end{align*}
Using Lemma \ref{lem:w-hardy'} with $2k=3-\beta$, we get
$$\int_{0}^\infty \int_{Y} \frac{|u|^2}{r^{3-\beta}}\;r^{n-1}drdh\leq\frac{4}{(n-3+\beta)^2}\int_{0}^\infty \int_{Y} \frac{|\partial_r u|^2}{r^{1-\beta}}\;r^{n-1}drdh.$$
Hence, for
\begin{equation}\label{equ:acondbetalv}
  \nu_0^2-\frac{(n-2)^2}4+\frac{(1-\beta)(n-3+\beta)}2>-\frac{(n-3+\beta)^2}{4}, \Leftrightarrow \nu_0^2>\frac{(1-\beta)^2}4,
\end{equation} there holds
\begin{equation}\label{equ:anotsmalv}
\begin{split}
  &\int_{0}^\infty \int_{Y}  \frac{|\partial_r u|^2}{r^{1-\beta}}\;r^{n-1}drdh+\int_{0}^\infty \int_{Y} \frac{|u|^2}{r^{3-\beta}}\;r^{n-1}drdh \\&\leq C\int_{0}^\infty \int_{Y}  |f|^2r^{1+\beta}\;r^{n-1}drdh.
  \end{split}
\end{equation}

Next we consider the case $\sigma\neq0$.
We multiply \eqref{helm-eqinblv} by $\bar{u}$ and integrate from both sides in $X$. Let $\sigma=\sigma_1+i\sigma_2$, then
  \begin{align}\label{equ:reparteslv}
 \int_{X}\Big(|\nabla u|^2+V_0(y)r^{-2}|u|^2\Big)\;dx-\sigma_1\int_{X}|u|^2\;dx=&{\rm Re}\int_{X}f\bar{u}\;dx,\\\label{equ:impartslv}
  -\sigma_2\int_{X}|u|^2\;dx=&{\rm Im}\int_{X}f\bar{u}\;dx.
  \end{align}

When $\sigma_2\neq0$, it follows from \eqref{equ:impartslv} that
$$|\sigma_2|\int_{X}|u|^2\;dx\leq\frac2{|\sigma_2|}\int_{X}|f|^2\;dx+\frac{|\sigma_2|}4\int_{X}|u|^2\;dx\Rightarrow~u\in L^2(X).$$
Combining this with \eqref{equ:reparteslv}, we obtain
  \begin{align}\label{equ:reparteslv'}
 \|\nabla u\|_{L^2(X)}^2\sim \|\sqrt{\LL_V}u\|_{L^2}^2&=\int_{X}\Big(|\nabla u|^2+V_0(y)r^{-2}|u|^2\Big)\;dx
 \\&=\sigma_1\int_{X} |u|^2\;dx-{\rm Re}\int_{X} f\bar{u}\; dx<\infty  \end{align}
Hence $u\in\dot{H}^1.$
Now multiplying \eqref{helm-eqinblv} by $\frac{\bar{u}}{r^{1-\beta}}$ and integrating in $X$, we get
\begin{align}\label{equ:reaa1lv}
&\int_{0}^\infty \int_{Y} \frac{|\nabla u|^2}{r^{1-\beta}}\;r^{n-1}drdh+\int_{0}^\infty \int_{Y} \big(V_0(y)+\frac{(1-\beta)(n-3+\beta)}2\big)\frac{|u|^2}{r^{3-\beta}}\;r^{n-1}drdh\\ \nonumber
&-\sigma_1\int_{0}^\infty\int_{Y} \frac{|u|^2}{r^{1-\beta}}\;r^{n-1}drdh ={\rm Re}\int_{0}^\infty \int_{Y}\frac{f\bar{u}}{r^{1-\beta}}\;r^{n-1}drdh.\end{align}
  and
  \begin{equation}\label{equ:imaa2lv}
    -\sigma_2\int_{0}^\infty\int_{Y} \frac{|u|^2}{r^{1-\beta}}\;r^{n-1}drdh={\rm Im}\int_{0}^\infty\int_{Y}\ \frac{f\bar{u}}{r^{1-\beta}}\;r^{n-1}drdh.
  \end{equation}
This together with Young's inequality yields $\int_{X}\frac{|u|^2}{r^{1-\beta}}\;dx<+\infty.$ Using this fact to \eqref{equ:reaa1lv} and that by the same argument as for \eqref{equ:anotsmalv}, we conclude that
 if $\nu_0^2>\frac{(1-\beta)^2}4,$ then the following inequality holds
\begin{equation*}
\begin{split}
  &\int_{0}^\infty \int_{Y}  \frac{|\partial_r u|^2}{r^{1-\beta}}\;r^{n-1}drdh+\int_{0}^\infty \int_{Y} \frac{|u|^2}{r^{3-\beta}}\;r^{n-1}drdh \\&\leq C\int_{0}^\infty \int_{Y}  |f|^2r^{1+\beta}\;r^{n-1}drdh.
  \end{split}
\end{equation*}
When $\sigma_2=0$, we have $\sigma_1<0$ due to $\sigma\notin \R^+$. Using  \eqref{equ:reparteslv}, we obtain $u\in{H}^1.$
And \eqref{equ:nabluestlv} follows from \eqref{equ:reaa1lv}.
\end{proof}

With Lemma \ref{lem:uhelmequlv} in hand, we now prove Lemma \ref{to0}. Note the fact that the compact support of $\phi'(r)$ belongs to $[0,m]\cup[M, M+1]$. One has $0\leq \phi'\leq C/m$ on $[0,m]$ and $-C\leq \phi'\leq 0$ on $[M,M+1]$.
Thus by \eqref{lowb} it suffices to prove the negative terms
\begin{equation}\label{M0}
\int_M^{M+1}\int_{Y} r^\beta e^{-2r\tau}\left(|\partial_r v|^2+\frac{1}{4r^2}|v|^2\right) dr dh\to 0, \quad \text{as}~ M\to \infty;
\end{equation}
and
\begin{equation}\label{m0}
\int_0^{m}\int_{Y} r^\beta e^{-2r\tau}\frac1{r^2}\left(|\nabla_h v|^2+(V_0(y)+\frac{(n-2)^2}4)|v|^2\right) dr dh\to 0, \quad \text{as} ~m\to 0.
\end{equation}
It is in fact enough to show that there exists a
sequence $M_n\to\infty$ along which it holds. We note that
$$\partial_rv=r^\frac{n-1}{2}e^{r\z}\big(\partial_ru+\z u+\frac{n-1}{2r}u\big).$$
So using the modified Hardy inequality \eqref{equ:modhardylv}, we have
\begin{align*}
&\int_M^{M+1}\int_{Y} r^\beta e^{-2r\tau}\left(|\partial_r v|^2+\frac{1}{4r^2}|v|^2\right) dr dh\\
\leq&C\int_{M}^{M+1}\int_{{Y}}r^{\beta} \big(|\partial_ru|^2+|\z|^2|u|^2\big)\;dh\; r^{n-1}\;dr\\
\leq&C(n,|\z|)\int_{M}^{M+1} r^\beta g(r)\;dr,
\end{align*}
where
$$g(r)=r^{n-1}\int_{{Y}}\big(|\partial_ru|^2+|u|^2\big)\;dh.$$
Note that by Lemma \ref{lem:uhelmequlv}, we have
$$\int_0^\infty g(r)\;dr<+\infty.$$
It thus follows that, given $\mu_j>0$, there exists a sequence $M_n^{(j)}\to\infty$ such that
$$\int_{M_n^{(j)}}^{M_n^{(j)}+1}g(r)\;dr<\frac{\mu_j}{M_n^{(j)}},$$
because otherwise the integral $\int_0^\infty g(r)\;dr$ would diverge. Using a diagonal argument
it thus follows that there exists a sequence $M_n\to\infty$ such that for $\beta\leq1$ (i.e. $\alpha\geq1$)
$$\int_{M_n}^{M_n+1}r^\beta g(r)\;dr\to0\quad \text{as}\quad n\to\infty,$$
which establishes \eqref{M0} along a sequence.

On the other hand, using Lemma \ref{lem:uhelmequlv}, for $\beta\geq0$ we have
\begin{align*}
&\int_0^{m}\int_{{Y}}  r^{\beta-2}e^{-2r\tau}\left(|\nabla_{h}  v|^2+|v|^2\right) dr dh\\
\leq&\int_0^{m}\int_{{Y}}  r^{-2}\left(|\nabla_{h}  u|^2+|u|^2\right)\;dh\;r^{n-1} \;dr \\
\lesssim&\int_{r\leq m}\left(|\nabla u|^2+\frac{|u|^2}{r^{2}}\right)\;dx\\
\to&0\quad\text{as}\quad m\to0.
\end{align*}

\end{proof}

\begin{proof}[{\bf The proof of Lemma \ref{lem:w-hardy}}]
This is a modification of the weighted Hardy inequality in \cite[Lemma 2.2]{BPST}. We just modify their argument to prove it.
Let the operator $G$ be defined as
$$G:=\frac1i\big(w\partial_r+\frac12w'\big).$$
It follows from \eqref{equ:fcond} that there exists a sequence $\{r_j\}_j:~r_j\to0$ such that
$$\lim_{j\to\infty}w(r_j)w'(r_j)|g(r_j)|^2=0.$$
This together with \eqref{equ:funccond} and \eqref{equ:wcond} with $ww'\leq 0$ yields that
\begin{align*}
  \|Gg\|_{L^2(\R^+)}^2= &\int_0^\infty \big(w^2|g'|^2+\frac14(w')^2|g|^2+\frac12ww'\partial_r|g|^2\big)\;dr \\
  = & \lim_{j\to\infty}\bigg\{\int_{r_j}^\infty\big(w^2|g'|^2+\frac14(w')^2|g|^2-\frac12\partial_r(ww')|g|^2\big)\;dr
  +\frac12w(r)w'(r)|g(r)|^2\Big|_{r=r_j}^{r=+\infty}\bigg\}\\
  \leq&\int_0^\infty\Big(w^2|g'|^2-\big(\frac14(w')^2+\frac12ww''\Big)|g|^2\big)\;dr.
\end{align*}

On the other hand, a simple computation shows that for the function $m(r)=-\frac{w(r)}{2r}$
\begin{align}\nonumber
  0\leq & \big\|(G-im)g\big\|_{L^2(\R^+)}^2 \\\nonumber
  = & \|Gg\|_{L^2(\R^+)}^2-\big\langle (w m'-m^2)g,g\big\rangle\\\nonumber
  \leq&\int_0^\infty\big(w^2|g'|^2-\big(\frac14(w')^2|g|^2+\frac12ww''\big)|g|^2\big)\;dr
  -\big\langle (w m'-m^2)f,f\big\rangle\\\label{equ:fm}
  =&\int_0^\infty\Big(w^2|g'|^2-\big(\frac14(w')^2+\frac12ww''+w m'-m^2\big)|g|^2\Big)\;dr.
\end{align}
Note that $m=-\frac{w}{2r}$, so by \eqref{equ:wcond}, we have
$$\frac14(w')^2+\frac12ww''+w m'-m^2=\frac14(w')^2+\frac12ww''-\frac{ww'}{2r}+\frac{w^2}{4r^2}
\geq\frac{w^2}{4r^2}.$$
Plugging this into \eqref{equ:fm}, we obtain
$$\int_0^\infty w^2\frac{|g|^2}{r^2}\;dr\leq 4\int_0^\infty w^2|g'(r)|^2\;dr.$$

\end{proof}

\begin{proof}[{\bf The proof of Lemma \ref{lem:asmalllv}:}]  Let
\begin{equation}\label{equ:wdef}
  w(r)=e^{-r\tau}r^{\frac{\beta-1}2}(\beta+2\tau r)^{1/2}.
\end{equation}

We first verify the assumption \eqref{equ:wcond} on $w(r)$ when $0<\beta\leq1$. A simple computation shows that
\begin{equation}\label{equ:deri}
  w'(r)=e^{-r\tau}r^{(\beta-1)/2}(\beta+2\tau r)^{-1/2}[-2\tau^2r+\frac12(\beta-1)\beta r^{-1} ]\leq0
\end{equation}
for $0<\beta\leq1.$ Secondly we have
\begin{align*}
  w''(r)= & e^{-r\tau}r^{(\beta-1)/2}(\beta+2\tau r)^{-3/2}[-2\tau^2r+\frac12(\beta-1)\beta r^{-1} ]^2\\
+&e^{-r\tau}r^{(\beta-1)/2} (\beta+2\tau r)^{-3/2}[-2\tau(\beta-1)\beta r^{-1} -2\tau^2\beta-\frac12(\beta-1)\beta^2 r^{-2}]
\end{align*}
which implies
\begin{align*}
 &w'(r)^2+2w(r)w''(r)\\=&e^{-2r\tau}r^{\beta-1}(\beta+2\tau r)^{-1}[-2\tau^2r+\frac12(\beta-1)\beta r^{-1} ]^2\\
&+e^{-2r\tau}r^{\beta-1}(\beta+2\tau r)^{-3/2}[-2\tau^2r+\frac12(\beta-1)\beta r^{-1} ]^2\\
&+e^{-2r\tau}r^{\beta-1} (\beta+2\tau r)^{-3/2}[-2\tau(\beta-1)\beta r^{-1} -2\tau^2\beta-\frac12(\beta-1)\beta^2 r^{-2}]\\=&e^{-2r\tau}r^{\beta-1}(\beta+2\tau r)^{-1}[-2\tau^2r+\frac12(\beta-1)\beta r^{-1} ]^2\\
&+e^{-2r\tau}r^{\beta-1}(\beta+2\tau r)^{-3/2}[-2\tau^2r+\frac12(\beta-1)\beta r^{-1} ]^2\\
&+e^{-2r\tau}r^{\beta-1} (\beta+2\tau r)^{-3/2}[-2\tau(\beta-1)\beta r^{-1} -2\tau^2\beta-\frac12(\beta-1)\beta^2 r^{-2}]\\
\geq& -2\tau^2\beta e^{-2r\tau}r^{\beta-1} (\beta+2\tau r)^{-3/2}.
\end{align*}
In addition, we have
\begin{align*}
-2w(r)w'(r)&=e^{-2r\tau}r^{\beta-1}(\beta+2\tau r)^{-1/2}[4\tau^2r-(\beta-1)\beta r^{-1} ]\\
&=e^{-2r\tau}r^{\beta-1}(\beta+2\tau r)^{-3/2}[4\tau^2r(\beta+2\tau r)-(\beta-1)\beta r^{-1}(\beta+2\tau r) ]
\\
&=e^{-2r\tau}r^{\beta-1}(\beta+2\tau r)^{-3/2}[4\tau^2\beta r+8\tau^3 r^2-(\beta-1)\beta r^{-1}(\beta+2\tau r) ]
\\
&\geq 4\tau^2\beta r e^{-2r\tau}r^{\beta-1}(\beta+2\tau r)^{-3/2}
\end{align*}
Hence,  for $0<\beta\leq1$
\begin{align*}
&r\left(w'(r)^2+2w(r)w''(r)\right)-2w(r)w'(r)\geq 2\tau^2\beta r e^{-2r\tau}r^{\beta-1}(\beta+2\tau r)^{-3/2}\geq0.
\end{align*}
Let
\begin{equation}\label{equ:fdef}
  g(r)=\Big(\int_{{Y}}|v(r,y)|^2\;dh\Big)^\frac12.
\end{equation}
Then, we have
$$g'(r)=\Big(\int_{{Y}}|v(r,y)|^2\;dh\Big)^{-\frac12}\int_{{Y}}{\rm Re}(v\partial_rv)\;dh$$
and
$$|g'(r)|^2\leq\int_{{Y}}|\partial_rv|^2\;dh.$$
Next we need to verify the assumptions \eqref{equ:funccond} and \eqref{equ:fcond} on $g$.

For the choice of $w$ as in \eqref{equ:wdef} and $g$ as in \eqref{equ:fdef}, we have
\begin{align*}
   & \int_0^\infty \big((w')^2|g|^2+w^2|g'|^2\big)\;dr\\
   \leq &\int_0^\infty  e^{-2r\tau}r^{\beta-1}(\beta+2\tau r)^{-1}[-2\tau^2r+\frac12(\beta-1)\beta r^{-1}]^2\int_{{Y}}|v(r,y)|^2\;dh\;dr\\
   &+\int_0^\infty  e^{-2r\tau}r^{\beta-1}(\beta+2\tau r)\int_{{Y}}|\partial_rv|^2\;dh\;dr.
\end{align*}
Recall that
 $$v(r,y)=r^\frac{n-1}{2} e^{r\z}u(r,y),$$
and
$$\partial_rv=r^\frac{n-1}{2}e^{r\z}\big(\partial_ru+\z u+\frac{n-1}{2r}u\big),$$
hence
\begin{align}\nonumber
  & \int_0^\infty \big((w')^2|g|^2+w^2|g'|^2\big)\;dr\\\label{equ:bound1}
  \leq& C(\tau)\int_0^\infty\int_{Y} \Big(\frac{|u|^2}{r^{3-\beta}}+\frac{|\partial_ru|^2}{r^{1-\beta}}\Big)\;dh r^{n-1}\;dr\\\label{equ:bound2}
  &+C(\tau)\int_0^\infty\int_{Y} \Big(\frac{|u|^2}{r^{-\beta}}+\frac{|\partial_ru|^2}{r^{-\beta}}\Big)\;dh r^{n-1}\;dr.
\end{align}
Note that the boundedness of \eqref{equ:bound1} and \eqref{equ:bound2} follow from \eqref{equ:nabluestlv}. \vspace{0.1cm}

Now we verify \eqref{equ:fcond}. Recall  $w$ as in \eqref{equ:wdef} and $g$ as in \eqref{equ:fdef}, we have
 \begin{align*}
   w(r)w'(r)|g(r)|^2 =& e^{-2r\tau}r^{\beta-1}[-2\tau^2r+\frac12(\beta-1)\beta r^{-1}]|g(r)|^2,
 \end{align*}
We are reduced to show
that $$\liminf_{r\to0}\tilde{g}(0)=0$$ with
\begin{equation}\label{equ:gdef}
  \tilde{g}(r):=r^\frac{\beta-2}{2}f(r)=r^\frac{\beta-2}{2}\Big(\int_{{Y}}|v(r,y)|^2\;dh\Big)^\frac12.
\end{equation}
Recall
$$v(r,y)=r^\frac{n-1}{2} e^{r\z}u(r,y).$$
For the above $\beta\in(\max\{0,1-2\nu_0\},1]$, we can choose $\beta_0\in(\max\{0,1-2\nu_0\},1)$ such that $\beta_0<\beta$, and let $\epsilon=\beta-\beta_0.$

Moreover, we have by \eqref{equ:nabluestlv} with $\beta_0\in(\max\{0,1-2\nu_0\},1)$
\begin{align}\nonumber
\int_0^1\frac{|\tilde{g}(r)|^2}{r^{1+\epsilon}}\;dr=&\int_0^1 r^{\beta-3-\epsilon}\int_{Y}|v(r,y)|^2\;dh\;dr\\\nonumber
=&\int_0^1\int_{{Y}}e^{2r\tau}\frac{|u(r,y)|^2}{r^{3-\beta_0}}\;dh\;r^{n-1}\;dr\\\label{equ:gzr}
\leq& C(\tau)\int_{r\leq1}\frac{|u|^2}{r^{3-\beta_0}}\;dx<+\infty
\end{align}
which shows that
\begin{equation}\label{equ:g0dui}
  \liminf_{r\to0}\tilde{g}(r)=0,
\end{equation}
 otherwise the integral $\int_0^1\frac{|\tilde{g}(r)|^2}{r^{1+\epsilon}}\;dr$ diverges.

Therefore we have verified the condition of Lemma \ref{lem:w-hardy}. By Lemma \ref{lem:w-hardy}, we obtain
\begin{align*}
  \int_0^\infty\int_{Y} e^{-2r\tau} r^{\beta-1}\left(\beta+ 2r\tau\right) \frac{|v|^2}{r^2}  dr dh=&\int_0^\infty w^2 \frac{|g|^2}{r^2}\;dr\\
  \leq&4\int_0^\infty w^2 |g'(r)|^2\;dr\\
  \leq&4\int_0^\infty e^{-2r\tau}r^{\beta-1}(\beta+2\tau r)\int_{{Y}}|\partial_rv|^2\;dh\;dr\\
  =&4\int_0^\infty \int_{{Y}}e^{-2r\tau}r^{\beta-1}(\beta+2\tau r)|\partial_rv|^2\;dh\;dr,
\end{align*}
which implies \eqref{equ:asmalllv}, and so we conclude the proof of Lemma \ref{lem:asmalllv}.
\end{proof}

\subsection{An application of the endpoint inhomogeneous Strichartz estimate} As a direct consequence of
\begin{equation}\label{inhS}
\Big\|\int_{s<t}e^{i(t-s)\LL_V}F(s) ds \Big\|_{L^2_t(\R;L^{\frac{2n}{n-2}}(X))}\lesssim\|F\|_{L^2_t(\R;L^{\frac{2n}{n+2}}(X))},
\end{equation}
we have the following uniform Sobolev estimate
\begin{proposition}
Let $\mathcal{L}_V$ be as above. Then the Sobolev inequality holds for
\begin{equation}\label{uSob}
\sup_{\sigma\notin \R^+} \|(\LL_V-\sigma)^{-1}f\|_{L^{\frac{2n}{n-2}}(X)}\leq C\|f\|_{L^{\frac{2n}{n+2}}(X)}.
\end{equation}
\end{proposition}
In fact, the inhomogeneous Strichartz estimate implies the uniform Sobolev inequality which was pointed out by Thomas Duyckaerts
and Colin Guillarmou; we also refer the reader to \cite[Remark 8.8]{HZ}.

\begin{proof}

By density, we assume $f\in \mathcal{S}(X)$. Let $u(t,x)=e^{it\sigma}f(x)$. Then, it solves
\begin{equation*}
\begin{cases}
i\partial_tu+\LL_V u=F(t,x),\\
u(0,x)=f(x)
\end{cases}
\end{equation*}
where
$$F(t,x)=e^{it\sigma}(\LL_V-\sigma)f.$$
Thus, by the endpoint inhomogeneous Strichartz estimate, for any $T>0$, we have
\begin{align}\label{equ:uftq}
\|u\|_{L^2([-T,T];L^\frac{2n}{n-2}(X))}\leq C\|f\|_{L_x^2}+C\|F\|_{L^2([-T,T];L^\frac{2n}{n+2}(X))}.
\end{align}
Let $$\Gamma(\sigma,T)=\|e^{it\sigma}\|_{L^2([-T,T])}\geq\sqrt{T}$$ which follows from the fact $|e^{it\sigma}|\geq1$ either on $[0,T]$ or $[-T,0]$. Hence $\Gamma(\sigma,T)\to \infty$ as $T\to \infty$.
By the definition of $u$ and $F$, we get
$$\|u\|_{L^2([-T,T]; L^\frac{2n}{n-2}(X))}=\Gamma(\sigma,T)\|f\|_{L^\frac{2n}{n-2}(X)},$$
and $$ \|F\|_{L^2([-T,T], L^\frac{2n}{n+2}(X))}=\Gamma(\sigma,T)\|(\LL_V-\sigma)f\|_{L^\frac{2n}{n+2}(X)}.$$
Dividing \eqref{equ:uftq} by $\Gamma(\sigma,T)$ in both sides, we obtain
\begin{equation}\label{equ:alequs}
\|f\|_{L^\frac{2n}{n-2}(X)}\leq \frac{C}{\Gamma(\sigma,T)}\|f\|_{L_x^2}+C\|(\LL_V-\sigma)f\|_{L^\frac{2n}{n+2}(X)},
\end{equation}
which implies the uniform Sobolev estimate \eqref{uSob} by taking $T\to\infty$.
\end{proof}

\section{The proof of  Theorem \ref{thm:Strichartz'} }

In this section, we prove Theorem \ref{thm:Strichartz'} by following our previous argument in \cite{ZZ2} for wave equation.
For self-contained, we modify and provide the detail argument.
To this end, we first recall the Sobolev inequality \cite[Proposition 3.1, Corollary 3.1]{ZZ2} associated with the operator $\LL_V$.

\begin{proposition}[Sobolev inequality for $\LL_V$]\label{P:sobolev}  Let $n\geq 3$ and $\nu_0$ be as above. Suppose $s\geq0$, and $2\leq p,q<\infty$.  Then
\begin{equation}\label{est:sobolev}
\big\|f(z)\big\|_{L^q(X)}\lesssim \big\|\mathcal L_V^\frac{s}2f\big\|_{L^p(X)}
\end{equation}
holds for $s=\tfrac{n}p-\tfrac{n}q$ and
\begin{equation}\label{est:sobolev_hyp}
\frac{n}{\min\{1+\frac n2+\nu_0-s, n\}}<q<\frac{n}{\max\{\frac n2-1-\nu_0, 0\}}.\end{equation}
\end{proposition}

\begin{remark} Since $p\geq2$ and $2\leq q<n/\max\{\frac n2-1-\nu_0, 0\}$, it follows $s=\frac np-\frac nq<1+\nu_0$.
Therefore the above result works for $s\in[0,1+\nu_0)$.
\end{remark}

\begin{proof} The case with $s=0$ is trivial. The cases where $s>0$ follow from \cite[Proposition 3.1, Corollary 3.1]{ZZ2}.

\end{proof}

\begin{proof}[The proof of Theorem \ref{thm:Strichartz'}.]  We first prove \eqref{Str-est'} when $(q,r)$ satisfies \eqref{adm-p-s'}.
Since $\frac1\rr>\frac12-\frac{1+\nu_0}n$ and $s\in[0,1+\nu_0)$, it follows that $r$ satisfies \eqref{est:sobolev_hyp}. So, by using \eqref{est:sobolev} and Theorem \ref{thm:Strichartz}, we get
\begin{equation}
\begin{split}
\|e^{it\mathcal{L}_V}u_0\|_{L^q(\R;L^\rr(X))}\lesssim \big\|\mathcal L_V^\frac{s}2 e^{it\mathcal{L}_V}u_0\big\|_{L^q(\R;L^{\tilde{r}}(X))}\lesssim
\big\|\mathcal L_V^\frac{s}2 u_0\big\|_{L^2(X)}
\end{split}
\end{equation}
where
\begin{equation*}
s=n(1/\tilde{\rr}-1/\rr),\quad 2/q=n(1/2-1/\tilde{\rr}).
\end{equation*}
Therefore we have proved \eqref{Str-est'} when $(q,\rr)\in \Lambda_{s,\nu_0}$.

Next we construct a counterexample to claim the the requirement $\frac1\rr>\frac12-\frac{1+\nu_0}n$ is necessary. Assume $u_0=(\mathcal{H}_{\nu_0}\chi)(r)$ is independent of $y$, where $\chi\in\CC_c^\infty([1/2,1])$ is valued in $[0,1]$
and $\mathcal{H}_{\nu_0}$ is in \eqref{hankel}. Due to the compact support of  $\chi$ and the
unitarity of $\mathcal{H}_{\nu_0}$ on $L^2$, we obtain $\|u_0\|_{\dot H^s}\leq C$. Now we conclude that there is a possibility such that
\begin{equation}
\begin{split} \|e^{it\mathcal{L}_V}u_0\|_{L^q(\R;L^\rr(X))}=\infty\end{split}
\end{equation}
when $\frac1\rr\leq\frac12-\frac{1+\nu_0}n$. By using \eqref{s.exp}, we write that
\begin{equation}
\begin{split} e^{it\mathcal{L}_V}u_0&=\int_0^\infty(r\rho)^{-\frac{n-2}2}J_{\nu_0}(r\rho)e^{
it\rho^2}(\mathcal{H}_{\nu_0}u_0)(\rho)\rho^{n-1}d\rho\\
&=\int_0^\infty(r\rho)^{-\frac{n-2}2}J_{\nu_0}(r\rho)e^{
it\rho^2}\chi(\rho)\rho^{n-1}d\rho.
\end{split}
\end{equation}
We recall the behavior of  $J_\nu(r)$ as $r\to 0+$. For the complex number $\nu$ with $\Re(\nu)>-1/2$, see \cite[Section B.6]{G}.
Now, we have
\begin{equation}\label{Bessel1}
J_{\nu}(r)=\frac{r^\nu}{2^\nu\Gamma(\nu+1)}+S_\nu(r)
\end{equation}
where
\begin{equation}\label{Bessel2}
S_{\nu}(r)=\frac{(r/2)^{\nu}}{\Gamma\left(\nu+\frac12\right)\Gamma(1/2)}\int_{-1}^{1}(e^{isr}-1)(1-s^2)^{(2\nu-1)/2}\mathrm{d
}s
\end{equation}
which satisfies
\begin{equation}\label{Bessel3}
|S_{\nu}(r)|\leq \frac{2^{-\Re\nu}r^{\Re\nu+1}}{(\Re\nu+1)|\Gamma(\nu+\frac12)|\Gamma(\frac12)}.\end{equation}
Now we compute for any $0<\epsilon\ll 1$
\begin{align*}
\|e^{it\mathcal{L}_V}u_0\|_{L^q(\R;L^\rr)}
=&\left\|\int_0^\infty(s\rho)^{-\frac{n-2}2}J_{\nu_0}(s\rho)e^{
it\rho^2}\chi(\rho)\rho^{n-1}d\rho\right\|_{L^q(\R;L^\rr)}\\
\geq&\left\|\int_0^\infty(s\rho)^{-\frac{n-2}2}J_{\nu_0}(s\rho)e^{
it\rho^2}\chi(\rho)\rho^{n-1}d\rho\right\|_{L^q([0,1/2];L^\rr_{s^{n-1}ds}[\epsilon,1])}\\
\geq& c\left\|\int_0^\infty(s\rho)^{-\frac{n-2}2}(s\rho)^{\nu_0}e^{
it\rho^2}\chi(\rho)\rho^{n-1}d\rho\right\|_{L^q([0,1/2];L^\rr_{s^{n-1}ds}[\epsilon,1])}\\
&-\left\|\int_0^\infty(s\rho)^{-\frac{n-2}2}S_{\nu_0}(s\rho)e^{
it\rho^2}\chi(\rho)\rho^{n-1}d\rho\right\|_{L^q([0,1/2];L^\rr_{s^{n-1}ds}[\epsilon,1])}.
\end{align*}
 First, we  observe that
\begin{equation}
\begin{split}
&\left\|\int_0^\infty(s\rho)^{-\frac{n-2}2}S_{\nu_0}(s\rho)e^{
it\rho^2}\chi(\rho)\rho^{n-1}d\rho\right\|_{L^q([0,1/2];L^\rr_{s^{n-1}ds}[\epsilon,1])}\\
\leq& C\left\|\int_0^\infty(s\rho)^{-\frac{n-2}2}(s\rho)^{\nu_0+1}\chi(\rho)\rho^{n-1}d\rho\right\|_{L^q([0,1/2];L^\rr_{s^{n-1}ds}[\epsilon,1])}\\
\leq& C\max\big\{\epsilon^{\nu_0+1-\frac{n-2}2+\frac n\rr},1\big\}.
\end{split}
\end{equation}
Next we estimate the lower boundedness
\begin{align*}
&\left\|\int_0^\infty(s\rho)^{-\frac{n-2}2}(s\rho)^{\nu_0}e^{
it\rho^2}\chi(\rho)\rho^{n-1}d\rho\right\|_{L^q([0,1/2];L^\rr_{s^{n-1}dr}[\epsilon,1])}\\
=&\left(\int_0^{\frac12}\left( \int_{\epsilon}^1 \left|\int_0^\infty(s\rho)^{-\frac{n-2}2}(s\rho)^{\nu_0}e^{
it\rho^2}\chi(\rho)\rho^{n-1}d\rho\right|^\rr s^{n-1}ds\right)^{q/\rr}dt\right)^{1/q}\\
=&C\left(\int_0^{\frac12}\left|\int_0^\infty\rho^{-\frac{n-2}2}\rho^{\nu_0}e^{
it\rho^2}\chi(\rho)\rho^{n-1}d\rho\right|^{q}dt\right)^{1/q}\begin{cases}\epsilon^{\nu_0-\frac{n-2}2+\frac n\rr} \quad\text{if}\quad \frac1\rr<\frac12-\frac{\nu_0+1}{n}\\
\ln\epsilon \qquad \qquad\text{if}\quad \frac1\rr=\frac12-\frac{\nu_0+1}{n}
\end{cases}\\
\geq& c\begin{cases}\epsilon^{\nu_0-\frac{n-2}2+\frac n\rr} \quad\text{if}\quad \frac1\rr<\frac12-\frac{\nu_0+1}{n}\\
\ln\epsilon \qquad \qquad\text{if}\quad \frac1\rr=\frac12-\frac{\nu_0+1}{n}
\end{cases}
\end{align*}
where we used the fact that $\cos(\rho^2 t)\geq 1/2$ for $t\in [0, 1/2]$ and that $\rho\in [1/2,1]$ together with the following inequality
\begin{equation}
\begin{split}
\left|\int_0^\infty\rho^{-\frac{n-2}2}\rho^{\nu_0}e^{
it\rho^2}\chi(\rho)\rho^{n-1}d\rho\right|\geq \frac12\int_0^\infty\rho^{-\frac{n-2}2}\rho^{\nu_0}\chi(\rho)\rho^{n-1}d\rho\geq c.
\end{split}
\end{equation}
Hence, we obtain
\begin{equation}
\begin{split}
\|e^{it\mathcal{L}_V}u_0\|_{L^q(\R;L^\rr)}&\geq c\epsilon^{\nu_0-\frac{n-2}2+\frac n\rr}-C\max\big\{\epsilon^{\nu_0+1-\frac{n-2}2+\frac n\rr},1\big\}
\\&\geq c\epsilon^{\nu_0-\frac{n-2}2+\frac n\rr}\to +\infty \quad \text{as}\quad \epsilon\to 0
\end{split}
\end{equation}
when $\frac1\rr<\frac12-\frac{\nu_0+1}n.$ And when $\frac1\rr=\frac12-\frac{\nu_0+1}n$, we get
\begin{align*}
\|e^{it\mathcal{L}_V}u_0\|_{L^q(\R;L^\rr)}&\geq c\ln\epsilon-C\to +\infty \quad \text{as}\quad \epsilon\to 0.
\end{align*}
This completes the proof of Theorem \ref{thm:Strichartz'}.

\end{proof}

\section{Proof of Theorem \ref{thm:NLS}}
In this section, we prove Theorem \ref{thm:NLS}. The key points are the Strichartz estimate in Theorem \ref{thm:Strichartz} and the Leibniz chain rule in Proposition \ref{Prop:Leibnitz}.

\subsection{Well-posedness theory}

\begin{proposition}[Local well-posedness theory]\label{localwell}
Let $n=3$. Assume that   $u_0\in H^1(X)$. Then
there exists $T=T(\|u_0\|_{H^1})>0$ such that the equation
\eqref{equ:cubic} has a unique solution $u$ with
\begin{equation}\label{equ:small}
u\in C(I; H^1(X))\cap L_t^{q}(I; H^1_\rr(X)),\quad I=[0,T),
\end{equation}
where the pair $(q,\rr)$ is an \emph{admissible pair} as in \eqref{adm-p}.
\end{proposition}

\begin{proof}
We follow the standard Banach fixed point argument to prove this
result. Consider the map
\begin{equation}\label{inte3}
\Phi(u(t))=e^{it\LL_V}u_0-i\gamma\int_{0}^te^{i(t-s)\LL_V}(|u|^2u(s))ds
\end{equation}
on the complete metric space $B_T$
\begin{align*}
B_T:=\big\{&u\in Y(I)\triangleq C_t(I; H^{1})\cap L_t^{q_0}(I;
H^1_{r_0}):\ \|u\|_{Y(I)}\leq2CC_1\|u_0\|_{H^1}\big\}
\end{align*}
with the metric
$d(u,v)=\big\|u-v\big\|_{L_t^{q_0}L_x^{r_0}(I\times X)}$ and $(q_0,r_0)$ satisfies \eqref{adm-p} and
\begin{equation}\label{equ:q0r0}
(q_0,r_0)=\begin{cases}
\big(5,\frac{30}{11}\big),\quad \text{if}\quad \nu_0>\frac25,\\
\big(\big(\frac2{\nu_0}\big)_+,\big(\frac6{3-2\nu_0}\big)_-\big)\quad \text{if}\quad \nu_0\leq\frac25
\end{cases}
\end{equation}
where $\nu_0$ is the positive square root of the smallest eigenvalue of $\Delta_h+V_0(y)+(n-2)^2/4$.
We need to
prove that the operator $\Phi$ defined by $(\ref{inte3})$
is well-defined on $B_T$ and is a contraction map under the metric $d$
for $I$.\vspace{0.1cm}

Let $u\in B_T$. By Sobolev embedding and equivalence of Sobolev spaces,
\[
\|u\|_{L_t^{q_0} L_x^6} \leq C\| u\|_{L_t^{q_0} \dot H^{\theta}_{r_0}}  \leq C\| u\|_{L_t^{q_0} H^1_{r_0}}\leq \tilde{C}\|u_0\|_{H^1},\quad
\theta=\frac3{r_0}-\frac12.
\]
Then, we have by Strichartz estimates and Proposition \ref{Prop:Leibnitz}
\begin{align*}
\big\|\Phi(u)\big\|_{Y(I)}
\leq& C\|u_0\|_{H^1}+C\big\|\langle \LL_V^{1/2}\rangle
(|u|^2u) \big\|_{L_t^{2}L_x^{6/5}(I\times X)
}\\
\leq&C\|u_0\|_{H^1}+CC_1|I|^{\frac12-\frac2{q_0}}\big\|\langle\LL_V^{1/2}\rangle
u\big\|_{L_t^{\infty}L_x^{2}}\|u\|_{L_t^{q_0}L_x^6}^2.
\end{align*}
Note $\|u\|_{Y(I)}\leq2CC_1\|u_0\|_{H^1}$ if $u\in B_T$, we see
that for $u\in B_T$,
\begin{align*}
\big\|\Phi(u)\big\|_{Y(I)}\leq
C\|u_0\|_{H^1}+\tilde{C}|I|^{\frac12-\frac2{q_0}}(2CC_1\|u_0\|_{H^1})^3.
\end{align*}
Taking $|I|$ sufficiently small such that
$$\tilde{C}|I|^{\frac12-\frac2{q_0}}((2CC_1\|u_0\|_{H^1})^3\leq\frac12\|u_0\|_{H^1},$$
we have $\Phi(u)\in B_T$ for $u\in B_T$. On the other hand, by the same argument as before, for $u,v\in B_T$, we have
\begin{align*}
d\big(\Phi(u),\Phi(v)\big)
\leq&C|I|^{\frac12-\frac2{5}}\big(\|u\|_{Y(I)}^2+\|v\|_{Y(I)}^2\big)d(u,v).
\end{align*}
Thus we derive by taking $|I|$ small enough
\begin{equation*}
d\big(\Phi(u),\Phi(v)\big)\leq\frac{1}{2}d(u,v).
\end{equation*}

The standard fixed point argument and applying again the Strichartz estimates give a unique solution $u$ of
\eqref{equ:cubic} on $I\times X$ which satisfies the bound
\eqref{equ:small}.\vspace{0.1cm}

\end{proof}

By using Proposition \ref{localwell}, mass and energy conservations, we conclude the proof of global well-posedness result of
Theorem \ref{thm:NLS} in defocusing case $\gamma=1$.

\subsection{Scattering theory}
The scattering result of Theorem \ref{thm:NLS} follows from the following Proposition.
\begin{proposition}[Small data implying scattering]\label{prop:small} Let $n=3$.
Assume $\|u_0\|_{H^1(X)}\leq \epsilon$ for a small constant $\epsilon$. Then, there exists a global solution $u$  to \eqref{equ:cubic}.
Moreover, the solution $u$ scatters in sense that
there are $u_\pm\in H^1(X)$ such that
\begin{equation}\label{equ:scattering}
\lim\limits_{t\to\pm \infty}\|u(t)-e^{it\LL_V}u_\pm\|_{H^1(X)}=0.
\end{equation}
\end{proposition}

\begin{proof}
First, we use the fixed point argument to show the global existence. To do this,
 we consider the map again
\begin{equation}
\Phi(u(t))=e^{it\LL_V}u_0-i\gamma\int_{0}^te^{i(t-s)\LL_V}(|u|^2u(s))ds
\end{equation}
on the complete metric space $B$
\begin{align*}
B:=\big\{&u\in S^1(\R):\ \|u\|_{S^1(\R)}\leq2C\epsilon\big\}
\end{align*}
with the metric
$d(u,v)=\big\|u-v\big\|_{L_t^{q_0}L_x^{r_0}(I\times X)}$, $(q_0,r_0)$ is as in \eqref{equ:q0r0}, and define
$$\|u\|_{S^1(\R)}:=\sup_{(q,r)\in\Lambda_0; q\geq 2_+}\| u\|_{L_t^q(\R;H^1_\rr(X))}.$$

Using Strichartz estimates, Proposition \ref{Prop:Leibnitz}, H\"older's inequality and Sobolev embedding, for $u\in B$, we get
\begin{align*}
\|u\|_{S^1(\R)}\leq& C\|u_0\|_{H^1(X)}+C\big\||u|^2u\big\|_{L_t^2(\R,H^1_{6/5}(X))}\\
\leq&C\|u_0\|_{H^1(X)}+C\|u\|_{L_t^{q_0}(\R,H^1_{r_0})}\|u\|_{L_t^q(\R,L_x^\rr)}^2\\
\leq&C\epsilon+C\|u\|_{L_t^{q_0}(\R,H^1_{r_0})}\|u\|_{L_t^\infty(\R,L_x^6)}\|u\|_{L_t^{2q}(\R,L_x^{r_1})}\\
\leq&C\epsilon+C\|u\|_{S^1(\R)}^3,
\end{align*}
where $(q_0,r_0)$ is as in \eqref{equ:q0r0} and
$$\frac12=\frac1{q_0}+\frac2q,~\frac56=\frac1{r_0}+\frac2\rr,~\frac2\rr=\frac16+\frac1{r_1}.$$
By continuous argument, we obtain $$\|u\|_{S^1(\R)}\leq 2C\epsilon.$$
Hence, we have $\Phi(u)\in B$ for $u\in B.$ On the other hand, by the same argument as before, we have for $u,
v\in B$,
\begin{align*}
d\big(\Phi(u),\Phi(v)\big)\leq&C\big\||u|^2u-|v|^2v\big\|_{L_t^2(\R,L_x^{6/5})}\\
\leq&C\|u-v\|_{L_t^{q_0}(\R,L_x^{r_0})}\big(\|u\|_{L_t^q(\R,L_x^\rr)}^2+\|v\|_{L_t^q(\R,L_x^\rr)}^2\big)\\
\leq&C\|(u,v)\|_{S(\R)}^2d(u,v)\\
\leq&C\epsilon^2d(u,v)\leq\frac12d(u,v),
\end{align*}
providing that $\epsilon$ is sufficiently small.

Therefore, the standard fixed point argument gives a unique global solution $u$ of
\eqref{equ:cubic}.

Next, we turn to prove the scattering part.
By time reversal symmetry, it suffices to prove this for positive
times. For $t>0$, we will show that $v(t):=e^{-it\LL_V}u(t)$ converges
in $H^1_x$ as $t\to+\infty$, and denote $u_+$ to be the limit. In
fact, we obtain by Duhamel's formula
\begin{equation}\label{equ4.21}
v(t)=u_0-i\gamma\int_0^te^{-i\tau \LL_V}(|u|^2u)(\tau)d\tau.
\end{equation}
Hence, for $0<t_1<t_2$, we have
$$v(t_2)-v(t_1)=-i\gamma\int_{t_1}^{t_2}e^{-i\tau \LL_V}(|u|^2u)(\tau)d\tau.$$
Arguing as before, we deduce that
\begin{align*}
\|v(t_2)-v(t_1)\|_{H^1(X)}=&\Big\|\int_{t_1}^{t_2}e^{-i\tau \LL_V}(|u|^{p-1}u)(\tau)d\tau\Big\|_{H^1(X)}\\
\lesssim&\big\||u|^2u
\big\|_{L_t^2H^1_{6/5}([t_1,t_2]\times X)}\\
\lesssim&\|u\|_{L_t^{q_0}H^1_{r_0}([t_1,t_2]\times X)}\|u\|_{L_t^q(\R,L_x^\rr)}^2
\\
\to&0\quad \text{as}\quad t_1,~t_2\to+\infty.
\end{align*}
As $t$ tends to $+\infty$, the limitation of \eqref{equ4.21} is well
defined. In particular, we find the asymptotic state
$$u_+=u_0-i\gamma\int_0^\infty e^{-i\tau \LL_V}(|u|^2u)(\tau)d\tau.$$
Therefore, we conclude the proof of Proposition \ref{prop:small}.

\end{proof}

\begin{center}

\end{center}

\begin{thebibliography}{99}
%\addcontentsline{toc}{section}{References}

\bibitem{BMW} D. Baskin, J. L. Marzuola, and J. Wunsch, Strichartz estimates on exterior polygonal domains, Geometric and Spectral Analysis, Contemporary Mathematics, 630( 2014),  291-306.

\bibitem{BW} D. Baskin, and J. Wunsch,  Resolvent estimates and local decay of waves on conic manifolds. Journal of Differential Geometry 95(2013), 183-214.

\bibitem{BFHM} M. D. Blair, G. A. Ford, S. Herr, and J. L. Marzuola, Strichartz estimates for the Schr\"odinger equation on polygonal domains, J. Geom. Anal. 22(2012), 339-351.

\bibitem{BFM} M. D. Blair, G. A. Ford, and J. L. Marzuola, Strichartz estimates for the wave equation on flat cones, IMRN, 3(2013), 562-591.

\bibitem{BM} J. M. Bouclet and H. Mizutani, Global-in-time Strichartz inequalities on
asymptotically flat manifolds with temperate
trapping, arXiv: 1602.06287.

\bibitem{BM1} J. M. Bouclet and H. Mizutani, Uniform resolvent and Strichartz estimates for Schr\"odinger equations with critical singularities, Trans. Amer. Math. Soc., 370(2018), 7293-7333.

\bibitem{BGH} N. Burq C. Guillarmou and A. Hassell, Strichartz estimates without loss on manifolds with hyperbolic trapped geodesics, Geom. Funct. Anal. 20(2010), 627-656.

\bibitem{BPSS} N. Burq, F. Planchon, J. Stalker and A. S.
Tahvildar-Zadeh, Strichartz estimates for the wave and Schr\"odinger
equations with the inverse-square potential, J. Funct. Anal., 203
(2003), 519-549.

\bibitem{BPST} N. Burq, F. Planchon, J. G. Stalker, A. S. Tahvildar-Zadeh, Strichartz estimates for the wave and {S}chr\"odinger equations with potentials of critical decay,
Indiana Univ. Math. J., 53 (2004), 1665-1680.

\bibitem{Carron} G. Carron, Le saut en z\'ero de la fonction de d\'ecalage spectral, J. Funct. Anal., 212 (2004), 222-260.

\bibitem{C1} J. Cheeger, On the spectral geometry of spaces with cone-like singularities,
Proceedings of the National Academy of Sciences of the United States of America, 76(1979), 2103-2106.

\bibitem{C2} J. Cheeger, Spectral geometry of singular Riemannian spaces,  J. Diff. Geo. 18(1983), 575-657.

\bibitem{CT} J. Cheeger and M. Taylor, On the diffraction of waves by conical singularities, I,
Comm. Pure Appl. Math., 35(1982), 275-331.

\bibitem{CT1} J. Cheeger and M. Taylor, On the diffraction of waves by conical singularities, II,
Comm. Pure Appl. Math., 35(1982), 487-529.


\bibitem{CK} M. Christ and A. Kiselev, Maximal functions
associated to filtrations, J. Funct. Anal., 179(2001),
409-425.

\bibitem{CW} M. Christ and M. Weinstein, Dispersion of small amplitude solutions of the generalized
Korteweg-de Vries equation, J. Funct. Anal., 100(1991), 87-109.

\bibitem{D} P. D'ancona, Kato smoothing and Strichartz estimates for wave equations with magnetic potentials, Commun. Math. Phys. 335 (2015), 1-16.

\bibitem{FFFP} L. Fanelli, V. Felli, M. Fontelos and A. Primo,
Time decay of scaling critical electromagnetic Schr\"odinger flows, Commun. Math. Phys. 324 (2013), 1033-1067.

\bibitem{Ford} G. A. Ford, The fundamental solution and
Strichartz estimates for the Schr\"odinger equation on flat
Euclidean cones, Comm. Math. Phys., 299(2010), 447-467.

\bibitem{FW} G. A. Ford, and J. Wunsch, The diffractive wave trace on manifolds with conic singularities, Adv. in Math. 304 (2017), 1330-1385.

\bibitem{Frie} F.G. Friedlander, The diffraction of sound pulses, I, Diffraction by a semi-infinite plane,
Proc. Roy. Soc. London, 186(1946), 322-344.

\bibitem{Frie1} F.G. Friedlander, The diffraction of sound pulses, II, Diffraction by a semi-infinite plane,
Proc. Roy. Soc. London, 186(1946), 344-351.

\bibitem{G} L. Grafakos, \emph{Classical Fourier Analysis}, Second Edition, Graduate Texts in Math., no 249,
Springer, New York, 2008.


\bibitem{GH} C. Guillarmou, and A. Hassell, Uniform Sobolev estimates for
non-trapping metrics, J. Inst. Math. Jussieu, 13(2014), 599-632.


\bibitem{GHS1}C. Guillarmou, A. Hassell and A. Sikora, Resolvent at low
energy III: the spectral measure, Trans. Amer. Math. Soc.,
365(2013), 6103-6148.

\bibitem{GHS2} C. Guillarmou, A. Hassell and A. Sikora, Restriction and
spectral multiplier theorems on asymptotically conic manifolds,
Analysis and PDE, 6(2013), 893-950.

\bibitem{GV} J. Ginibre and G. Velo, Scattering theory in the energy space for a class of nonlinear
Schr\"odinger equations, J. Math. Pure Appl., 64(1985), 363-401.

\bibitem{Hor} L. H\"ormander,  The spectral function of an elliptic operator. Acta Math. 121 (1968), 193-218.

\bibitem{HL} A. Hassell and P. Lin, The Riesz transform for homogeneous Schr\"odinger operators on metric cones,
Rev. Mat. Iberoamericana, 30(2014), 477-522.


\bibitem{HTW1} A. Hassell, T. Tao and J. Wunsch, A Strichartz inequality for the
Schr\"odinger equation on non-trapping asymptotically conic
manifolds, Commun. in PDE, 30(2005), 157-205.

\bibitem{HTW} A. Hassell, T. Tao and J. Wunsch, Sharp Strichartz estimates on
non-trapping asymptotically conic manifolds, Amer. J. Math.,
128(2006), 963-1024.

\bibitem{HW} A. Hassell and J. Wunsch, The Schr\"odinger propagator for
scattering metrics, Annals of Mathematics, 162(2005), 487-523.

\bibitem{HZ} A. Hassell and J. Zhang, Global-in-time Strichartz estimates on nontrapping asymptotically conic manifolds,
Analysis and PDE, 9(2016), 151-192.

\bibitem{JSS} J. L. Journ\'e,  A. Soffer,  C. D. Sogge, Decay estimates for Schr\"odinger operators,  Comm. Pure Appl. Math. 44 (1991), 573-604.

\bibitem{Kato} T. Kato, Perturbation Theory for Linear Operators, Springer-Verlag, 1995.

\bibitem{KSWW} H. Kalf, U. W. Schmincke, J. Walter, and R. W\"ust, On the spectral theory of Schr\"odinger and Dirac operators with strongly singular potentials. In
\emph{Spectral theory and differential equations}. 182-226. Lect. Notes in Math. 448 (1975) Springer, Berlin.

\bibitem{KT} M. Keel and T. Tao, Endpoint Strichartz estimates,
Amer. J. Math., 120(1998), 955-980.

\bibitem{KMVZZ1} R. Killip, C. Miao, M. Visan, J. Zhang, and J. Zheng, Sobolev spaces adapted to the Schr\"{o}dinger operator with inverse-square potential, Math. Z. 288(2018), 1273-1298.

\bibitem{LL} J. M. L\'evy-Leblond, Electron capture by polar molecules, Phys. Rev. 153 (1967), 1-4.

\bibitem{L1} H. Li , La transformation de Riesz sur les vari\'et\'es coniques,  J. Funct. Anal. 168 (1999), 145-238.

\bibitem{MMT} J. Marzuola, J. Metcalfe, D. Tataru, Strichartz estimates and local smoothing estimates for asymptotically flat Schr\"odinger equations, Journal of Funct. Anal. 255 (2008) 1497-1553.

\bibitem{Melrose} R. B. Melrose, Spectral and scattering theory for the Laplacian on asymptotically Euclidian spaces.
In Spectral and scattering theory (Sanda, 1992), volume 161 of Lecture Notes in Pure and Appl. Math., pages 85-130. Dekker, New York, 1994.


\bibitem{MZZ} C. Miao, J. Zhang and J. Zheng, Maximal estimates for Schr\"odinger equation with inverse-square
potential, Pacific Journal of Mathematics 273(2015), 1-19.



\bibitem{Miz} H. Mizutani, Strichartz estimates for Schr\"odinger equations on
scattering manifolds, Commu. in PDE, 37(2012),
169-224.

\bibitem{Mo} E. Mooer, Heat kernel asymptotics on manifolds with conic singularities, J. Anal. Math., 78(1999), 1-36.

\bibitem{MS} D. M\"uller and A. Seeger, Regularity properties of wave propagation on conic manifolds and applications to spectral multipliers,
Adv. in Math. 161(2001), 41-130.


\bibitem{RS}  I. Rodnianski and W. Schlag,
Time decay for solutions of Schr\"odinger equations with rough and time-dependent potentials, Invent. Math., 155(2004), 451-513.



\bibitem{Russ} E. Russ,  Riesz transform on graphs for $1\leq p\leq
2$, Math. Scand., 87(2000), 133-160.

\bibitem{Str} R. Strichartz, Restrictions of Fourier transforms to quadratic surfaces and decay of solutions of wave
equations, Duke. Math. J., 44 (1977),  705-714.

\bibitem{ST} G. Staffilani, D. Tataru, Strichartz estimates for a Schr\"odinger
operator with nonsmooth coefficients, Commu. in PDE,
27(2002), 1337-1372.


\bibitem{Stein1} E.M. Stein, Harmonic
analysis: real variable methods, orthogonality and oscillatory
integrals, Princeton Mathematical Series, vol. 43, Princeton
University Press, Princeton, NJ, 1993.


\bibitem{Som} A. Sommerfeld, Mathematische Theorie der Diffraction, Math. Ann., 47(1896), 317-374.

\bibitem{Taylor} M. Taylor, Partial Differential Equations, Vol II,
Springer, 1996.

\bibitem{Taylor1} M. E. Taylor, Tools for PDE. Mathematical Surveys and Monographs, 81. American Mathematical Society, Providence, RI, 2000.

\bibitem{Tbook} T.Tao, Nonlinear Dispersive Equations, Local and Global Analysis,
CBMS Reg. Conf. Ser. Math., vol. 106, Amer. Math. Soc., Providence,
RI, ISBN: 0-8218-4143-2, 2006, published for the Conference Board of
the Mathematical Science, Washington, DC.


\bibitem{wang} X. Wang,   Asymptotic expansion in time of the Schr\"odinger group on conical manifolds, Ann. Inst. Fourier, 56(2006), 1903-1945.

\bibitem{Watson} G. N. Watson,  A Treatise on the Theory of Bessel Functions, Second Edition
Cambridge University Press, (1944).

\bibitem{Wunsch} J. Wunsch, A poisson relation for conic manifolds, Mathematical Research Letters 9(2002), 813–828.


\bibitem{zhang}  J. Zhang, Linear restriction estimates for Schr\"odinger equation on metric cones, Commu. in PDE, 40(2015), 995-1028.

 \bibitem{ZZJFA} J. Zhang and J. Zheng, Scattering theory for nonlinear Schr\"odinger equations with inverse-square potential, J. Funct. Anal. 267(2014), 2907-2932.

 \bibitem{ZZ1}  J. Zhang and J. Zheng, Global-in-time Strichartz estimates for Schr\"odinger on scattering manifolds, Commu. in PDE, 42(2017), 1962-1981.

  \bibitem{ZZ2}  J. Zhang and J. Zheng,  Strichartz estimates and wave equation in a conic singular space, Math. Ann. 376(2020), 525-581.












\end{thebibliography}
\end{document}